\documentclass[3p,sort,times]{elsarticle}
\usepackage{graphicx}
\usepackage{amssymb}
\usepackage{yfonts}
\usepackage{amsmath}
\usepackage[ruled,vlined]{algorithm2e}
\DontPrintSemicolon
\SetAlFnt{\scriptsize}
\usepackage{mathrsfs}
\usepackage{epstopdf}
\usepackage{stmaryrd}
\usepackage[normalem]{ulem}
\usepackage{subfig}
\usepackage{amsthm}
\usepackage{lmodern}
\usepackage{enumitem}
\usepackage{amsfonts}
\usepackage{mathtools}
\usepackage{booktabs}
\usepackage{xcolor}
\usepackage{stackrel}
\usepackage{comment}
\theoremstyle{plain}
\newtheorem{thm}{Theorem}
\newtheorem{Lem}{Lemma}
\newtheorem{cor}{Corollary}

\theoremstyle{remark}
\newtheorem{rem}{Remark}

\allowdisplaybreaks

\numberwithin{equation}{section}

\definecolor{Green}{HTML}{008000}

\newcommand{\pderivative}[2]{\frac{\partial #1}{\partial #2}}

\newcommand{\dmat}{\mat{D}}

\newcommand{\mmat}{\mat{{M}}}
\newcommand{\qmat}{\mat{{Q}}}

\newcommand{\average}[1]{\left\{\!\left\{ #1\right\}\!\right\}}

\newcommand{\half}{\frac{1}{2}}
\newcommand{\mat}[1]{\mathbf{#1}}
\renewcommand{\vec}[1]{\underline{#1}}

\newcommand{\Bmat}{\mat{{B}}}

\newcommand{\op}{\mathcal{L}}
\newcommand{\opx}{\mathcal{L}_x}
\newcommand{\opy}{\mathcal{L}_y}
\newcommand{\opz}{\mathcal{L}_z}

\newcommand{\textred}[1]{\textcolor{black}{#1}}
\newcommand{\textblue}[1]{\textcolor{black}{#1}}

\begin{document}

\begin{frontmatter}
\title{Split Form Nodal Discontinuous Galerkin Schemes with Summation-By-Parts Property for the Compressible Euler Equations}
\author[unikoelnMath]{Gregor J.~Gassner\corref{cor1}}
\cortext[cor1]{Corresponding author.}
\ead{ggassner@math.uni-koeln.de}
\author[unikoelnMath]{Andrew R.~Winters}
\author[fsu]{David Kopriva}

\address[unikoelnMath]{Mathematisches Institut, Universit\"{a}t zu K\"{o}ln, Weyertal 86-90, 50931 K\"{o}ln, Germany}
\address[fsu]{Department of Mathematics, The Florida State University, Tallahassee FL 32312, USA}

\begin{abstract}
Fisher and Carpenter (\textit{High-order entropy stable finite difference schemes for non-linear conservation laws: Finite domains, Journal of Computational Physics, 252:518--557, 2013}) found a remarkable equivalence of general diagonal norm high-order summation-by-parts operators to a subcell based high-order finite volume formulation. This equivalence enables the construction of provably entropy stable schemes \textblue{by a specific choice of the subcell finite volume flux. We show that besides the construction of entropy stable high order schemes, a careful choice of subcell finite volume fluxes generates split formulations of quadratic or cubic terms. Thus, by changing the subcell finite volume flux to a specific choice, we are able to generate, in a systematic way, all common split forms of the compressible Euler advection terms, such as the Ducros splitting and the Kennedy and Gruber splitting. Although these split forms are not entropy stable, we present a systematic way to prove which of those split forms are at least kinetic energy preserving. With this, we show we construct a unified high-order split form DG framework. We investigate with three dimensional numerical simulations of the inviscid Taylor-Green vortex and show that the new split forms enhance the robustness of high order simulations in comparison to the standard scheme when solving turbulent vortex dominated flows. In fact, we show that for certain test cases, the novel split form discontinuous Galerkin schemes are more robust than the discontinuous Galerkin scheme with over-integration.}
\end{abstract}

\end{frontmatter}

\section{Introduction}\label{sec:intro}

This paper \textblue{presents robust} nodal discontinuous Galerkin (DG) approximations for the advective terms of the three dimensional compressible Navier-Stokes equations, namely the compressible Euler equations. In the discontinuous Galerkin community, stabilising an approximation is frequently done by \textblue{polynomial de-aliasing through over-integration \cite{Kirby2003,tcfd2012,mengaldo2015}}. The typical DG implementation is under-integrated. It is well-known that when the numerical quadrature in the DG scheme is constructed so that flux functions that depend linearly on the solution (e.g. linear advection equation) are integrated exactly, the approximation retains the formal order of accuracy \cite{CS1,CS2,CS4,CS5,CS6}. The exact number of quadrature points depends on the polynomial ansatz space, on the element type and, of course, on the specific quadrature rule used. A similar concept is used for nodal DG schemes where the ansatz uses interpolation and (multi-)variate Lagrange-type basis functions. 
In many cases, DG discretisations for linear fluxes are directly applied to problems with non-linear flux functions by simply exchanging the linear flux function with the corresponding non-linear flux. The reasoning is clear, as the minimum number of quadrature (or interpolation) nodes necessary to obtain the expected order of convergence gives an implementation with the lowest number of arithmetic operations, and thus increased efficiency at first sight. 

\subsection{Stabilisation strategies for discontinuous Galerkin based discretisations}\label{sec:dg_dealiasing}

It turns out that the strategy of minimal effort has a drastic impact when the numerical solution is under-resolved. In such cases, e.g. under-resolved turbulence or shocks, aliasing errors due to variational errors corrupt the approximate solution and may even drive a non-linear instability. Whereas such a non-linear instability is often masked by excessive artificial viscosity when using low order approximations, high-order discretisations with their lower inherent numerical dissipation are prone to such instabilities. In fact, without additional counter-measures such high-order discretisations are unstable and crash \cite{Kirby2003}. It is of course arguable why one should even \textblue{consider} under-resolution at all, as the accuracy of results with low resolution is at least questionable. However, recent investigations show that it is possible to achieve quite accurate results with (very) high-order DG schemes, even with under-resolution as long as proper de-aliasing mechanisms are augmented \cite{tcfd2012}.

A very popular strategy is the use of polynomial de-aliasing, as mentioned above. Motivated by spectral methods, the authors of \cite{Kirby2003} proposed to de-alias by increasing the number of quadrature nodes according to the non-linearity of the flux function so that the variational terms are evaluated exactly. As a prime example, the non-linear Burger's equation has a quadratic flux function and hence roughly a factor of $1.5$ times the number of interpolation nodes is needed in each spatial direction to integrate the non-linear flux exactly. It is also remarkable that for DG discretisations with exact evaluations of the integral, Shu and Jiang \cite{cell_entropy_dg} proved a cell entropy inequality and, with this, non-linear $L_2$ stability. It is very important to note however, that Shu and Jiang consider \textit{scalar} non-linear conservation laws and consequently non-linear stability of DG discretisations with exact integration is only valid for \textit{scalar} conservation laws. Their stability estimate does not carry over to systems of non-linear conservation laws such as the compressible Euler equations, independent of the number of quadrature nodes used to evaluate the inner products. 

In fact, recent results by \cite{rodrigo_iLES} show that even with up to four times the number of quadrature nodes in each spatial direction, DG discretisations with either local Lax-Friedrichs or Roe's flux function crash for under-resolved turbulence simulations. The ``up to four times the quadrature nodes" statements necessitates an additional remark: For the compressible Euler equations, the polynomial ansatz is typically done for the conserved quantities such as mass, momentum and energy. But, the flux functions are rational polynomials of the conserved quantities. This is problematic because the available quadrature rules are only exact for polynomial integrands. Thus, it is formally impossible to implement an exact integration based on standard quadrature rules for the compressible Euler equations. 

To formally prove non-linear stability of DG discretisations for systems of conservation laws, it is necessary to reformulate the equations in terms of entropy variables. Here, the polynomial ansatz (and the test-function) space approximates the entropy variables and not the conserved variables. This is important, as it is now formally possible to show that the DG discretisation satisfies a cell entropy inequality for systems \cite{HM14_764}. However, it is important to note, that again, the stability proof relies on exact evaluation of the variational terms, similar to the proof of Shu and Jiang \cite{cell_entropy_dg}. As discussed above, in the case of the compressible Euler equations it is impossible (or impractical) to find a fixed number of quadrature nodes so that the variational errors due to the rational non-linearity of the flux functions with respect to the entropy variables are zero. At least an adaptive numerical quadrature approach would be necessary, which of course makes this strategy quite cumbersome with respect to implementation and efficiency. 

Up to now, the only known stability proof without relying on exact evaluation of inner products for the compressible Euler equations is presented in Carpenter et al. \cite{carpenter_esdg}. The proof is possible because it uses a very specific form of the DG discretisation. \textblue{We note that in addition to the correct treatment of the volume integral terms, the numerical fluxes at the interfaces as well as the boundary conditions need to be treated carefully \cite{parsani2015entropy,parsaniCITEKEY,svard2014entropy} to obtain entropy stability.}

By choosing a nodal DG ansatz with Gauss-Lobatto (GL) nodes used for both interpolation and numerical integration, the so-called discontinuous Galerkin spectral element method with collocation (DGSEM) results, e.g. \cite{koprivabook}. In the present work, we always use interpolation and integration based on the GL nodes. This variant of the DG methodology is special because it possesses the summation-by-parts (SBP) property. The discrete mass matrix $\mmat$ and the discrete derivative matrix $\dmat$ of the DGSEM satisfy all formal definitions of a SBP operator \cite{Strand199447}
\begin{equation}
\label{eq:sbp}
\qmat:=\mmat\,\dmat\quad\text{ with }\quad\qmat+\qmat^T=\Bmat,
\end{equation}
where $\Bmat=\textrm{diag}([-1,0,\ldots,0,1])$ is the boundary evaluation operator \textblue{and the rows of $\qmat$ are undivided differences}, e.g. \cite{gassner_skew_burgers,carpenter_esdg}. Furthermore, the mass matrix is diagonal and is used to define a discrete $L_2$-norm, e.g. to compute errors. To be more specific, the DGSEM operators form a so-called diagonal norm SBP operator. It was possible in \cite{carpenter_esdg,fisher2013} to construct a specific form of the DGSEM that satisfies a cell entropy type inequality for all conservation laws while retaining the nodal nature of the discretisation, i.e. without relying on exact evaluation of the inner products. The only necessary ingredient to achieve a cell entropy type inequality is a two-point numerical flux function that gives exact entropy conservation in a standard finite volume type discretisation \cite{carpenter_esdg}. A remarkable advantage of these derivations are that they do not rely on any DG specifics, but only on the SBP property. Thus, the non-linear stability results directly carry over to all discretisations constructed with SBP operators, where the mass matrix $\mmat$ (often called norm matrix) is diagonal as in the case of DGSEM. 
 
\subsection{Alternative stabilisation strategies for the compressible Euler equations}\label{sec:fd_dealiasing}

We leave the DG community momentarily and investigate de-aliasing strategies used in other high-order communities such as finite differences, e.g. \cite{ducros2000,Morinishi2010276,kennedy2008,FDaliasing,larsson2007} and spectral methods, e.g. \cite{blaisdell1996effect,Zang199127}. A very prominent example for de-aliasing is the use of alternative formulations of the non-linear advection terms, e.g. so-called split formulations. Due to the non-linear character of the advective terms of the Euler equations, there are many ways to re-write the equations.  A good overview of the different split forms can be found in \cite{pirozzoli2011}. 

We consider here split forms of the three dimensional compressible Euler equations. In a general operator notation, we can write the Euler equations as 
\begin{equation}
\label{eq:operator_form}
U_t + \opx(U)+ \opy(U)+ \opz(U) = 0,
\end{equation}
with the conservative variables
\begin{equation}
U=\begin{pmatrix}
   u_1\\
   u_2\\
   u_3\\
   u_4\\
   u_5\\
  \end{pmatrix}:=\begin{pmatrix}
   \rho\\
   \rho\,u\\
   \rho\,v\\
   \rho\,w\\
   \rho\,e\\
  \end{pmatrix},
\end{equation}
where $\rho\,e=\rho\,\theta+\frac{1}{2}\rho\,(u^2+v^2+w^2)$ and  $\rho$, $u,v,w$, $p$, $\theta$, $e$ are density, velocity, pressure, specific inner energy and specific total energy respectively, and $\op_{x,y,z}(U)$ are the non-linear spatial differential operators in the respective Cartesian direction $x,y,z$ acting on $U$. We close the system by considering a perfect gas equation which relates the internal energy and pressure as $p = (\gamma -1)\,\rho\theta$, where $\gamma$ denotes the adiabatic coefficient.

If we consider, for instance, the standard divergence form of the Euler equations, the non-linear operators are
\begin{equation}
\label{eq:op_divergence}
\resizebox{0.925\textwidth}{!}{$
\begin{aligned}
\opx^{div}(U):= F(U)_x=  \begin{bmatrix}
\rho\,u\\
\rho\,u^2 +p\\
\rho\,u\,v\\
\rho\,u\,w\\
(\rho\,e+p)\,u
\end{bmatrix}_x,\quad
\opy^{div}(U):= G(U)_y= \begin{bmatrix}
\rho\,v\\
\rho\,u\,v \\
\rho\,v^2+p\\
\rho\,v\,w\\
(\rho\,e+p)\,v
\end{bmatrix}_y,\quad
\opz^{div}(U):=  H(U)_z=\begin{bmatrix}
\rho\,w\\
\rho\,u\,w \\
\rho\,v\,w\\
\rho\,w^2+p\\
(\rho\,e+p)\,w
\end{bmatrix}_z.
\end{aligned}$}
\end{equation}

The divergence form \eqref{eq:op_divergence} is typically used to construct a DG discretisation, as it directly gives discrete conservation of the resulting approximation. However, depending on how one interprets the non-linearities of the Euler flux (quadratic, cubic or rational) there are several ways to re-write the equations in an equivalent split form. 

We use different split form approximations of derivatives of products to determine a particular splitting. The derivative of a product of two quantities is approximated by
\begin{equation}\label{eq:two_split}
\begin{split}
(a\,b)_x &= \frac{1}{2}\,(a\,b)_x + \frac{1}{2}\left(a_x\,b + a\,b_x\right),
\end{split}
\end{equation}
which originates from the general split form 
\begin{equation}
\label{eq:splitformquadratic}
\begin{split}
(a\,b)_x &= \alpha\,(a\,b)_x + (1-\alpha)\,\left(a_x\,b + a\,b_x\right),\quad \alpha\in\mathbb{R},
\end{split}
\end{equation}
when choosing $\alpha=1/2$. We note that with the quadratic splitting it is possible to prove stability of linear variable coefficient problems \cite{kopriva2014}.

For the derivative of cubic terms Kennedy and Gruber \cite{kennedy2008} proposed the following split form 
\begin{equation}
\resizebox{\textwidth}{!}{$
\begin{aligned}
\pderivative{}{x}(abc) = \alpha\pderivative{}{x}(abc) &+ \beta\left[a\pderivative{}{x}(bc) + bc\pderivative{}{x}(a)\right]+ \kappa\left[b\pderivative{}{x}(ac) + ac\pderivative{}{x}(b)\right]+ \delta\left[c\pderivative{}{x}(ab) + ab\pderivative{}{x}(c)\right] \\&+ \epsilon\left[bc\pderivative{}{x}(a) + ac\pderivative{}{x}(b) + ab\pderivative{}{x}(c)\right],
\end{aligned}$}
\end{equation}
where $\epsilon = 1-\alpha-\beta-\kappa-\delta$ and $\alpha,\beta,\kappa,\delta\in\mathbb{R}$. From all the possible combinations, we choose the case $\alpha=\beta=\kappa=\delta=\frac{1}{4}$ and $\epsilon = 0$, which gives
\begin{equation}\label{eq:three_split}
\resizebox{\textwidth}{!}{$
\begin{aligned}
\begin{split}
(a\,b\,c)_x &= \frac{1}{4}\,(a\,b\,c)_x + \frac{1}{4}\left(a_x\,(b\,c) + b_x\,(a\,c)+c_x\,(a\,b)\right)+ \frac{1}{4}\left(a\,(b\,c)_x + b\,(a\,c)_x+c\,(a\,b)_x\right).
\end{split}
\end{aligned}$}
\end{equation}

With the two split forms \eqref{eq:two_split} and \eqref{eq:three_split} it is now possible to compose new equivalent forms of the non-linear operator $\op(U)$ to enhance the stability of discretisations. 
For instance, Morinishi \cite{Morinishi2010276} introduced a skew-symmetric form for the momentum equations, which was rewritten in Gassner \cite{gassner_kepdg} assuming time continuity as  
\begin{equation}\label{eq:op_morinishi}
\resizebox{0.9\hsize}{!}{$
\opx^{MO}(U) := \begin{bmatrix}
(\rho\,u)_x\\[0.1cm]
\half\left((\rho u^2)_x+\rho u\,(u)_x + u\,(\rho u)_x\right) + p_x\\[0.1cm]
\half\left((\rho u v)_x+\rho u\,(v)_x + v\,(\rho u)_x\right)\\[0.1cm]
\half\left((\rho u w)_x+\rho u\,(w)_x + w\,(\rho u)_x\right)\\[0.1cm]
\left((\rho\theta+p)\,u\right)_x + \half\left(\rho\,u^2\,(u)_x + u\,(\rho\,u^2)_x + \rho\,u\,v\,(v)_x + v\,(\rho\,u\,v)_x + \rho\,u\,w\,(w)_x + w\,(\rho\,u\,w)_x \right)
\end{bmatrix},$}
\end{equation} 
with the operators in the $y$ and $z$ direction defined similarly. It is important to note that the form of Morinishi does not use the split form of the quadratic terms ($\alpha=1/2$), but the pure advective form ($\alpha=0$ in \eqref{eq:splitformquadratic}) in the energy equation. The form \eqref{eq:op_morinishi} allows the construction of a formally kinetic energy preserving discontinuous Galerkin method \cite{gassner_kepdg}, but we will show in the numerical results section that the positive stabilisation effect of this alternative form of Morinishi is the lowest in comparison to the other forms presented below. A possible reason could be that only the advective form is used in the total energy equation. 

In contrast to Morinishi's flux, Ducros et al. \cite{ducros2000} proposed the following form, where only the split form ($\alpha=1/2$) of the quadratic product  \eqref{eq:two_split} is used
\begin{equation}\label{eq:op_ducros}
\opx^{DU}(U) := \begin{bmatrix}
\half\left((\rho u)_x+\rho(u)_x + u(\rho)_x\right)\\[0.1cm]
\half\left((\rho u^2)_x+\rho u(u)_x + u(\rho u)_x\right) + p_x\\[0.1cm]
\half\left((\rho u v)_x+\rho v(u)_x + u(\rho v)_x\right)\\[0.1cm]
\half\left((\rho u w)_x+\rho w(u)_x + u(\rho w)_x\right)\\[0.1cm]
\half\left(((\rho e +p)u)_x+(\rho e + p)(u)_x + u(\rho e + p)_x\right)
\end{bmatrix}.
\end{equation} 
However, this alternative formulation does not lead to a discretisation that is formally kinetic energy preserving\textblue{, e.g. \cite{pirozzoli2011}}. 

The cubic form \eqref{eq:three_split} was first introduced and applied in \cite{kennedy2008}. The operator proposed by Kennedy and Gruber in the $x-$direction reads
\begin{equation}
\label{eq:op_KG}
\opx^{KG}(U) := \begin{bmatrix}
\half\left((\rho u)_x+\rho(u)_x + u(\rho)_x\right)\\[0.1cm]
\frac{1}{4}\left[(\rho u^2)_x+\rho (u^2)_x + 2u(\rho u)_x + u^2(\rho)_x + 2\rho u(u)_x \right]+ p_x\\[0.1cm]
\frac{1}{4}\left[(\rho u v)_x+\rho (uv)_x + u(\rho v)_x + v(\rho u)_x + uv(\rho)_x + \rho v(u)_x + \rho u(v_x)\right]\\[0.1cm]
\frac{1}{4}\left[(\rho u w)_x+\rho (uw)_x + u(\rho w)_x + w(\rho u)_x + uw(\rho)_x + \rho w(u)_x + \rho u(w_x)\right]\\[0.1cm]
\half\left((pu)_x+p(u)_x+u(p_x)\right)+\frac{1}{4}\left[(\rho e u)_x + \rho(e u)_x + e(\rho u)_x + u(\rho e)_x\right. \\
\left.\qquad\qquad\qquad\qquad\qquad\qquad+ eu(\rho)_x + \rho u (e)_x + \rho e (u)_x\right]
\end{bmatrix}.
\end{equation}
The standard quadratic form is used for the continuity equation, whereas the cubic form is used for the advection terms in the momentum equation. In the energy equation, both the quadratic form for $p\,u$ and the cubic form for $\rho\,e\,u$ are applied. Again, this form allows one to construct formally kinetic energy preserving discretisations \cite{jameson2008}.

In his overview, Pirozzoli \cite{pirozzoli2011} re-intrepreted the work of Kennedy and Gruber and used a similar but slightly different form of the flux splitting. Whereas the continuity and momentum equations don't change, Pirozzoli re-wrote the energy term $(\rho\,e+p)u$ as $\rho\,h\,p$, where the specific enthalpy is given by $h=e+p/\rho$. With this, only the cubic form \eqref{eq:three_split} can be used for the energy equation, which results in
\begin{equation}\label{eq:op_pirozzoli}
\resizebox{0.9\textwidth}{!}{$
\opx^{PI}(U) := \begin{bmatrix}
\half\left((\rho u)_x+\rho(u)_x + u(\rho)_x\right)\\[0.1cm]
\frac{1}{4}\left[(\rho u^2)_x+\rho (u^2)_x + 2u(\rho u)_x + u^2(\rho)_x + 2\rho u(u)_x \right]+ p_x\\[0.1cm]
\frac{1}{4}\left[(\rho u v)_x+\rho (uv)_x + u(\rho v)_x + v(\rho u)_x + uv(\rho)_x + \rho v(u)_x + \rho u(v_x)\right]\\[0.1cm]
\frac{1}{4}\left[(\rho u w)_x+\rho (uw)_x + u(\rho w)_x + w(\rho u)_x + uw(\rho)_x + \rho w(u)_x + \rho u(w_x)\right]\\[0.1cm]
\frac{1}{4}\left[(\rho u h)_x + \rho(u h)_x + h(\rho u)_x + u(\rho h)_x+ u h(\rho)_x + \rho u(h)_x + \rho h (u)_x\right]
\end{bmatrix}.$}
\end{equation}

There are other split forms available in literature, e.g. given by Kravchenko and Moin \cite{FDaliasing}. However, we will restrict this discussion to the forms presented above. The goal of the paper is to show how to discretise such forms for discontinuous Galerkin methods. More specifically, we focus on DGSEM (with GL nodes), as this specific variant satisfies the diagonal norm SBP \eqref{eq:sbp} which is key to achieve a conservative approximation for the split forms, e.g. \cite{kopriva2014}. 

The remainder of the paper is organised as follows: In the next section the nodal collocation spectral element framework is introduced. The main Sec. \ref{sec:DG-Disc2} introduces the volume flux difference form and introduces specific numerical volume fluxes that generate known split forms of the compressible Euler equations. In  Sec. \ref{sec:numerical results} numerical experiments that support the theoretical findings are presented with our conclusions drawn in the last section.  

\section{The Nodal Discontinuous Galerkin Spectral Element Method}\label{sec:DG-Disc}

In this section, we provide the basic construction of a nodal discontinuous Galerkin spectral element method (DGSEM) on tensor product elements. The implementation of the DGSEM for the compressible Euler equations is based on \cite{Hindenlang2012}, which includes a detailed description of the standard weak form discretisation of the compressible Euler equations in divergence form. It is important to note \textblue{that} we consider the strong form of the DGSEM in this work due to its close relation to the so-called simultaneous-approximation term (SAT) SBP finite difference methods \cite{gassner_skew_burgers}. In \cite{KoprivaGassner_GaussLob} it was shown that the weak and strong forms are algebraically equivalent. This equivalence of strong form and weak form is itself equivalent to the SBP property of an operator \cite{gassner_skew_burgers}. We specifically focus on the volume discretisation of the Euler terms. All other parts of the implementation remain unchanged. Thus, it is straightforward to extend an existing DGSEM code to the split form approximations presented in this paper. We restrict the following discussion and the numerical results section to Cartesian meshes as our focus is on the effect of the non-linear terms of the compressible Euler equations. For completeness, we provide the algebraic extensions to curvilinear meshes in \ref{sec:curvilinear}.

The first step in the discretisation is to subdivide the computational domain into non-overlapping hexahedral elements $C$. Each element is transformed to the reference element $C_0=[-1,1]^3$ with a polynomial mapping. The degree of the mapping is chosen to be at most the degree of the element-wise polynomial approximation $N$ to ensure free-stream preservation \cite{Kopriva:2006er}. For non-curved hexahedral elements, the mapping we use is trilinear and reads as 
\begin{equation}\label{eq:StandardHexMap}
\begin{aligned}
\vec{X}(\vec{\xi}) =(X(\vec{\xi}),Y(\vec{\xi}),Z(\vec{\xi}))^T=  &\frac{1}{8}\left\{\vec{x}_1(1-\xi)(1-\eta)(1-\zeta)+\vec{x}_2(1+\xi)(1-\eta)(1-\zeta)\right.\\
            & + \vec{x}_3(1+\xi)(1+\eta)(1-\zeta)+\vec{x}_4(1-\xi)(1+\eta)(1-\zeta) \\
            & + \vec{x}_5(1-\xi)(1-\eta)(1+\zeta)+\vec{x}_6(1+\xi)(1-\eta)(1+\zeta) \\
            & + \left.\vec{x}_7(1+\xi)(1+\eta)(1+\zeta)+\vec{x}_8(1-\xi)(1+\eta)(1+\zeta)\right\},
\end{aligned}
\end{equation}
where $\vec{x}_i$, $i=1,\ldots,8$ are the physical coordinates of the corners of the hexahedral element and $\vec{\xi}=(\xi,\eta,\zeta)^T$ are the reference coordinates. As we restrict ourselves here to Cartesian meshes, the Jacobian and metric terms simplify to 
\begin{equation}
\label{eq:cartesian_metric}
J = \frac{1}{8}\Delta x\Delta y\Delta z,\quad X_{\xi} = \half\Delta x,\quad Y_{\eta} = \half\Delta y,\quad Z_{\zeta} = \half\Delta z,
\end{equation}
with element side lengths $\Delta x$, $\Delta y$, and $\Delta z$.

For each element, each component of the conservative variable vector is approximated by a polynomial of degree $N$ in reference space, e.g. for the density 
\begin{equation}
\label{eq:polynom}
u_1(x,y,z,t)\big|_{C} = \rho(x,y,z,t)\big|_{C}\approx \rho(\xi,\eta,\zeta,t) := \sum\limits_{i,j,k=0}^N \rho^{i,j,k}(t)\,\ell_i(\xi)\,\ell_j(\eta)\ell_k(\zeta),
\end{equation}
where $\{u_1^{i,j,k}(t)\}_{i,j,k=0}^{N}$ are the time dependent nodal degrees of freedom of the element $C$ at the nodes $(i,j,k)$. Each nodal interpolant is defined with $(N+1)^3$ GL nodes $\{\xi_i\}_{i=0}^N$, $\{\eta_j\}_{j=0}^N$, and $\{\zeta_j\}_{j=0}^N$ in the reference cube $C_0$. The associated Lagrange basis functions are given by
\begin{equation}
\label{eq:lagrange_basis}
 \ell_j(\xi)=\prod\limits_{i=0,i\neq j}^N\frac{\xi - \xi_i}{\xi_j-\xi_i},\qquad j=0,...,N,
\end{equation}
and satisfy the cardinal property
\begin{equation}
\label{cardinal}
\ell_j(\xi_i)=\delta_{ij},
\end{equation}
where $\delta_{ij}$ denotes Kronecker's symbol with $\delta_{ij}=1$ for $i=j$ and $\delta_{ij}=0$ for $i\neq j$. Integrating the polynomial Lagrange basis functions over the unit interval $[-1,1]$ gives the GL quadrature weights $\{\omega_j\}_{j=0}^N$. The GL nodes and weights form a quadrature rule with integration precision $2N-1$ that is used in the nodal DGSEM for the approximation of the weak form \textblue{integrals}. The Lagrange basis functions, $\{\ell_j\}_{j=0}^N$, are discretely orthogonal to each other, and the mass matrix is diagonal
\begin{equation}
\mmat:=\textrm{diag}([\omega_0,...,\omega_N]).
\end{equation}
In addition to the discrete integration, the polynomial basis functions and the interpolation nodes form a discrete differentiation. We introduce the polynomial derivative matrix
\begin{equation}
\label{eq:DmatDef}
D_{ij}=\frac{\partial\ell_j}{\partial\xi}\Bigg|_{\xi=\xi_i},\quad i,j=0,\ldots,N.
\end{equation}
As mentioned in the introduction, the derivative operator \eqref{eq:DmatDef} is special, as it satisfies the SBP-property \eqref{eq:sbp} for all polynomial degrees $N$, e.g. \cite{gassner_skew_burgers,carpenter_esdg,gassner_kepdg}. We define the matrix $\qmat:=\mmat\dmat$, which has the SBP-property $\qmat + \qmat^T=\mat{B}:=\textrm{diag}(-1,0,\ldots,0,1)$. The SBP-property is used to mimic integration-by-parts, by manipulating the derivative matrix to become
\begin{equation}
\dmat=\mmat^{-1}\qmat = \mmat^{-1}\mat{B}-\mmat^{-1}\qmat^T.
\end{equation}

With all these ingredients, we are able to formulate the standard DGSEM approximation of the divergence form of the flux equations \eqref{eq:operator_form}. We use the elemental mapping and the metric terms \eqref{eq:cartesian_metric} to transform the system \eqref{eq:operator_form} into the reference element
\begin{equation}
\label{eq:_transformed_operator_form}
J\,U_t + \widetilde{\op}_\xi(U) + \widetilde{\op}_\eta(U) + \widetilde{\op}_\zeta(U) = 0,
\end{equation}
with
\begin{equation}
\widetilde{\op}_\xi(U)= Y_\eta Z_\zeta \op_\xi(U),\,\, \widetilde{\op}_\eta(U)= X_\xi Z_\zeta \op_\eta(U),\,\, \widetilde{\op}_\zeta(U)=X_\xi Y_\eta \op_\zeta(U),
\end{equation}
where we used the restriction to Cartesian meshes to simplify the expressions. 

We collect the three dimensional DGSEM approximation of the divergence form of the equations written in indicial notation. Consider a component $l$ of the system \eqref{eq:_transformed_operator_form} and a GL node $(i,j,k)$. The DGSEM approximation in strong form is
\begin{equation}\label{eq:RHSStrong}
\begin{aligned}
(\op_\xi(U))_{ijk}^l&\approx \left[F^{*,l}(1,\eta_j,\zeta_k;\vec{n}) - F^l_{Njk}\right] - \left[F^{*,l}(-1,\eta_j,\zeta_k;\vec{n}) - F^l_{0jk}\right]+\sum_{m=0}^N D_{im}F^l_{mjk},\\
(\op_\eta(U))_{ijk}^l&\approx \left[G^{*,l}(\xi_i,1,\zeta_k;\vec{n}) - G^l_{iNk}\right] - \left[G^{*,l}(\xi_i,-1,\zeta_k;\vec{n}) - G^l_{i0k}\right]+\sum_{m=0}^N D_{jm}G^l_{imk},\\
(\op_\zeta(U))_{ijk}^l&\approx \left[H^{*,l}(\xi_i,\eta_j,1;\vec{n}) - H^l_{ijN}\right] - \left[H^{*,l}(\xi_i,\eta_j,-1;\vec{n}) - H^l_{ij0}\right]+\sum_{m=0}^N D_{km}H^l_{ijm},
\end{aligned}
\end{equation}
where $l=1,\ldots,5$. We use collocation for the non-linear flux functions, e.g. $F^l_{ijk}:=f^l(U_{ijk})$, and denote the outward pointing normal vector by $\vec{n}$. Due to the discontinuous nature of the approximation space, it is necessary to introduce numerical surface flux functions, which resolve the jumps at the interface. These specific surface flux functions depend on the left and right state values at an interface and are marked by a $*$. We will specify the numerical surface fluxes in Sec. \ref{sec:numflux}.\\ 
\textblue{\begin{rem}
We note, that in the case of over-integration the fluxes $F^l_{mjk}$, $G^l_{imk}$, $H^l_{ijm}$ in sum in \eqref{eq:RHSStrong} need to be changed. Instead of using a collocation approach where the nodal coefficients of the fluxes are just evaluated with the nodal values of the discrete solution $U$, it is necessary to compute a proper $L_2$-projection of the non-linear fluxes onto the polynomial space of degree $N$. This can be, for instance, implemented with an intermediate quadrature grid with a high number of nodes (such that the integrals in the $L_2$-projection are evaluated exactly), see e.g.  \cite{Kirby2003,tcfd2012,mengaldo2015}.\\ 
\end{rem}}

The resulting semi-discrete form is a coupled system of ordinary differential equations in time, which we integrate with a low storage 5-stage 4th order accurate explicit Runge-Kutta method, e.g. \cite{Kennedy1994}.

\section{Split form stabilisation for  DGSEM}\label{sec:DG-Disc2}

In this section we demonstrate how to implement the different split forms. To do so, we begin with a standard strong form DGSEM implementation \eqref{eq:RHSStrong}, and modify the discrete volume integrals that lead to $\sum_{m=0}^N D_{im}F^l_{mjk}$, $\sum_{m=0}^N D_{jm}G^l_{imk}$ and $\sum_{m=0}^N D_{km}H^l_{ijm}$ to include the different split forms presented in the introduction, Sec. \ref{sec:fd_dealiasing}.

\subsection{DGSEM with numerical volume flux function}\label{sec:DGSEM_fluxform}

An important result that we use in this work is presented in Fisher et al. \cite{fisher2013} and Carpenter et al. \cite{carpenter_esdg}. These authors showed that it is possible to rewrite the application of a differencing operator $\dmat$ with the diagonal SBP property into an equivalent subcell based finite volume type differencing formulation 
\begin{equation}
\label{eq:flux-differencing_vol}
\begin{split}
&\sum_{m=0}^N D_{im}F^l_{mjk} = \frac{\bar{F}^l_{(i+1)jk} - \bar{F}^l_{(i)jk}}{\omega_i},\quad i,j,k=0,...,N,\\
&\sum_{m=0}^N D_{jm}G^l_{imk}= \frac{\bar{G}^l_{i(j+1)k} - \bar{G}^l_{i(j)k}}{\omega_j},\quad i,j,k=0,...,N,\\
&\sum_{m=0}^N D_{km}H^l_{ijm}= \frac{\bar{H}^l_{ij(k+1)} - \bar{H}^l_{ij(k)}}{\omega_k},\quad i,j,k=0,...,N,
\end{split}
\end{equation}
with consistent auxiliary fluxes $\{\bar{F}^l_{(ii)jk}\}_{ii,j,k=0}^{(N+1),N,N}$, $\{\bar{G}^l_{i(jj)k}\}_{i,jj,k=0}^{N,(N+1),N}$ and $\{\bar{H}^l_{ij(kk)}\}_{i,j,kk=0}^{N,N,(N+1)}$. This result is important, as it directly gives conservation of diagonal norm SBP discretisations in the sense of Lax-Wendroff due to the telescoping flux differencing \cite{fisher2013}. In addition to the desired conservation property, the flux differencing interpretation \eqref{eq:flux-differencing_vol} enables the construction of entropy conserving discretisations without relying on exact integration. 

Fisher et al. \cite{fisher2013} and Carpenter et al. \cite{carpenter_esdg} used a diagonal norm SBP operator, the volume flux differencing relation \eqref{eq:flux-differencing_vol}, and the existence of a two-point entropy conserving flux function $F_{EC}^{\#} = F_{EC}^{\#}(U_{ijk},U_{mjk})$ to obtain a high-order accurate discretisation 
\begin{equation}
\label{eq:highorder_flux}
\frac{\bar{F}^l_{(i+1)jk} - \bar{F}^l_{(i)jk}}{\omega_i} \approx 2\,\sum\limits_{m=0}^N D_{im}\,F^{\#,l}_{EC}(U_{ijk},U_{mjk}).
\end{equation}
The important theorems and proofs can be found e.g. in Fisher et al. \cite{fisher2013} (Thms. 3.1 and 3.2 on page 15 and 18). By choosing the two-point entropy flux $F_{EC}^{\#}$ also as the numerical surface flux in a DGSEM discretisation, with the volume terms computed with \eqref{eq:highorder_flux}, it was shown that the resulting DG discretisation conserves the (discrete) integral of the entropy \cite{carpenter_esdg}. Note that the goal of the work of Fisher and Carpenter et al. \cite{carpenter_esdg,fisher2013} is not to construct an entropy conserving discretisation, but a discretisation that is entropy stable. However, from a scheme that exactly preserves entropy, it is possible to introduce dissipation terms, e.g. at the element interfaces so that the (discrete) integral of entropy is guaranteed to decrease (typically named \textit{entropy stability})\textblue{, e.g. \cite{parsaniCITEKEY,gassner_skew_burgers}}. Thus, entropy conservation can be seen as an intermediate step to entropy stability. 

In the present work, we rewrite the volume terms using the flux differencing form \eqref{eq:highorder_flux}. However, we keep the expression general and use the yet to be specified numerical volume fluxes $F^{\#}$, $G^{\#}$, and $H^{\#}$
\begin{equation}\label{eq:RHSfluxform}
\begin{aligned}
(\op_\xi(U))_{ijk}^l&\approx \left[F^{*,l}(1,\eta_j,\zeta_k;\vec{n}) - F^l_{Njk}\right] - \left[F^{*,l}(-1,\eta_j,\zeta_k;\vec{n}) - F^l_{0jk}\right]+2\,\sum\limits_{m=0}^N D_{im}\,F^{\#,l}(U_{ijk},U_{mjk}),\\
(\op_\eta(U))_{ijk}^l&\approx \left[G^{*,l}(\xi_i,1,\zeta_k;\vec{n}) - G^l_{iNk}\right] - \left[G^{*,l}(\xi_i,-1,\zeta_k;\vec{n}) - G^l_{i0k}\right]+2\,\sum\limits_{m=0}^N D_{jm}\,G^{\#,l}(U_{ijk},U_{imk}),\\
(\op_\zeta(U))_{ijk}^l&\approx \left[H^{*,l}(\xi_i,\eta_j,1;\vec{n}) - H^l_{ijN}\right] - \left[H^{*,l}(\xi_i,\eta_j,-1;\vec{n}) - H^l_{ij0}\right]+2\,\sum\limits_{m=0}^N D_{km}\,H^{\#,l}(U_{ijk},U_{ijm}).
\end{aligned}
\end{equation}
We show in the next section that the reformulation of the volume integrals in \eqref{eq:RHSfluxform} plays an important role and that it is possible to generate many variants of DGSEM for the compressible Euler equations. 
We note that, in principle, it is possible to choose the numerical surface and volume flux differently. However, to restrict the myriads of possible combinations in this work, we couple the choice of the numerical volume flux and the numerical surface flux, detailed in Sec. \ref{sec:numflux} below.

\subsection{Split form DSGEM}\label{sec:split form}

In this section, we discuss three simple identities used to introduce specific numerical volume flux functions that recover the alternative compressible Euler operator split forms discussed in the introduction. To formulate these identities we consider three specific choices of the numerical volume flux, namely an arithmetic mean, the product of two arithmetic means and the product of three arithmetic means. We use the typical DG notation
\begin{equation}
\average{a}_{im}:= \frac{1}{2}(a_i + a_m),
\end{equation}
 for an arithmetic mean of a generic nodal quantity $a_i,\,\,i=0,...,N$. 

\begin{Lem}[Discrete split forms]\label{Lem}  Using the arithmetic mean, the product of two arithmetic means, or the product of three arithmetic means in the alternative numerical volume flux form of the DGSEM \eqref{eq:RHSfluxform}, it is possible to recover the standard DGSEM derivative form \eqref{eq:flux-differencing_vol}, the discrete quadratic split form \eqref{eq:two_split} and the discrete cubic split form \eqref{eq:three_split} respectively. More precisely, for generic nodal vector fields $\vec{a}=(a_0,...,a_N)^T$, $\vec{b}=(b_0,...,b_N)^T$, and $\vec{c}=(c_0,...,c_N)^T$,
\begin{equation}\label{eq:skewsym_identities}
\begin{aligned}
2\,\sum\limits_{m=0}^N D_{im}\,\average{a}_{im} &= (\dmat\,\vec{a})_i,\\
2\,\sum\limits_{m=0}^N D_{im}\,\average{a}_{im}\,\average{b}_{im} &=\frac{1}{2}\left(\dmat\,\uuline{a}\,\vec{b} + \uuline{a}\,\dmat\,\vec{b} + \uuline{b}\,\dmat\,\vec{a}\right)_i,\\
2\,\sum\limits_{m=0}^N D_{im}\,\average{a}_{im}\,\average{b}_{im}\,\average{c}_{im} &=\frac{1}{4}\left(\dmat\,\uuline{a}\,\uuline{b}\,\vec{c} + \uuline{a}\,\dmat\,\uuline{b}\,\vec{c} + \uuline{b}\,\dmat\,\uuline{a}\,\vec{c} + \uuline{c}\,\dmat\,\uuline{a}\,\vec{b} + \uuline{b}\,\uuline{c}\,\dmat\,\vec{a} + \uuline{a}\,\uuline{c}\,\dmat\,\vec{b} + \uuline{a}\,\uuline{b}\,\dmat\,\vec{c}\right)_i,
\end{aligned}
\end{equation}
where we introduce the additional notation that double underline denotes a matrix with a vector along the diagonal, e.g., $\uuline{a} = diag(\vec{a})$.
\end{Lem}

\begin{proof}
The proof of each of the split form identities is straightforward. However, each proof requires some algebraic manipulation unrelated to the focus of this paper. So we have moved the proofs to \ref{LemProof}.
\end{proof}

\begin{rem}
We note that the right hand sides of \eqref{eq:skewsym_identities} are straightforward discretisations of the quadratic and the cubic split forms \eqref{eq:two_split} and \eqref{eq:three_split}, respectively. 
\end{rem}

\begin{rem}\label{rem:advective_form}
By combining the split form identities, it is easy to see that the pure advective form of the quadratic product can be generated with the numerical volume flux of the form $\average{a}_{im}\,\average{b}_{im} -\frac{1}{2}\average{a\,b}_{im}$.
\end{rem}

\begin{rem}
We further note that all the derivations and proofs hold for general diagonal norm SBP operators and thus, for instance, all results in this paper directly carry over to finite difference diagonal norm SBP operators. 
\end{rem}

\begin{rem}
Analogously, it is possible to find equivalent forms of quartic products or products with even more terms when necessary. 
\end{rem}

With the identities presented in Lemma \ref{Lem}, it is now possible to state the second main result of this work, summarised in the following theorem.

\begin{thm}[Equivalent Numerical Volume Fluxes]\label{Thm2} A high-order accurate and consistent DGSEM discretisation of the alternative Euler formulations of Morinishi \eqref{eq:op_morinishi}, Ducros et al. \eqref{eq:op_ducros}, Kennedy and Gruber \eqref{eq:op_KG}, and Priozzoli \eqref{eq:op_pirozzoli} can be generated by translating the quadratic and cubic split forms into the corresponding numerical volume flux. To compress the notation for the volume fluxes we note that the averages are taken in each spatial direction. For example, the average density in the $x$-direction or $z$-direction are given by
\begin{equation}
\begin{aligned}
F^{\#,1}(U_{ijk},U_{mjk}) &= \average{\rho} = \half\left(\rho_{ijk} + \rho_{mjk}\right),\\[0.1cm]
H^{\#,1}(U_{ijk},U_{ijm}) &= \average{\rho} = \half\left(\rho_{ijk} + \rho_{ijm}\right),
\end{aligned}
\end{equation}
respectively. The volume fluxes for the standard DGSEM in the divergence form and for each alternative formulation of the Euler equations are:
\begin{itemize}
\item[] \underline{Standard DG \cite{koprivabook}}:
\begin{equation}
\begin{aligned}
F_{standard}^{\#}(U_{ijk},U_{mjk}) = &\begin{bmatrix} \average{\rho u} \\[0.1cm]
\average{\rho u^2} + \average{p} \\[0.1cm]
\average{\rho u v} \\[0.1cm]
\average{\rho u w} \\[0.1cm]
\average{u(\rho e + p)}\end{bmatrix},
\quad
G_{standard}^{\#}(U_{ijk},U_{imk}) = \begin{bmatrix} \average{\rho v} \\[0.1cm]
\average{\rho u v}  \\[0.1cm]
\average{\rho v^2} + \average{p} \\[0.1cm]
\average{\rho v w} \\[0.1cm]
\average{v(\rho e + p)}\end{bmatrix},
\\ \\
&\qquad\qquad H_{standard}^{\#}(U_{ijk},U_{ijm})= \begin{bmatrix} \average{\rho w} \\[0.1cm]
\average{\rho u w}\\[0.1cm]
\average{\rho v w} \\[0.1cm]
\average{\rho w^2} + \average{p} \\[0.1cm]
\average{w(\rho e + p)}\end{bmatrix}.
\end{aligned}
\end{equation}

\item[] \underline{Morinishi \cite{Morinishi2010276}}:
\begin{equation}
\resizebox{0.85\textwidth}{!}{$
\begin{aligned}
F_{MO}^{\#}(U_{ijk},U_{mjk}) = &\begin{bmatrix} \average{\rho\,u} \\[0.1cm]
\average{\rho u}\average{u} + \average{p} \\[0.1cm]
\average{\rho u}\average{v} \\[0.1cm]
\average{\rho u}\average{w} \\[0.1cm]
\average{(\rho\theta+p)\,u} + \average{\rho\,u^2}\average{u}+\average{\rho\,u\,v}\average{v}+\average{\rho\,u\,w}\average{w}-\frac{1}{2}\left(\average{\rho\,u^3}+\average{\rho\,u\,v^2}+\average{\rho\,u\,w^2} \right)\end{bmatrix},\\ \\
G_{MO}^{\#}(U_{ijk},U_{imk}) = &\begin{bmatrix} \average{\rho\,v} \\[0.1cm]
\average{\rho v}\average{u}\\[0.1cm]
\average{\rho v}\average{v} + \average{p}  \\[0.1cm]
\average{\rho v}\average{w} \\[0.1cm]
\average{(\rho\theta+p)\,v} + \average{\rho\,u\,v}\average{u}+\average{\rho\,v^2}\average{v}+\average{\rho\,v\,w}\average{w}-\frac{1}{2}\left(\average{\rho\,u^2\,v}+\average{\rho\,v^3}+\average{\rho\,v\,w^2} \right)\end{bmatrix},\\ \\
H_{MO}^{\#}(U_{ijk},U_{ijm}) = &\begin{bmatrix} \average{\rho\,w} \\[0.1cm]
\average{\rho w}\average{u}\\[0.1cm]
\average{\rho w}\average{v} + \average{p}  \\[0.1cm]
\average{\rho w}\average{w} \\[0.1cm]
\average{(\rho\theta+p)\,w} + \average{\rho\,u\,w}\average{u}+\average{\rho\,v\,w}\average{v}+\average{\rho\,w^2}\average{w}-\frac{1}{2}\left(\average{\rho\,u^2\,w}+\average{\rho\,v^2\,w}+\average{\rho\,w^3} \right)\end{bmatrix}.
\end{aligned}$}
\end{equation}

\item[] \underline{Ducros et al. \cite{ducros2000}}:
\begin{equation}
\begin{aligned}
F_{DU}^{\#}(U_{ijk},U_{mjk}) = &\begin{bmatrix} \average{\rho}\average{u} \\[0.1cm]
\average{\rho u}\average{u} + \average{p} \\[0.1cm]
\average{\rho v}\average{u} \\[0.1cm]
\average{\rho w}\average{u} \\[0.1cm]
(\average{\rho e} + \average{p})\average{u}\end{bmatrix},
\quad 
G_{DU}^{\#}(U_{ijk},U_{imk}) = \begin{bmatrix} \average{\rho}\average{v} \\[0.1cm]
\average{\rho u}\average{v}  \\[0.1cm]
\average{\rho v}\average{v} + \average{p}\\[0.1cm]
\average{\rho w}\average{v} \\[0.1cm]
(\average{\rho e} + \average{p})\average{v}\end{bmatrix},
\\ \\
&\qquad\qquad H_{DU}^{\#}(U_{ijk},U_{ijm}) = \begin{bmatrix} \average{\rho}\average{w} \\[0.1cm]
\average{\rho u}\average{w}  \\[0.1cm]
\average{\rho v}\average{w} \\[0.1cm]
\average{\rho w}\average{w}+ \average{p} \\[0.1cm]
(\average{\rho e} + \average{p})\average{w}\end{bmatrix}.
\end{aligned}
\end{equation}

\item[] \underline{Kennedy and Gruber \cite{kennedy2008}}:
\begin{equation}
\resizebox{0.86\textwidth}{!}{$
\begin{aligned}
F_{KG}^{\#}(U_{ijk},U_{mjk}) = &\begin{bmatrix} \average{\rho}\average{u} \\[0.1cm]
\average{\rho}\average{u}^2 + \average{p} \\[0.1cm]
\average{\rho}\average{u}\average{v} \\[0.1cm]
\average{\rho}\average{u}\average{w} \\[0.1cm]
\average{\rho}\average{u}\average{e} + \average{p}\average{u}\end{bmatrix},
\quad 
G_{KG}^{\#}(U_{ijk},U_{imk}) = \begin{bmatrix} \average{\rho}\average{v} \\[0.1cm]
\average{\rho}\average{u}\average{v}  \\[0.1cm]
\average{\rho}\average{v}^2 + \average{p}\\[0.1cm]
\average{\rho}\average{v}\average{w}\\[0.1cm]
\average{\rho}\average{v}\average{e} + \average{p}\average{v}\end{bmatrix},
\\ \\
&\qquad\qquad H_{KG}^{\#}(U_{ijk},U_{ijm})  = \begin{bmatrix} \average{\rho}\average{w} \\[0.1cm]
\average{\rho}\average{u}\average{w}  \\[0.1cm]
\average{\rho}\average{v}\average{w} \\[0.1cm]
\average{\rho}\average{w}^2+ \average{p} \\[0.1cm]
\average{\rho}\average{w}\average{e} + \average{p}\average{w}\end{bmatrix}.
\end{aligned}$}
\end{equation}

\item[] \underline{Pirozzoli \cite{pirozzoli2011}}: 
\begin{equation}\label{eq:numflux_priozzli}
\begin{aligned}
F_{PI}^{\#}(U_{ijk},U_{mjk}) = &\begin{bmatrix} \average{\rho}\average{u} \\[0.1cm]
\average{\rho}\average{u}^2 + \average{p} \\[0.1cm]
\average{\rho}\average{u}\average{v} \\[0.1cm]
\average{\rho}\average{u}\average{w} \\[0.1cm]
\average{\rho}\average{u}\average{h}\end{bmatrix},
\quad 
G_{PI}^{\#}(U_{ijk},U_{imk})= \begin{bmatrix} \average{\rho}\average{v} \\[0.1cm]
\average{\rho}\average{u}\average{v}  \\[0.1cm]
\average{\rho}\average{v}^2 + \average{p}\\[0.1cm]
\average{\rho}\average{v}\average{w}\\[0.1cm]
\average{\rho}\average{v}\average{h}\end{bmatrix},
\\ \\
&\qquad\qquad H_{PI}^{\#}(U_{ijk},U_{ijm}) = \begin{bmatrix} \average{\rho}\average{w} \\[0.1cm]
\average{\rho}\average{u}\average{w}  \\[0.1cm]
\average{\rho}\average{v}\average{w} \\[0.1cm]
\average{\rho}\average{w}^2+ \average{p} \\[0.1cm]
\average{\rho}\average{w}\average{h}\end{bmatrix}.
\end{aligned}
\end{equation}
\end{itemize}
\end{thm}

\begin{proof}
The proof follows directly when applying the results from Lem. \ref{Lem} to the respective alternative Euler formulations. All of the alternative formulations are built from either the divergence form, the advective form, the quadratic or cubic split forms. We note that we use the first part of Lem. \ref{Lem} to generate the standard flux divergence form. Furthermore, we need the pure advective form in the energy equation of the Morinishi (MO) form. To obtain the pure advective form we use the quadratic split form and subtract the divergence form scaled by one half, see Rem. \ref{rem:advective_form}. The other quadratic and cubic split forms are directly translated into their flux form using the second and third identity from the Lem. \ref{Lem}. Note that all fluxes are consistent, i.e. 
\begin{equation}
F^{\#}(U_{ijk},U_{ijk}) = F(U_{ijk}),
\end{equation}
and symmetric in their arguments, e.g.
\begin{equation}
F^{\#}(U_{ijk},U_{mjk}) = F^{\#}(U_{mjk},U_{ijk}).
\end{equation}
In the second part of the proof by Fisher et al. \cite{fisher2013} it was shown via Taylor expansion (second part of the proof of Thm. 3.1 on pages 15 and 16) that for general two-point entropy conserving flux functions the volume flux differencing approximation \eqref{eq:highorder_flux} gives a high-order accurate scheme. However, the only properties that are used to prove high-order accuracy are consistency of the numerical volume flux and its symmetry. Thus, we can directly extend the proof by Fisher et al. to the numerical volume fluxes presented in Thm. \ref{Thm2}, as all numerical volume fluxes are indeed symmetric and consistent in their arguments, as noted above. Another way of proving the high-order accuracy is to use again the identities of Lem. \ref{Lem}, as it is clear that the equivalent split form discretisations (right hand side of the identities in Lem. \ref{Lem}) are high-order accurate discretisations of the continuous split forms with the high-order accurate derivative matrix $\dmat$. 
\end{proof}

\begin{rem}
For central finite difference discretisations, Pirozzoli \cite{pirozzoli2010} already found a way to rewrite his alternative formulation \eqref{eq:op_pirozzoli} into an equivalent flux differencing form. In fact, ignoring boundary conditions, central finite difference schemes considered by Pirozzoli have the SBP property and thus the numerical volume flux described here \eqref{eq:numflux_priozzli} coincides with the one found by Pirozzoli. 
\end{rem}

\begin{rem} To compare the split form DGSEM \eqref{eq:RHSfluxform} to entropy conservative methods, we use the variants introduced by Carpenter et al. \cite{carpenter_esdg}. As stated above, the motivation of the work by Fisher et al. \cite{fisher2013} and Carpenter et al. \cite{carpenter_esdg} is the construction of an entropy stable discretisation. The volume terms of the DGSEM are entropy conserving when using a two-point entropy conserving flux function in the volume flux differencing formulation \cite{carpenter_esdg}. The numerical surface fluxes can then be used to control the dissipation of the discretisation and to guarantee that entropy is decreasing, i.e. entropy stability. To generate two DGSEM variants with entropy conserving volume terms, we choose the following two numerical volume fluxes:
\begin{itemize}
\item[] \underline{Ismail and Roe \cite{ismail2009}}: Introduce the parameter vector
\begin{equation}\label{zinProof}
\vec{z} = \left[\sqrt{\frac{\rho}{p}},\sqrt{\frac{\rho}{p}}u,\sqrt{\frac{\rho}{p}}v,\sqrt{\frac{\rho}{p}}w,\sqrt{\rho p}\right]^T,
\end{equation}
and the averaged quantities for the primitive variables
\begin{equation}
\begin{aligned}
&\qquad\qquad\qquad\hat{\rho} = \average{z_1}z_5^{\ln},\;\;\hat{u}=\frac{\average{z_2}}{\average{z_1}},\;\; \hat{v}=\frac{\average{z_3}}{\average{z_1}},\;\;\hat{w} = \frac{\average{z_4}}{\average{z_1}},\\[0.1cm]
&\hat{p}_1 = \frac{\average{z_5}}{\average{z_1}},\;\;\hat{p}_2 = \frac{\gamma+1}{2\gamma}\frac{z_5^{\ln}}{z_1^{\ln}} + \frac{\gamma-1}{2\gamma}\frac{\average{z_5}}{\average{z_1}},\;\; \hat{h} = \frac{\gamma \hat{p}_2}{\hat{\rho}(\gamma-1)} + \half(\hat{u}^2 + \hat{v}^2 + \hat{w}^2),
\end{aligned}
\end{equation}
where we introduce the logarithmic mean 
\begin{equation}\label{logMean}
(\cdot)^{\ln} \coloneqq \frac{(\cdot)_L - (\cdot)_R}{\ln((\cdot)_L) - \ln((\cdot)_R)},
\end{equation}
needed for the entropy conserving flux functions. We note that a numerically stable procedure to compute the logarithmic mean is described by Ismail and Roe \cite[Appendix B]{ismail2009}. The entropy conserving fluxes in each physical direction read as
\begin{equation}
\label{eq:IRflux}
\resizebox{0.86\textwidth}{!}{$
 F_{IR}^{\#}(U_{ijk},U_{mjk}) = \begin{bmatrix} \hat{\rho}\hat{u} \\[0.1cm]
\hat{\rho}\hat{u}^2 + \hat{p}_1 \\[0.1cm]
\hat{\rho}\hat{u}\hat{v} \\[0.1cm]
\hat{\rho}\hat{u}\hat{w} \\[0.1cm]
\hat{\rho}\hat{u}\hat{h}\end{bmatrix},
\quad 
G_{IR}^{\#}(U_{ijk},U_{imk})= \begin{bmatrix} \hat{\rho}\hat{v} \\[0.1cm]
\hat{\rho}\hat{u}\hat{v}  \\[0.1cm]
\hat{\rho}\hat{v}^2 + \hat{p}_1\\[0.1cm]
\hat{\rho}\hat{v}\hat{w}\\[0.1cm]
\hat{\rho}\hat{v}\hat{h}\end{bmatrix},
\quad H_{IR}^{\#}(U_{ijk},U_{ijm}) = \begin{bmatrix} \hat{\rho}\hat{w} \\[0.1cm]
\hat{\rho}\hat{u}\hat{w}  \\[0.1cm]
\hat{\rho}\hat{v}\hat{w} \\[0.1cm]
\hat{\rho}\hat{w}^2+ \hat{p}_1 \\[0.1cm]
\hat{\rho}\hat{w}\hat{h}\end{bmatrix}.$}
\end{equation}

\item[] \underline{Chandrashekar \cite{chandrashekar2013}}: Introduce notation for the inverse of the temperature 
\begin{equation}
\label{eq:beta}
\beta = \frac{1}{RT} = \frac{\rho}{2 p},
\end{equation}
and the average pressure and enthalpy
\begin{equation}
\resizebox{0.85\textwidth}{!}{$
\hat{p} = \frac{\average{\rho}}{2\average{\beta}},\quad \hat{h} = \frac{1}{2\beta^{\ln}(\gamma-1)} -\half\left(\average{u^2} + \average{v^2} + \average{w^2}\right) + \frac{\hat{p}}{\rho^{\ln}} + \average{u}^2+ \average{v}^2+ \average{w}^2,$}
\end{equation}
then we have the entropy conserving and kinetic energy preserving flux components
\begin{equation}\label{EKEP}
\resizebox{0.85\textwidth}{!}{$
F_{CH}^{\#}(U_{ijk},U_{mjk}) = \begin{bmatrix} \rho^{\ln}\average{u} \\[0.1cm]
{\rho^{\ln}}\average{u}^2 + \hat{p} \\[0.1cm]
{\rho^{\ln}}\average{u}\average{v} \\[0.1cm]
{\rho^{\ln}}\average{u}\average{w} \\[0.1cm]
{\rho^{\ln}}\average{u}\hat{h}\end{bmatrix},
\quad 
G_{CH}^{\#}(U_{ijk},U_{imk})= \begin{bmatrix} \rho^{\ln}\average{v} \\[0.1cm]
{\rho^{\ln}}\average{u}\average{v}  \\[0.1cm]
{\rho^{\ln}}\average{v}^2 + \hat{p}\\[0.1cm]
{\rho^{\ln}}\average{v}\average{w} \\[0.1cm]
{\rho^{\ln}}\average{v}\hat{h}\end{bmatrix},
\quad 
H_{CH}^{\#}(U_{ijk},U_{ijm}) = \begin{bmatrix} \rho^{\ln}\average{w} \\[0.1cm]
{\rho^{\ln}}\average{u}\average{w} \\[0.1cm]
{\rho^{\ln}}\average{v}\average{w} \\[0.1cm]
{\rho^{\ln}}\average{w}^2 + \hat{p}  \\[0.1cm]
{\rho^{\ln}}\average{w}\hat{h}\end{bmatrix}.$}
\end{equation}
\end{itemize}
\end{rem}

\subsection{Kinetic energy preservation of the split forms}\label{sec:kep}

In \cite{jameson2008} Jameson derived a general condition on a numerical flux function for finite volume schemes to generate kinetic energy preserving discretisations. The kinetic energy is 
\begin{equation}
\kappa:=\frac{1}{2}\rho\,(u^2+v^2+w^2) = \frac{(\rho\,u)^2+(\rho\,v)^2+(\rho\,w)^2}{2\,\rho}.
\end{equation}
The kinetic energy balance can be directly derived from the compressible Euler equations and reads 
\begin{equation}
\label{eq:kep_balance_paper}
\resizebox{0.9\textwidth}{!}{$
\frac{\partial \kappa}{\partial t} + \left(\frac{1}{2}\rho\,u\,(u^2+v^2+w^2)\right)_x + \left(\frac{1}{2}\rho\,v\,(u^2+v^2+w^2)\right)_y+ \left(\frac{1}{2}\rho\,w\,(u^2+v^2+w^2)\right)_z +u\,p_x  +v\,p_y +w\,p_z=0.
$}
\end{equation}

Note, that the influence of the advective terms of the momentum flux $\rho\,u^2,\,\rho\,u\,v,\,\rho\,u\,w,...$ can be re-cast into flux form, whereas the pressure gradient in the momentum flux yields a non-conservative term (pressure work) in the kinetic energy balance \eqref{eq:kep_balance_paper}. It is exactly this structure of the kinetic energy balance that is reflected in the definition of kinetic energy preserving by Jameson: Ignoring boundary conditions, a discretisation is termed kinetic energy preserving when the (discrete) integral of the kinetic energy is not changed by the advective terms, but only by the pressure work.

For finite volume schemes, Jameson \cite{jameson2008} found the following general form for the components of the numerical surface fluxes of the momentum equations that give kinetic energy preservation:
\begin{alignat}{3}\label{eq:kep_property}
&F^{*,2} = F^{*,1}\,\average{u} + \widetilde{p},\quad &&F^{*,3} = F^{*,1}\,\average{v},\quad &&F^{*,4} = F^{*,1}\,\average{w},\nonumber\\
&G^{*,2} = G^{*,1}\,\average{u} ,\quad &&G^{*,3} = G^{*,1}\,\average{v}+\widetilde{p},\quad &&G^{*,4} = G^{*,1}\,\average{w},\\
&H^{*,2} = H^{*,1}\,\average{u} ,\quad &&H^{*,3} = H^{*,1}\,\average{v},\quad &&H^{*,4} = H^{*,1}\,\average{w}+\widetilde{p},\nonumber
\end{alignat}
where $\widetilde{p}$ can be any consistent numerical trace approximation of the pressure. This structure is enough to guarantee that the advective terms in the resulting discrete kinetic energy balance are consistent and in conservative form. \textred{We note that in the original conditions for kinetic energy preservation of Jameson, the pressure discretisation $\widetilde{p}$ can be any consistent approximation. However, we will show in the numerical results section that the discretisation of the pressure plays an important role for the discrete kinetic energy preservation. We will show numerically that the best results for discrete kinetic energy preservation are achieved when an arithmetic mean is used for $\widetilde{p}$.}\\
\\ 
If we inspect the numerical volume flux functions presented in Thm. \ref{Thm2}, the conditions \eqref{eq:kep_property} hold for the variants MO, KG, PI and CH. We will demonstrate that the conditions \eqref{eq:kep_property} are enough to guarantee discrete kinetic energy preservation for the flux difference formulation \eqref{eq:RHSfluxform}.

Gassner \cite{gassner_kepdg} showed that Morinishi's alternative formulation allows one to construct a kinetic energy preserving DG discretisation. Using the same approach, it is possible to prove that KG and PI are kinetic energy preserving as well. In contrast, the method of \eqref{eq:kep_property} cannot be applied to the CH formulation, as no explicit split form is known. However, it is possible to make use of the flux differencing form of Sec. \ref{sec:DGSEM_fluxform} to prove that numerical volume fluxes satisfying the conditions \eqref{eq:kep_property} generate high-order accurate schemes that are kinetic energy preserving in the sense of Jameson, i.e. that the advective terms can be re-cast into a conservative form \textred{and only the approximations of the pressure work is in non-conservative form}. We summarise the third main result of this work in the following theorem. 

\begin{thm}[Kinetic energy preservation]\label{Thm3} If the numerical volume flux functions $F^\#$, $G^\#$, and $H^\#$ satisfy the condition \eqref{eq:kep_property}, the volume terms of the resulting flux differencing discretisation \eqref{eq:RHSfluxform} are kinetic energy preserving in the sense of Jameson \cite{jameson2008}, \textred{i.e. that the advective terms can be re-cast into a conservative form and only the approximations of the pressure work is in non-conservative form}. More precisely, the volume terms of the discrete kinetic energy balance are given by
\begin{equation}\label{eq:kin_pres}
\begin{split}
\left(\frac{\partial \kappa}{\partial t}\right)_{ijk}\approx &-2\,\sum\limits_{m=0}^N D_{im}\,f^{\#,\kappa}(U_{ijk},U_{mjk})+D_{jm}\,g^{\#,\kappa}(U_{ijk},U_{imk})  
+D_{km}\,h^{\#,\kappa}(U_{ijk},U_{ijm})\\
&-2\,\sum\limits_{m=0}^N u_{ijk}\,D_{im}\,\widetilde{p}(U_{ijk},U_{mjk})+v_{ijk}\,D_{jm}\,\widetilde{p}(U_{ijk},U_{imk})+w_{ijk}\,D_{km}\,\widetilde{p}(U_{ijk},U_{ijm}),
\end{split}
\end{equation}
with a consistent and symmetric pressure discretisation $\widetilde{p}$ and with the consistent and symmetric advective fluxes of the discrete kinetic energy 
\begin{equation}
\begin{split}
f^{\#,\kappa}(U_{ijk},U_{mjk}) &:=\frac{1}{2}\,F^{\#,1}(U_{ijk},U_{mjk})\left(u_{ijk}\,u_{mjk}+v_{ijk}\,v_{mjk}+w_{ijk}\,w_{mjk}\right),\\
g^{\#,\kappa}(U_{ijk},U_{imk}) &:=\frac{1}{2}\,G^{\#,1}(U_{ijk},U_{imk})\left(u_{ijk}\,u_{imk}+v_{ijk}\,v_{imk}+w_{ijk}\,w_{imk}\right),\\
h^{\#,\kappa}(U_{ijk},U_{ijm}) &:=\frac{1}{2}\,H^{\#,1}(U_{ijk},U_{ijm})\left(u_{ijk}\,u_{ijm}+v_{ijk}\,v_{ijm}+w_{ijk}\,w_{ijm}\right).
\end{split}
\end{equation}
\end{thm}

\begin{proof}
Again, we move the proof to the \ref{Thm3Proof}.
\end{proof}

\begin{rem}
Note, that the terms of the discrete kinetic energy balance \eqref{eq:kin_pres} are consistent to the continuous kinetic energy balance \eqref{eq:kep_balance_paper} \textred{and that the advective terms are in conservative form and the pressure work is in non-conservative form. This property guarantees that the discrete total kinetic energy is not dissipated by the advective terms. Hence, this property is important to construct discretisations that have low artificial dissipation of kinetic energy.} 
\end{rem}

\begin{rem}
The kinetic energy preserving flux function proposed by Jameson is identical to the $PI$ flux \eqref{eq:numflux_priozzli}.
\end{rem}

\begin{cor}
In Gassner \cite{gassner_kepdg} it was shown that when the volume terms of the DGSEM are kinetic energy preserving and the numerical surface flux function at element interfaces satisfies the general conditions \eqref{eq:kep_property} then the multi-element discretisation is kinetic energy preserving as well. Thus, if the numerical volume and surface flux are the same it follows that the DGSEM approximation satisfies the kinetic energy preservation condition \eqref{eq:kep_property}. In particular, the variants MO, KG, PI and CH are kinetic energy preserving in the sense that the (discrete) integral of the kinetic energy is only changed by the pressure work (ignoring boundary conditions). 
\end{cor}

\subsection{Numerical surface flux functions}\label{sec:numflux}

As mentioned above, we connect the choice of the volume numerical flux to the choice of the surface numerical flux. As a blueprint, we use the local Lax-Friedrichs numerical flux function, e.g. 
\begin{equation}
F^*(U_-,U_+) := \frac{1}{2}\left(F(U_-) + F(U_+)\right) - \frac{1}{2}\,\lambda_{max}\,\left[U_+ - U_-\right],
\end{equation}
where $\lambda_{max}$ is an estimate of the fastest wave speed at the interface. Note that the structure of the local Lax-Friedrichs flux (LLF) is a consistent symmetric part $\frac{1}{2}\left(F(U_-) + F(U_+)\right)$ and a stabilisation term $- \frac{1}{2}\,\lambda_{max}\,\left[U_+ - U_-\right]$, which introduces dissipation that depends on the jump of the approximation at the interface.  

Analogously, we replace the consistent symmetric part by the consistent symmetric numerical volume flux function $F^\#$, which is also used to generate the volume terms of the DGSEM. We add a stabilisation term $- Stab(U_-,U_+)$, that we design so that it introduces dissipation at the interface depending on the jump of the approximate solution
\begin{equation}
\label{eq:numericalflux}
F^*(U_-,U_+) := F^\#(U_-,U_+) - Stab(U_-,U_+).
\end{equation}
For the variants $MO$, $DU$, $KG$, and $PI$ we choose the same stabilisation term as in the local Lax Friedrichs flux
\begin{equation}
\label{eq:llf}
Stab_{MO}(U_-,U_+)=Stab_{DU}(U_-,U_+)=Stab_{KG}(U_-,U_+)=Stab_{PI}(U_-,U_+):=\frac{1}{2}\,\lambda_{max}\,\left[U_+ - U_-\right],
\end{equation} 
where $\lambda_{max}$ is the maximum of the maximum wave speeds from the left and right at the interface. The stabilisation term for $CH$ is a direct 3D extension of the term introduced in \cite{chandrashekar2013} and is identical to \eqref{eq:llf} for the first four components, but differs in the fifth component (energy equation):
\begin{equation}
\resizebox{0.9\hsize}{!}{$\begin{split}
Stab^5_{CH}&(U_-,U_+):=\frac{1}{2}\lambda_{max}\Bigg\{ \left[\frac{1}{2(\gamma -1)\beta^{ln}} + (u^+u^-+v^+v^-+w^+w^-) \right]\,[\rho^+-\rho^-]\\ 
&+ \average{\rho}\average{u}[u^+-u^-] + \average{\rho}\average{v}[v^+-v^-]+ \average{\rho}\average{w}[w^+-w^-]+\frac{\average{\rho}}{2(\gamma-1)}\left[\frac{1}{\beta^+} - \frac{1}{\beta^-} \right]\Bigg\},
\end{split}$}
\end{equation} 
where $\beta$ is introduced in \eqref{eq:beta}. 

Note that the standard LLF type dissipation terms in the first four components ensures guaranteed kinetic energy dissipation \cite{chandrashekar2013}, whereas the specific fifth component of $CH$ stabilisation term additionally ensures guaranteed entropy dissipation \cite{chandrashekar2013}. For the $IR$ stabilisation, we use the term introduced in \cite{carpenter_esdg}
\begin{equation}
Stab_{IR}(U_-,U_+) = \frac{1}{2}\lambda_{max}\,\mathcal{H}\,[V^+ - V^-],
\end{equation}
where $V = [\frac{\gamma - s}{\gamma-1}-\frac{\rho\,\|u\|^2}{2\,p}, \frac{\rho\,u}{p}, \frac{\rho\,v}{p}, \frac{\rho\,w}{p}, -\frac{\rho}{p} ]^T$ are the entropy variables and $s=\ln(p) - \gamma\,\ln(\rho)$ is the specific thermodynamic entropy. The matrix $\mathcal{H}$ is the symmetric positive definite entropy Jacobian of the conservative variables
\begin{equation}
\mathcal{H}:=\frac{\partial U}{\partial V}  
= \begin{bmatrix}
\rho & \rho u & \rho v & \rho w & \rho e \\[0.1cm]
\rho u & \rho u^2 + p & \rho u v & \rho u w & \rho {h} u \\[0.1cm]
\rho v & \rho u v & \rho v^2 + p& \rho v w & \rho {h} v \\[0.1cm]
\rho w & \rho u w & \rho v w & \rho w^2 + p & \rho {h} w \\[0.1cm]
\rho e  & \rho{h} u & \rho{h} v & \rho{h} w & \rho{h}^2 - \frac{a^2 p }{\gamma-1}\\[0.1cm]
\end{bmatrix}, 
\end{equation}
where the sound speed is defined as $a^2 = \frac{\gamma\,p}{\rho}$. The maximum wave speed $\lambda_{max}$ and the entries of the entropy Jacobian $\mathcal{H}$ are evaluated at the average states used to compute the $IR$ numerical flux functions \eqref{eq:IRflux}. Again, this stabilisation term is designed so that the discrete entropy is dissipated at element interfaces. Thus, the variants IR and CH with stabilisation terms are entropy stable.

\section{Numerical Investigations}\label{sec:numerical results}

The main purpose of this section is to investigate the \textblue{robustness} of the new split form DGSEM with respect to different choices of grid resolution and/or polynomial degree. We compare the robustness and behaviour of the standard DGSEM, the variant DU, the new split form variants with kinetic energy preserving volume terms MO, KG and PI, the variant with entropy conserving volume terms IR (Carpenter et al. \cite{carpenter_esdg}) and the variant with entropy conserving and kinetic energy preserving volume terms CH. 

We start however with an investigation of the experimental order of convergence, discrete entropy conservation and kinetic energy preservation. All simulation results are obtained with an explicit five stage fourth order accurate low storage Runge-Kutta scheme, where the time step is computed according to a CFL type condition with the local maximum wave speed and the relative grid size $\Delta x/(N+1)$, where $\Delta x$ is the element size. If not specified otherwise, we use $CFL=0.5$ for all computations. The adiabatic coefficient is chosen as $\gamma=1.4$.

\subsection{$h$- and $p$-convergence}

We first investigate the $h$- and $p$-convergence properties of the schemes. We choose the following manufactured solution for the unsteady compressible Euler equations
\begin{equation}
\begin{split}
\rho &= 2 + \frac{1}{10}\,\sin(\pi\,(x+y+z - 2\,t)),\\
u&=1,\\
v&=1,\\
w&=1,\\
\rho\,e &= \left(2 + \frac{1}{10}\,\sin(\pi\,(x+y+z - 2\,t))\right)^2, 
\end{split}
\end{equation}
with the corresponding source term 
\begin{equation}
\begin{split}
q_{\rho} &= c_1\,\cos(\pi\,(x+y+z - 2\,t))\\
q_{\rho\,u} &= c_2\,\cos(\pi\,(x+y+z - 2\,t))+c_3\,\cos(2\,\pi\,(x+y+z - 2\,t)),\\
q_{\rho\,v} &= c_2\,\cos(\pi\,(x+y+z - 2\,t))+c_3\,\cos(2\,\pi\,(x+y+z - 2\,t)),\\
q_{\rho\,w} &= c_2\,\cos(\pi\,(x+y+z - 2\,t))+c_3\,\cos(2\,\pi\,(x+y+z - 2\,t)),\\
q_{\rho\,e} &= c_4\,\cos(\pi\,(x+y+z - 2\,t))+c_5\,\cos(2\,\pi\,(x+y+z - 2\,t)),\\
\end{split}
\end{equation}
where $c_1=\frac{\pi}{10} $, $c_2=-\frac{1}{5}\,\pi+\frac{1}{20}\,\pi\,(1+5\,\gamma)$, $c_3 =\frac{\pi}{100}\,(\gamma -1)$, $c_4 = \frac{1}{20}\left(-16\,\pi + \pi\,(9+15\,\gamma)\right)$, and $c_5 = \frac{1}{100}\,\left(3\,\pi\,\gamma-2\,\pi\right)$. The end time is set to $t_{end}=10$. The computational domain is a fully periodic box with the size $[-1,1]^3$.\\
\\ 
As stated above, derivations in this work are valid for all diagonal norm SBP operators, such as those from the finite difference community. A conceptional difference of a spectral element method to a finite difference approximation is the ability for $p$-convergence. In $p$-convergence studies the grid is fixed and the number of the nodes inside an element is increased. In a DGSEM, increasing the number of element nodes automatically increases the polynomial degree of the approximation and hence increases the approximation order in tandem with the resolution. The resulting \textit{spectral} convergence can be observed in Fig. \ref{fig:p_conv}, where all schemes are investigated with and without interface stabilisation on a regular $4^3$ grid.  As an example, the $L_2$-errors in density are shown. All other quantities show a similar behaviour. The $L_2$-norm is computed discretely with the collocated GL quadrature used for the scheme. 

For this test, the errors of the different schemes are remarkably close, especially when interface stabilisation is activated. Note that without interface stabilisation, the numerical surface flux is symmetric in its arguments. This causes an odd/even effect, which can be clearly observed in the left plot of Fig. \ref{fig:p_conv} and is in accordance to similar observations \cite{gassner_skew_burgers,gassner_kepdg,gassner2015}. The simulations for higher polynomial degree $N$ plateau out at about $10^{-8}$ in the right plot due to the accuracy of the time integration scheme.  By decreasing the CFL from $0.5$ to $0.25$, we can see in the right plot of Fig. \ref{fig:p_conv} that the level of the plateau decreases exactly by a factor of $16$ as expected for the fourth order accurate  RK method. 
\begin{figure}[!htbp]
\centering
\includegraphics[trim=0 0 0 0,clip,width=0.45\textwidth]{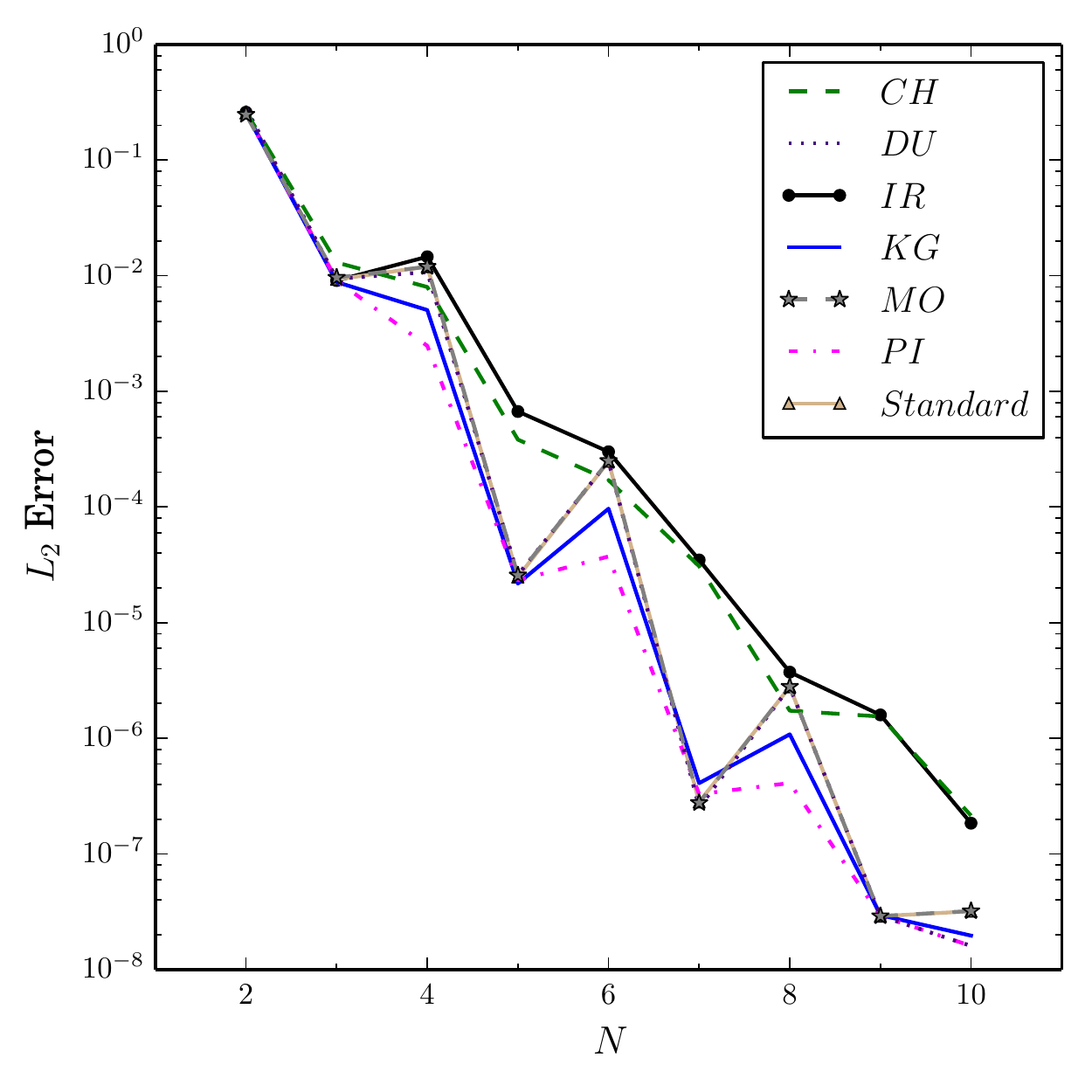}\hspace*{2mm}\includegraphics[trim=0 0 0 0,clip,width=0.45\textwidth]{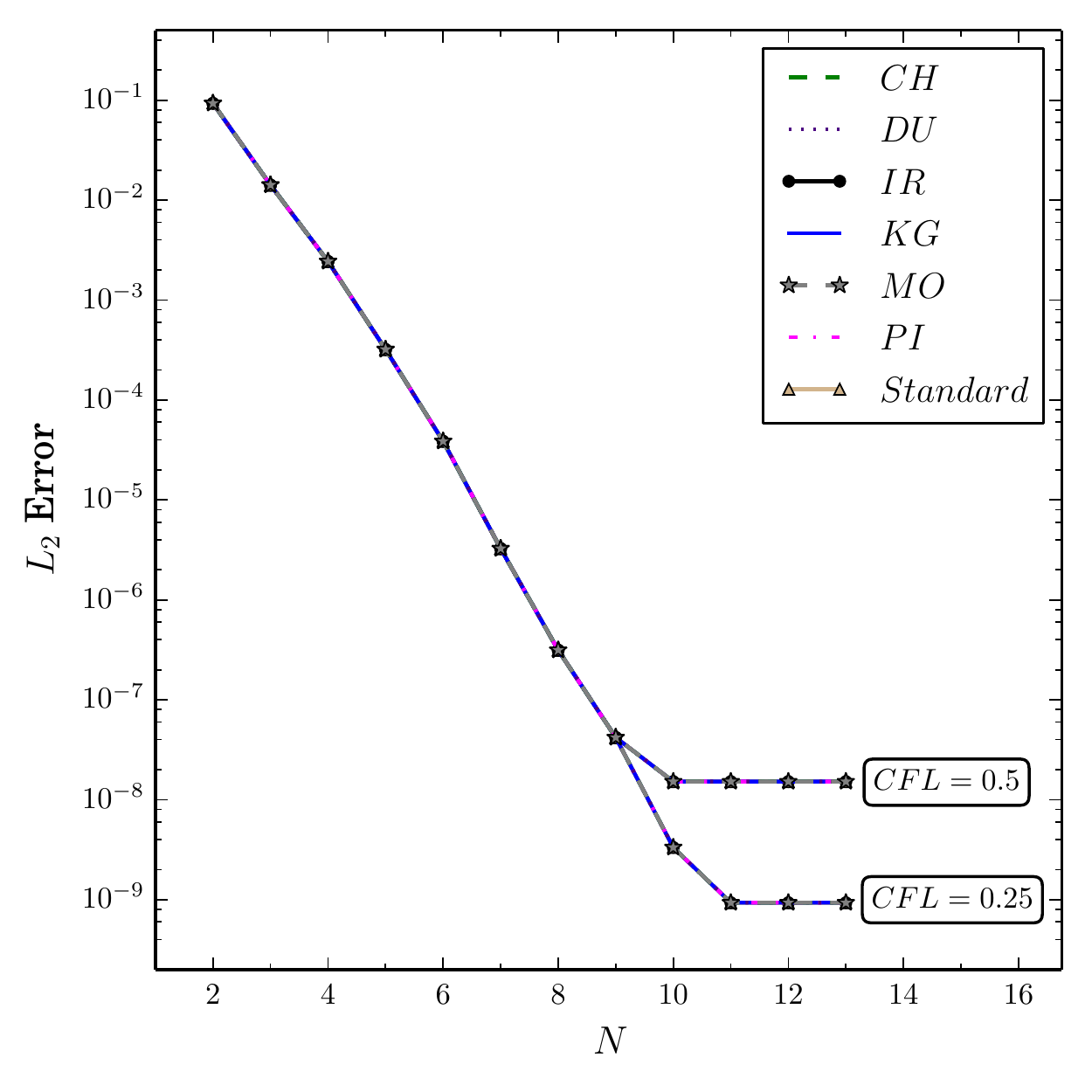}
\caption{\label{fig:p_conv} $p$-convergence for all the schemes on a regular $4^3$ grid. Plot of the $L_2$-error in density is shown. Left: without interface stabilisation. Right: with interface stabilisation.}
\end{figure}

The next figure, Fig. \ref{fig:h_conv_nostab}, shows the $h$-convergence behaviour for the different split form DGSEM schemes without stabilisation for polynomial degree $N=3$ and $N=4$. Without interface stabilisation terms, we again obtain the odd/even behaviour: for odd polynomial degrees $N$ we do not get the optimal convergence rate $N+1$, but only a convergence rate of $N$. 
\begin{figure}[!htbp]
\centering
\includegraphics[trim=0 0 0 0,clip,width=0.45\textwidth]{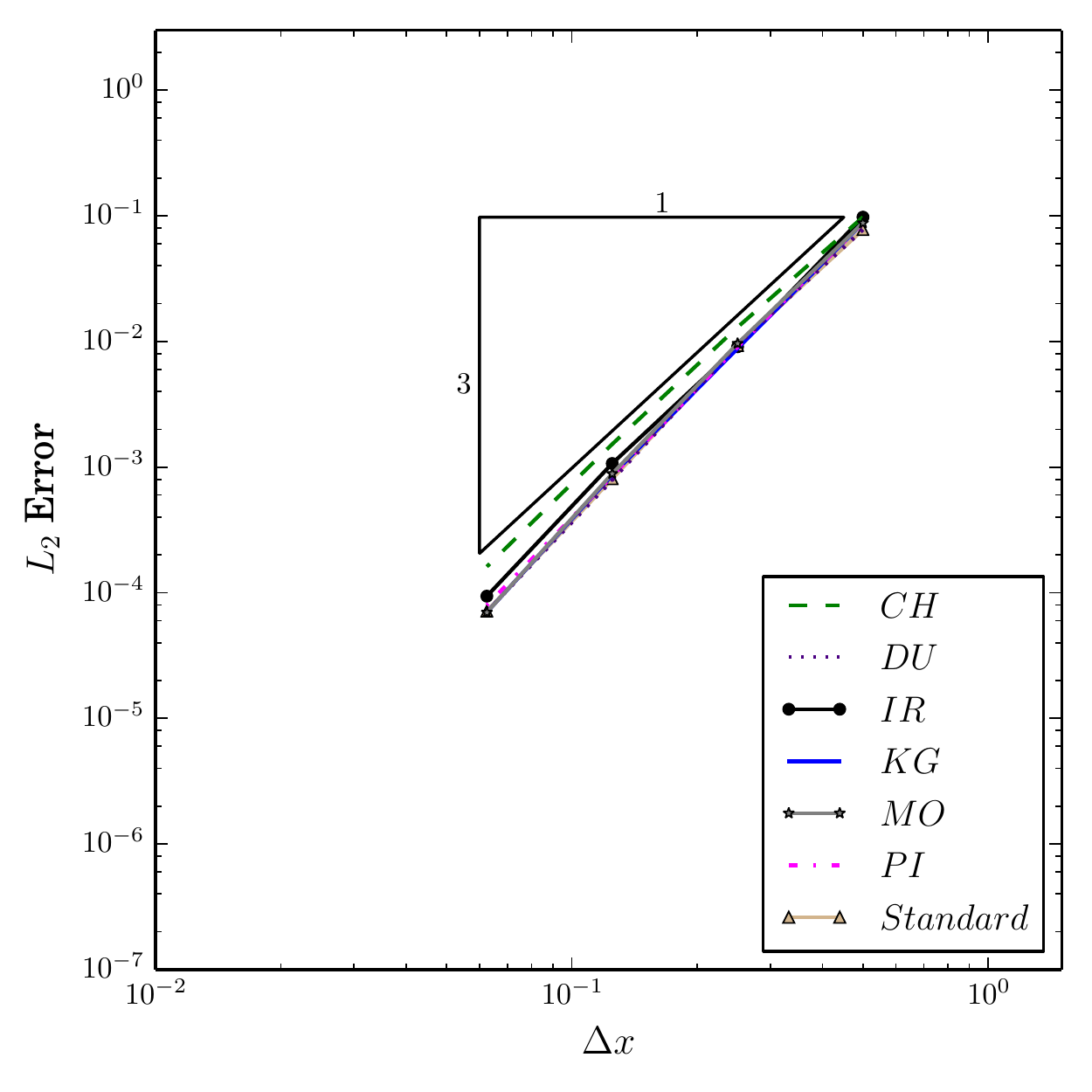}\hspace*{2mm}\includegraphics[trim=0 0 0 0,clip,width=0.45\textwidth]{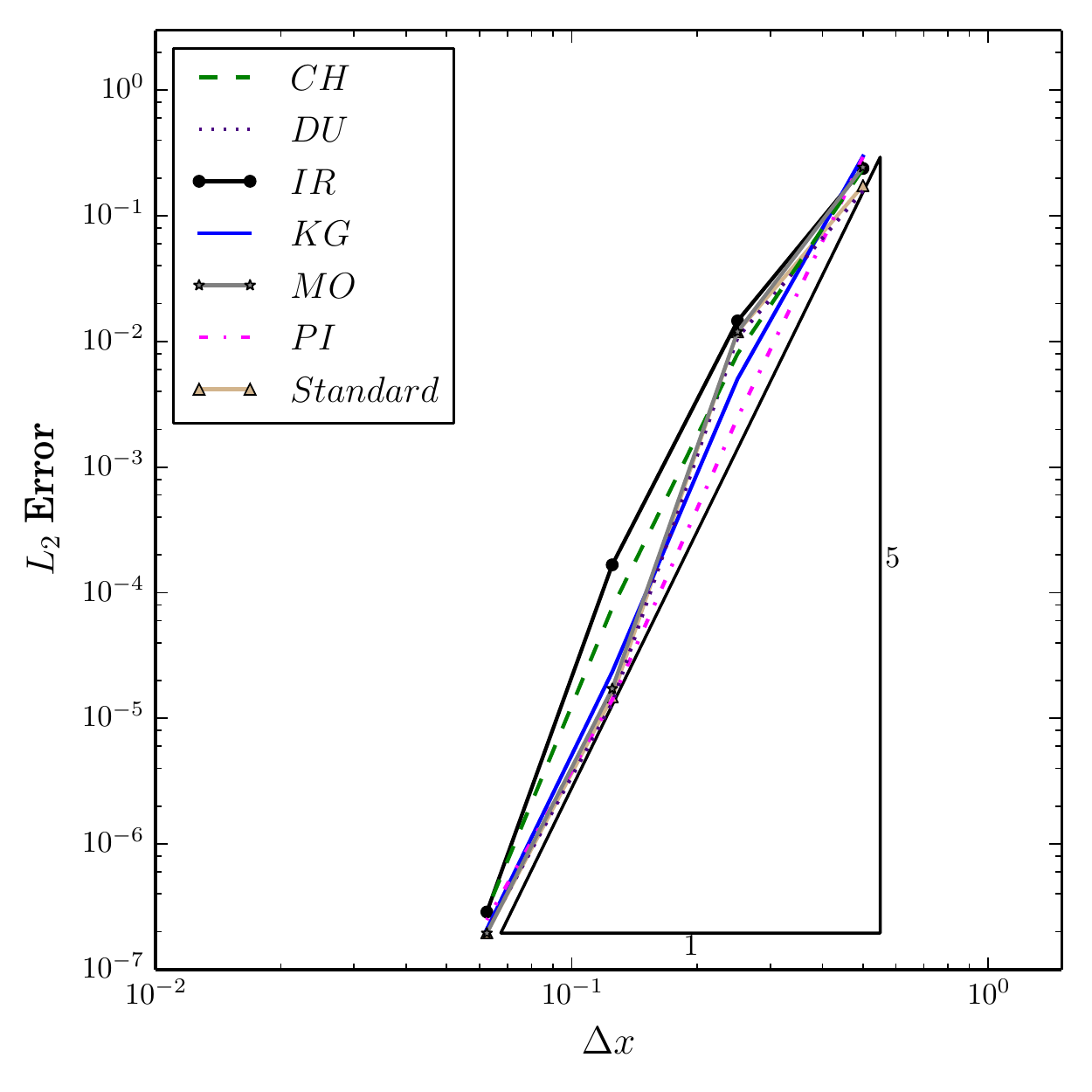}
\caption{\label{fig:h_conv_nostab} $h$-convergence for all the schemes without interface stabilisation. Grid sequence goes from $2^3$ up to $16^3$. Plot of the $L_2$-error in density is shown. The odd/even effect can be seen. Left: $N=3$, experimental order of convergence $\approx 3$. Right: $N=4$, experimental order of convergence $\approx 5$.}
\end{figure}

If we switch on the interface stabilisation terms, again, the errors produced by the different fluxes are almost the same. All the schemes show now the optimal order of convergence, i.e. for a polynomial of degree $N$, we get $N+1$st order convergence in $h$ as can be seen in Fig. \ref{fig:h_conv_stab}.  
\begin{figure}[!htbp]
\centering
\includegraphics[trim=0 0 0 0,clip,width=0.45\textwidth]{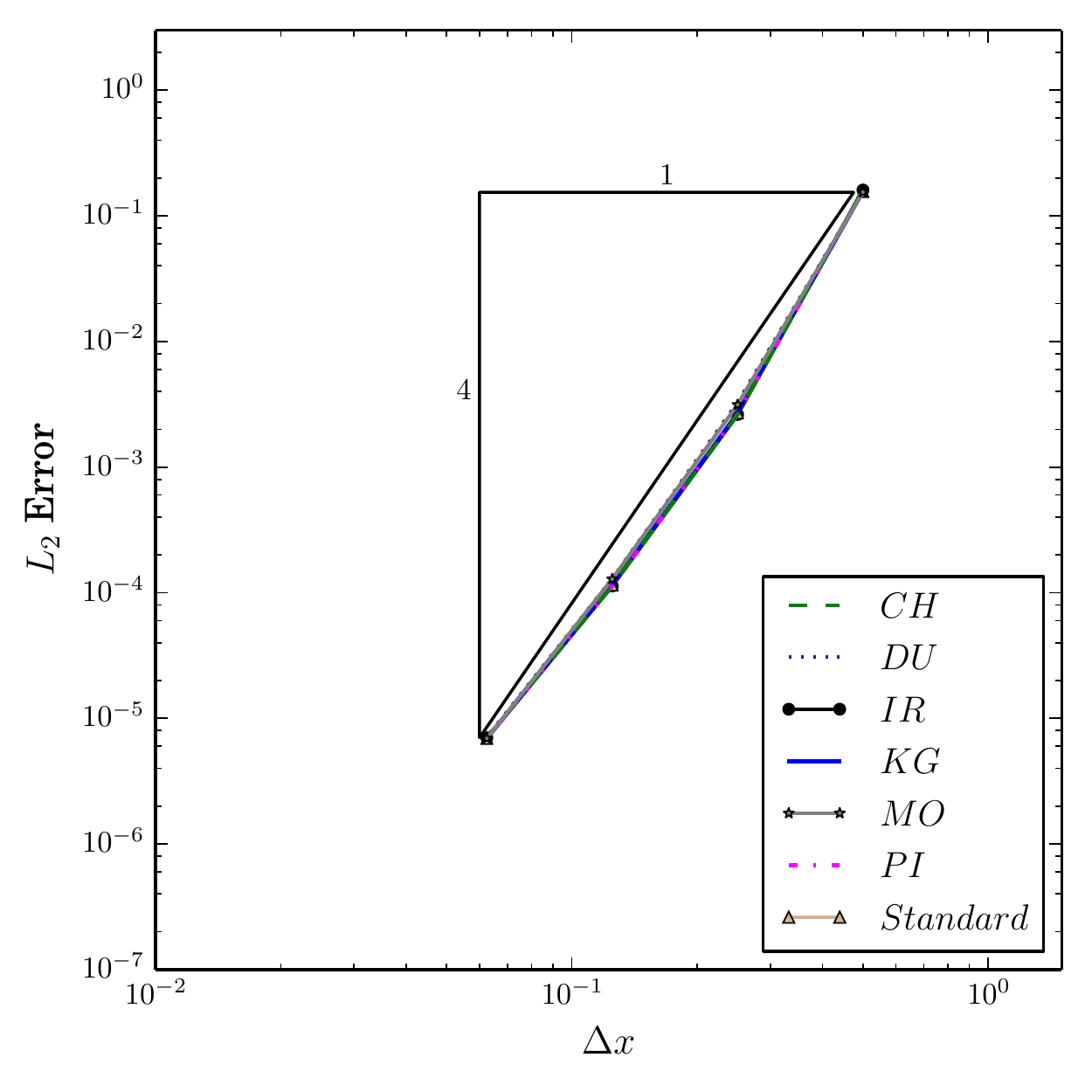}\hspace*{2mm}\includegraphics[trim=0 0 0 0,clip,width=0.45\textwidth]{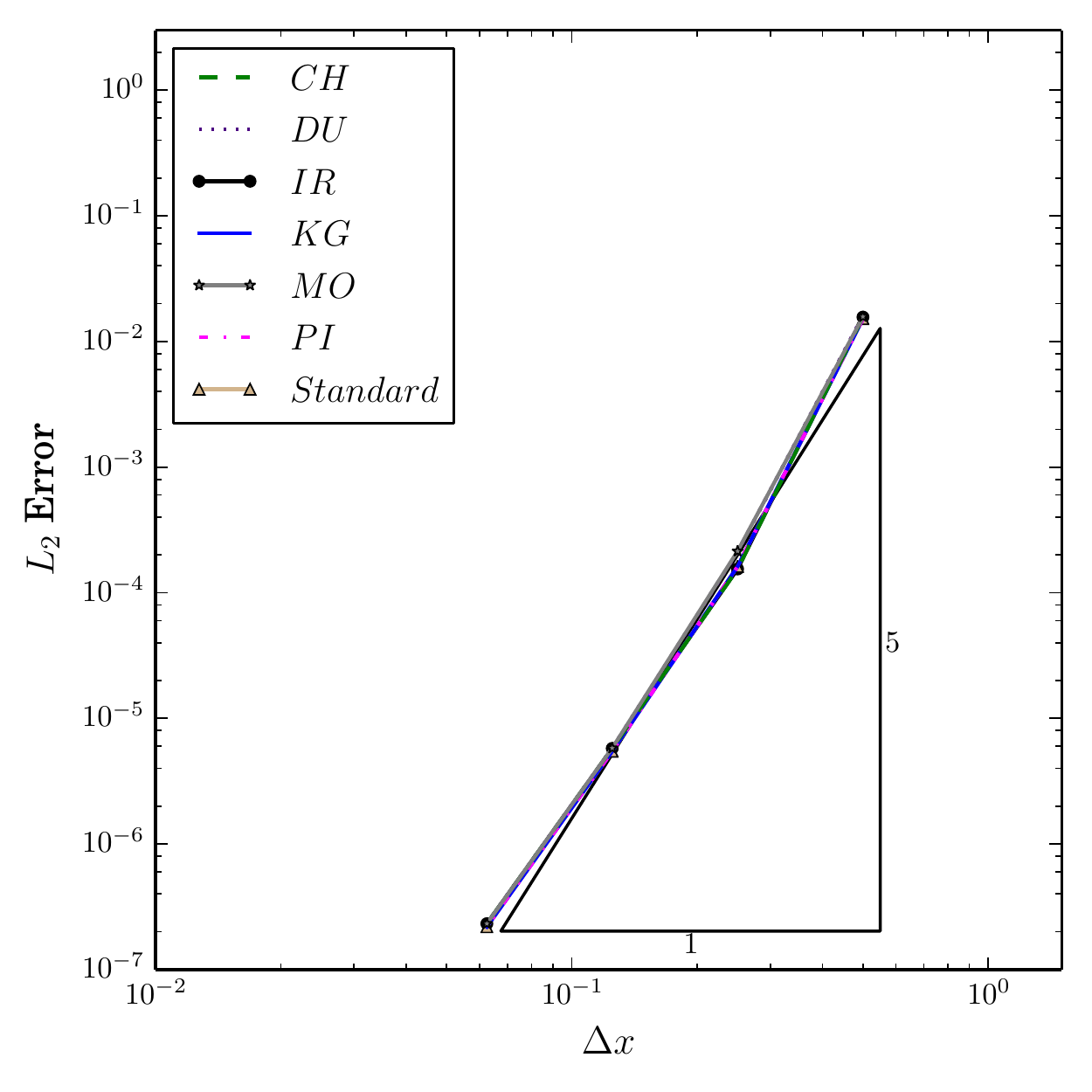}\caption{\label{fig:h_conv_stab} $h$-convergence for the all schemes with interface stabilisation. Grid sequence goes from $2^3$ up to $16^3$. Plot of the $L_2$-error in density is shown. For polynomial degree $N$, we get $N+1$ convergence in $h$. Left: $N=3$. Right: $N=4$.}
\end{figure}

In summary, all schemes show the expected $p$- and $h$-convergence behaviour. 

\subsection{Entropy conservation and kinetic energy preservation}

All presented schemes conserve mass, momentum, and total energy by construction. In this section, we focus on the auxiliary conservation properties of the methods and investigate the entropy \textred{conservation} and kinetic energy \textred{preservation} of the schemes. To investigate the auxiliary conservation properties, we deactivate the numerical dissipation introduced by the surface stabilisation terms and only use the symmetric and consistent parts at element interfaces. 

In contrast to the last section, it does not make sense to investigate the auxiliary conservation properties with well resolved smooth solutions, as in this case all methods would converge spectrally fast to them, if the solution supports it. To make this investigation challenging, we consider the inviscid Taylor-Green vortex. The initial condition in the periodic $[0,2\,\pi]^3$ box is 
\begin{equation}
\begin{split}
\rho &= 1,\\
u &= \sin(x)\,\cos(y)\,\cos(z),\\
v &= -\cos(x)\,\sin(y)\,\cos(z),\\
w &= 0,\\
p &= \frac{100}{\gamma} + \frac{1}{16}\,\left(\cos(2\,x)\,\cos(2\,z) + 2\,\cos(2\,y)+2\,\cos(2\,x)+\cos(2\,y)\,\cos(2\,z)\right).
\end{split}
\end{equation}
The evolution of these simple initial conditions is quite challenging due to the non-linear interaction of scales as well as the transition to a turbulence like flow field that occurs for large enough times. Without physical viscosity, there is no small scale limit and thus approximations of the inviscid Taylor-Green vortex are always under-resolved for large enough times. This behaviour makes this test case a challenge for the robustness of high-order methods. 

Before we investigate the robustness of the schemes, we use this test setup to investigate their auxiliary \textred{conservation/preservation properties}. The inviscid Taylor-Green vortex solution conserves both the total kinetic energy and the total entropy for all times. \textred{Although total kinetic energy is in general not a conserved quantity, it is for the specific inviscid Taylor-Green vortex setup in fully periodic domains. Due to periodicity, the pressure work when integrated over the domain cancels out and does not change the total kinetic energy. This allows us to experimentally investigate the actual kinetic energy behaviour of the different schemes. Note, that the kinetic energy preservation property of Jameson is only a statement for the influence of the advective terms on the total kinetic energy. However, we will show that discretisation errors in the pressure and velocity naturally cause errors in the pressure work contribution to the discrete total kinetic energy as well, i.e. the discrete kinetic energy is not conserved discretely to machine precision for this case. Our investigations clearly show that the discretisation of the pressure matters and that the simple arithmetic mean yields the best results.} 

As mentioned above, these investigations only make sense when the interface stabilisation is omitted. Otherwise, the numerical viscosity would dissipate kinetic energy and entropy when the approximation is under-resolved. However, without interface stabilisation the resulting discretisations are basically dissipation free and the robustness is fragile. A general observation is that the higher the polynomial degree and the higher the overall number of spatial degrees of freedom, the harder it is to successfully run the simulations until the final time. Without interface stabilisation, the highest polynomial degree we could choose and still get meaningful results is $N=3$. For $N=3$ it is interesting to note that all proposed split form schemes run for all tested resolutions (up to $256^3$), except for the standard and MO variants, which crash almost immediately. For higher polynomial degrees, the general trend is that low resolutions with $16^3$ or $32^3$ might run, but all schemes crash for resolutions greater than or equal to $64^3$. 

Nevertheless, the choice of polynomial degree $N=3$ and $16^3$ grid cells ($64^3$ degrees of freedom) allows us to investigate the auxiliary conservation properties. The left part of Fig. \ref{fig:cons} shows the time evolution of the total entropy density
\begin{equation}
S = -\frac{\rho s}{\gamma-1},\quad s = \ln(p) - \gamma\ln(\rho),
\end{equation}
for the different schemes. The right part shows the temporal evolution of the total kinetic energy. 
\begin{figure}[!htbp]
\centering
\includegraphics[trim=0 0 0 0,clip,width=0.45\textwidth]{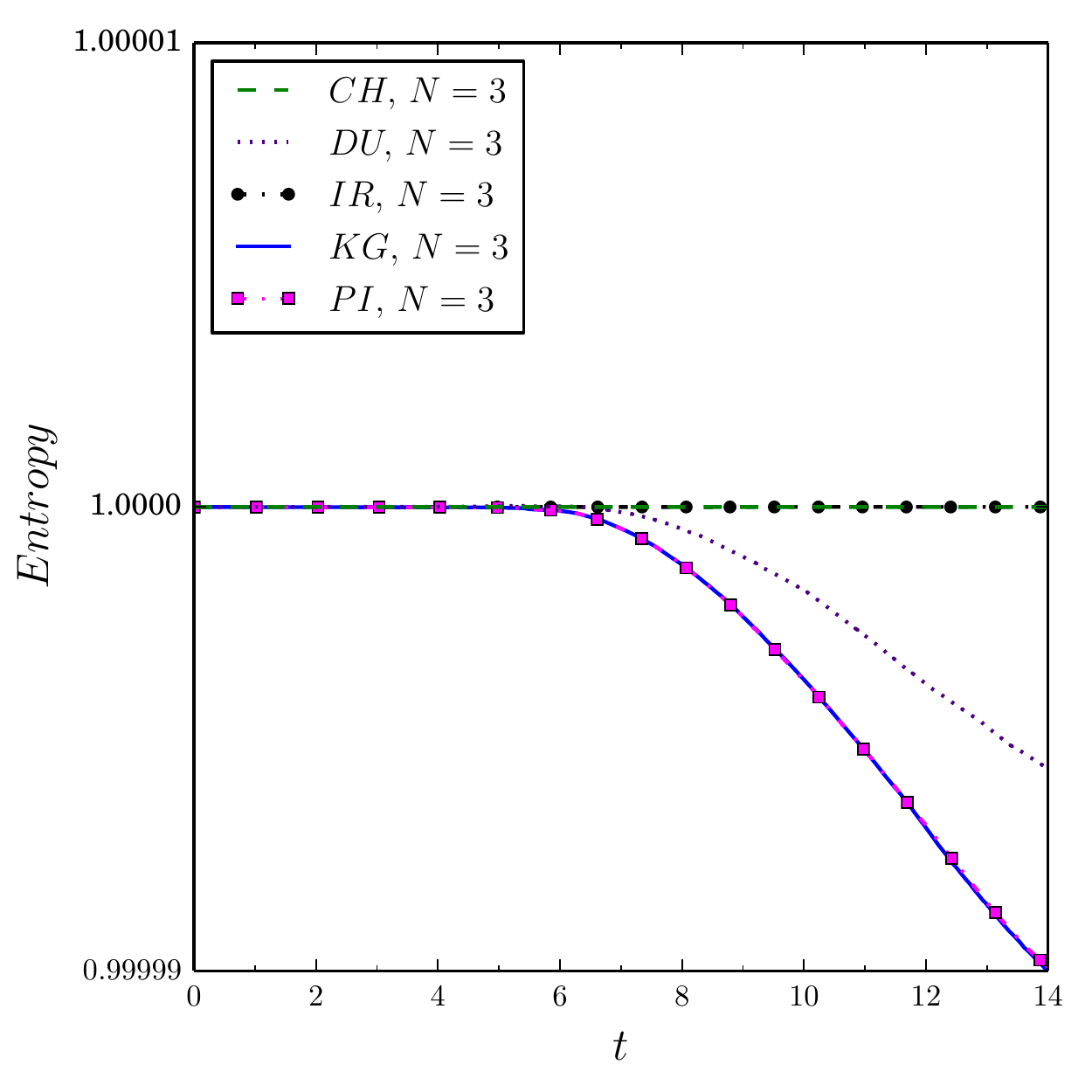}\hspace*{2mm}\includegraphics[trim=0 0 0 0,clip,width=0.45\textwidth]{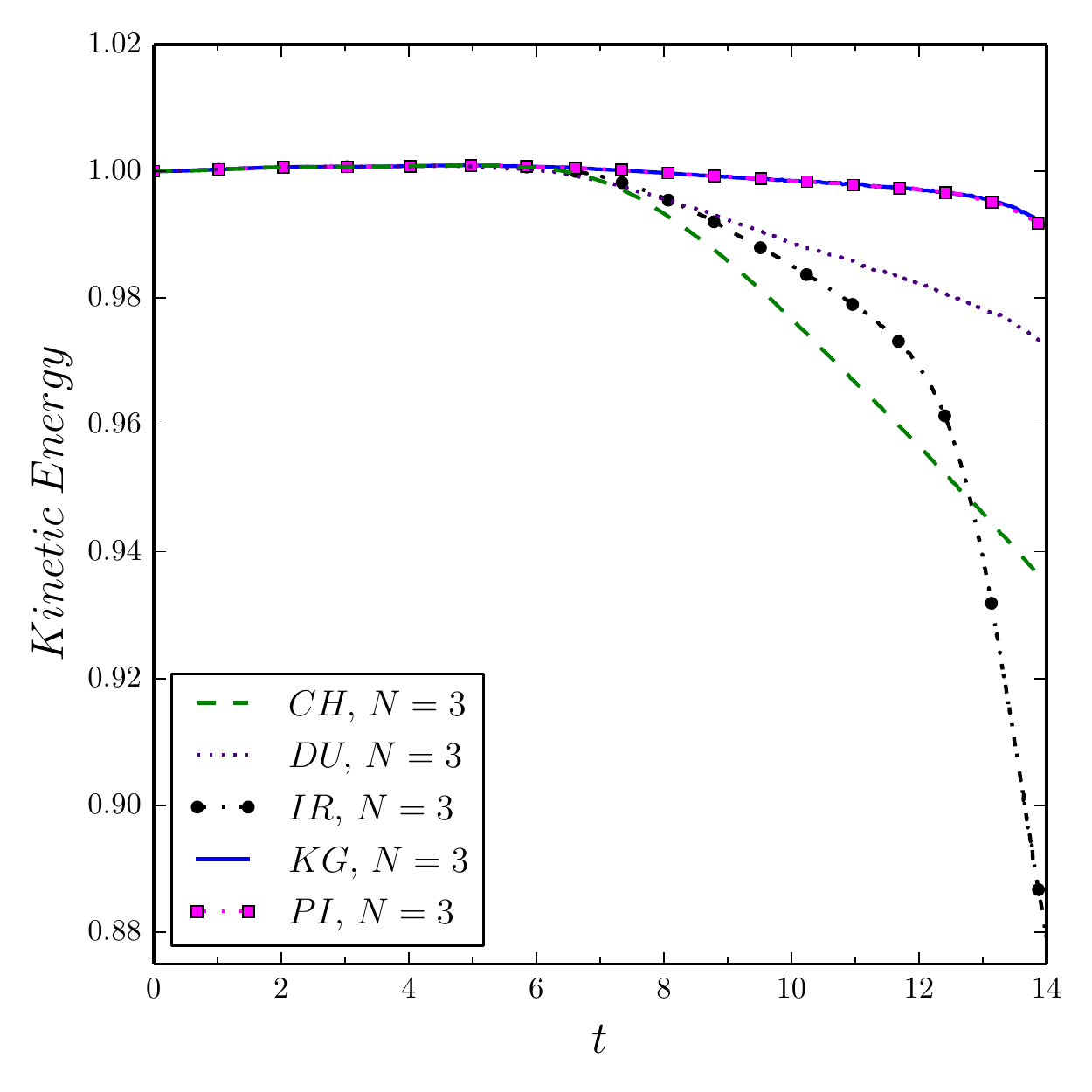}
\caption{\label{fig:cons}Time evolution of the discrete total entropy and total kinetic energy for the case $N=3$ and $16^3$ grid cells. All results are obtained with $CFL=0.1$.}
\end{figure}

Focusing first on the entropy conservation, as expected the IR and CH schemes conserve entropy basically to machine precision, whereas the other schemes show a small decay of the total entropy. However, the $y-$axis shows that the entropy dissipation of the other formulations is very small and arguably negligible.

The entropy conservation results lie in stark contrast to the kinetic energy preservation where substantial differences are observed. Note, that we have three schemes KG, PI and CH which, by construction, preserve the kinetic energy in the sense that the discrete pressure work changes the total kinetic energy, not the advective terms. We can see in the right part of Fig. \ref{fig:cons} that KG and PI give similar results and indeed best preserve the total kinetic energy. Especially in comparison to the other discretisations, which show a loss of up to 10\% of total kinetic energy. The IR scheme performs the worst.

There are two noteworthy remarks. First, the DU scheme conserves neither entropy nor preserves kinetic energy. However, this scheme has a lower loss of total kinetic energy compared to the entropy conserving schemes IR and CH. Second, although the CH scheme is formally entropy conserving \textit{and} kinetic energy preserving, the results clearly show that the loss of kinetic energy is much larger than for the KG and PI schemes. This is an unexpected result and demands further investigation. It turns out that although the advective terms do not formally contribute to the evolution of the kinetic energy as desired, the discretisation of the pressure plays a crucial role for the kinetic energy. Both, KG and PI use a simple arithmetic mean for the pressure and it turns out that this seems to be important for the kinetic energy evolution.

The CH scheme chooses the discretisation of the pressure in such a way that entropy is discretely conserved. The pressure discretisation is much more complicated than the simple arithmetic mean and a possible explanation is that this introduces additional errors that affect the discrete total kinetic energy. To support this claim, we modify the CH scheme by changing the pressure discretisation to the simple arithmetic mean.  The right part of Fig. \ref{fig:keep} shows the evolution of the discrete total kinetic energy. The CH scheme with the pressure fix now behaves like the KG and PI scheme and significantly reduces the loss of kinetic energy in comparison to the original CH scheme. However, the CH scheme with pressure fix is, of course, not entropy conservative anymore as shown in the left part of Fig. \ref{fig:keep}. 
\begin{figure}[!htbp]
\centering
\includegraphics[trim=0 0 0 0,clip,width=0.45\textwidth]{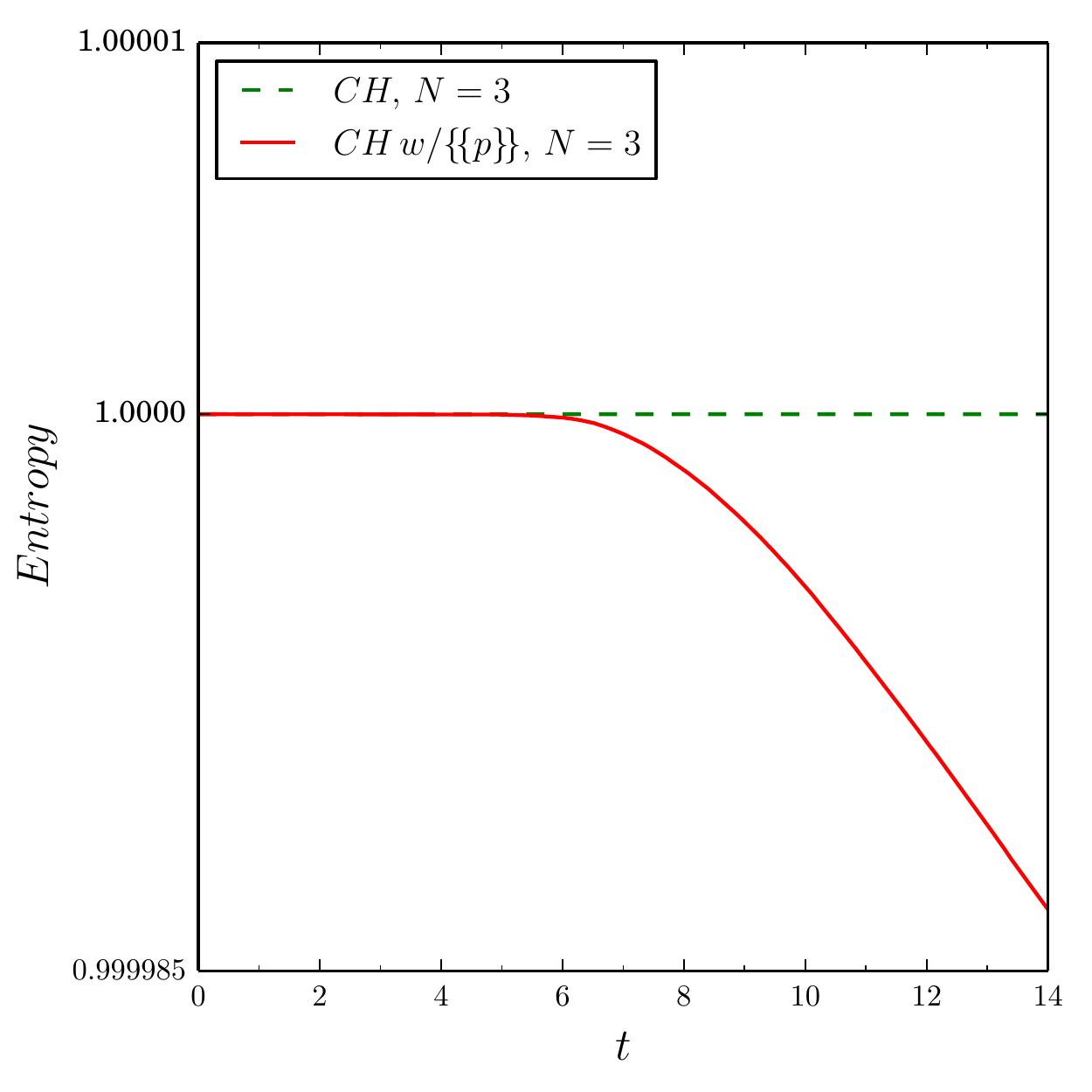}\hspace*{2mm}\includegraphics[trim=0 0 0 0,clip,width=0.45\textwidth]{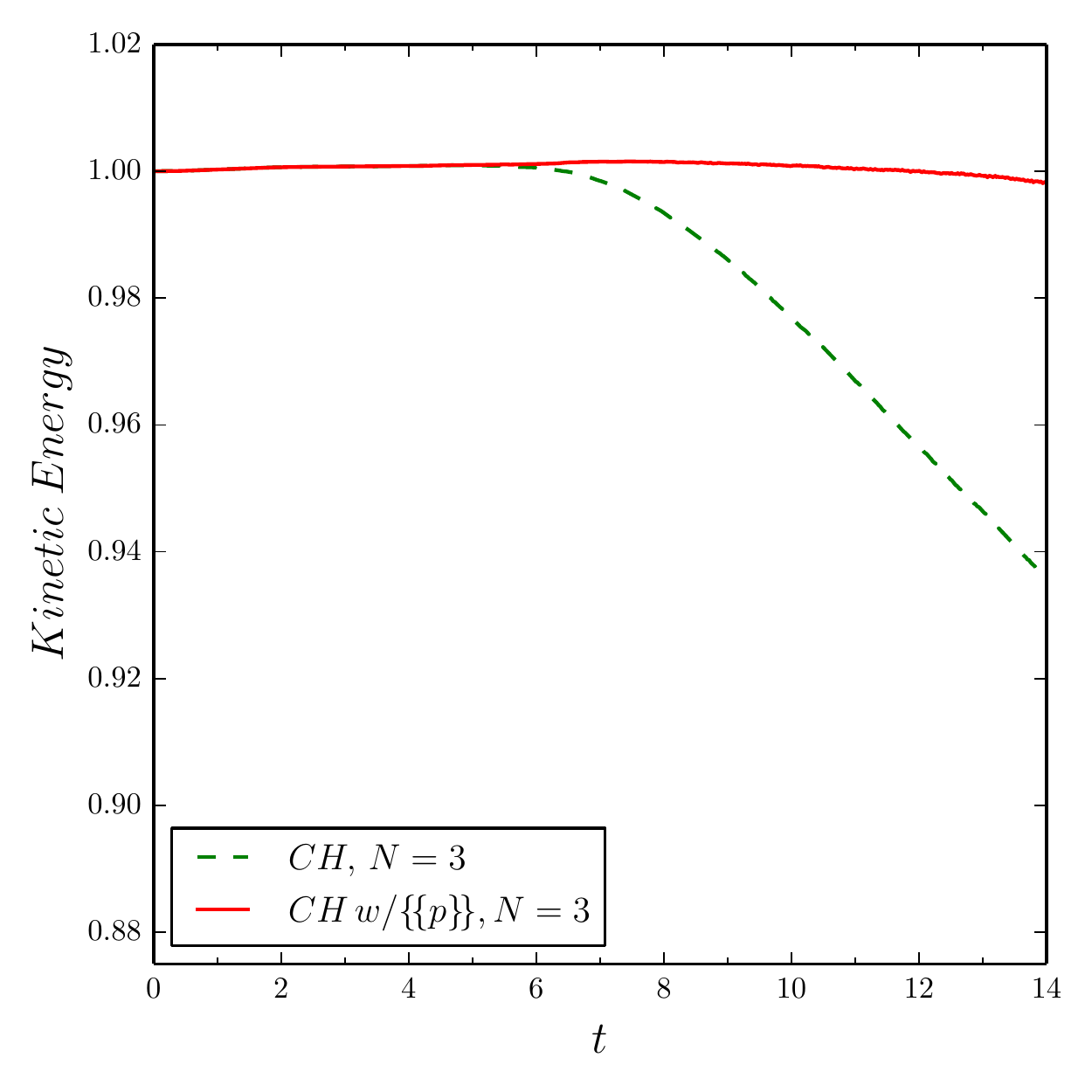}
\caption{\label{fig:keep}Comparison of original CH scheme and the CH scheme where the pressure discretisation is changed to a simple arithmetic mean, as in the KG and PI variants. The plot shows the time evolution of the discrete total entropy (left) and the discrete total kinetic energy (right) for the case $N=3$ and $16^3$ grid cells. All results are obtained with $CFL=0.1$.}
\end{figure}

Summarising this section, the IR and CH schemes show machine precision entropy conservation as expected. The KG and PI schemes show near discrete kinetic energy \textred{conservation for the inviscid Taylor-Green vortex}. The CH scheme shows discrete kinetic energy conservation only with a modification of the pressure discretisation\textred{, which demonstrates that for actual kinetic energy preservation the discretisation of the pressure has a strong impact.}  While the DU scheme does not possess any specific auxiliary conservation properties, it is still worth pointing out that it lost less of the total kinetic energy compared to the entropy conserving schemes. 

\subsection{Robustness}

In this main results section, the robustness of the schemes is investigated. In contrast to the former investigations, we now activate the interface stabilisation. The interface stabilisation terms depend on the jump of the solution at the interface and introduce dissipation. The main result of the numerical investigations is that with the stabilisation terms, all split schemes except the MO variant are robust in the sense that even for very high polynomial degrees the simulations do not crash even with severe under-resolution. 

To make our point, we mimic a table presented in Moura et al. \cite{rodrigo_iLES}, where the stability of an over-integrated DG method for the inviscid Taylor-Green vortex is reported. It is important to note that for high resolutions, the over-integrated DG discretisations crashed and it was concluded that the dissipation from the surface integrals is not enough to stabilise the discretisation. In contrast to the over-integration results, all simulations performed with the DGSEM split forms listed in Tbl. \ref{tab:TGVStability} for the schemes KG, IR, PI, CH, and DU successfully run up to the final time of $t=14$. Note that the configurations marked with a box are the ones that run to the final time when over-integration is used \cite{rodrigo_iLES}. All other configurations could not be successfully completed, even when decreasing the time step and increasing the number of integration nodes up to a factor of four in each spatial direction [personal communication with Rodrigo Moura]. Thus, the entropy stable schemes and the new split form schemes offer a substantial advantage with respect to robustness compared to fully integrated DG.
\begin{table}[!ht]
\begin{center}
  \begin{tabular}{l|ccccc}
 \hline
 DOFs/$N$ & $3$ & $4$& $5$& $6$& $7$ \\[0.0cm]\hline
 $113^3$ & {\boxed{$28$}} & {\boxed{$23$}} & {{$19$}} & {{$16$}} & {{$14$}} \\[0.0cm]\hline
$159^3$ & {\boxed{$39$}} & {\boxed{$32$}} & {{$28$}} & {{$23$}} & {{$19$}} \\[0.0cm]\hline
$227^3$ & {\boxed{$56$}} & {\boxed{$45$}} & {{$39$}} & {{$32$}} & {{$28$}} \\[0.0cm]\hline
$320^3$ & {{$80$}}    & {{$64$}}     & {{$53$}} & {{$46$}} & {{$40$}} \\[0.1cm]\hline
$450^3$ & {{$113$}} & {{$90$}}      & {{$75$}} & {{$64$}} & {{$56$}} \\
 \toprule
  \end{tabular}
  \caption{List of grid cell and polynomial degree configurations for KG, PI, IR, DU, CH schemes that were investigated. \textred{The entries refer to the numbers of elements in one direction.} None of these schemes crashed. The number of DOFs was chosen to keep a factor of $\sqrt{2}$ between consecutive DOFs while allowing for (approximately) integer numbers of elements across the range of polynomial degrees considered. The configurations marked by a box are the ones that are stable with over-integration, whereas the other configuration could not be successfully finished with over-integration alone \cite{rodrigo_iLES}.}
  \label{tab:TGVStability}
\end{center}
\end{table}

To further assess the behaviour of the schemes for the inviscid Taylor-Green vortex, we investigate the evolution of the total kinetic energy dissipation rate, $-d\,\kappa/dt$, and the enstrophy
\begin{equation}
\sigma : = \frac{1}{|\Omega|}\int\limits_{\Omega}\frac{\rho}{2}\omega\cdot\omega\,d\Omega,
\end{equation}
where $\omega$ is the vorticity vector. In the left part of Fig. \ref{fig:dkdt_N3}, the resulting dissipation rate for the configuration $N=3$ with $16^3$ grid cells ($64^3$ degrees of freedom) for all schemes is plotted. Although there is a difference between the split form variants, it is fairly small. In the right part of Fig. \ref{fig:dkdt_N3} we focus on the KG variant only and increase the number of degrees of freedom to compare the $64^3$ results to the configuration with $32^3$ grid cells. 
\begin{figure}[!htbp]
\centering
\includegraphics[trim=0 0 0 0,clip,width=0.45\textwidth]{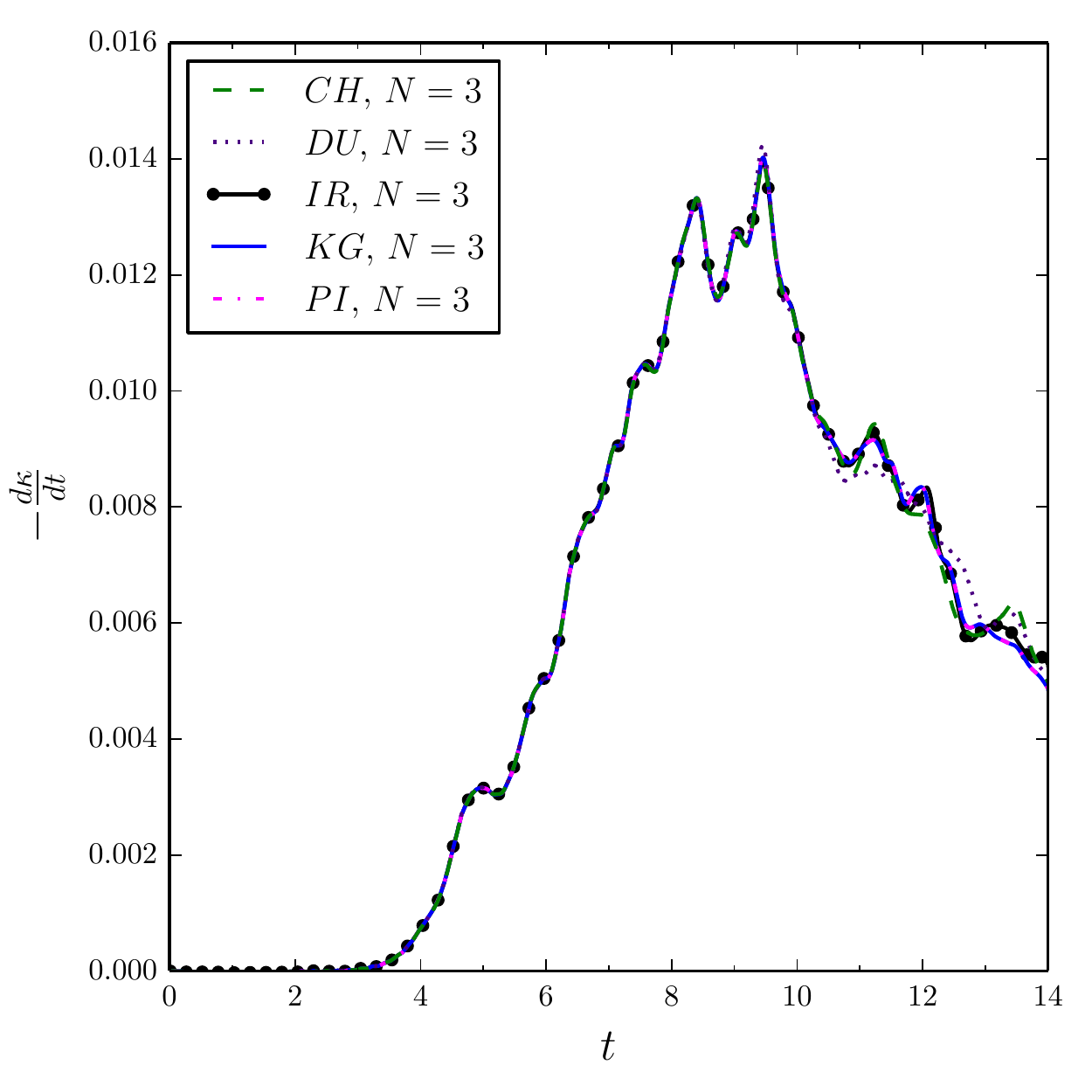}\hspace*{2mm}\includegraphics[trim=0 0 0 0,clip,width=0.45\textwidth]{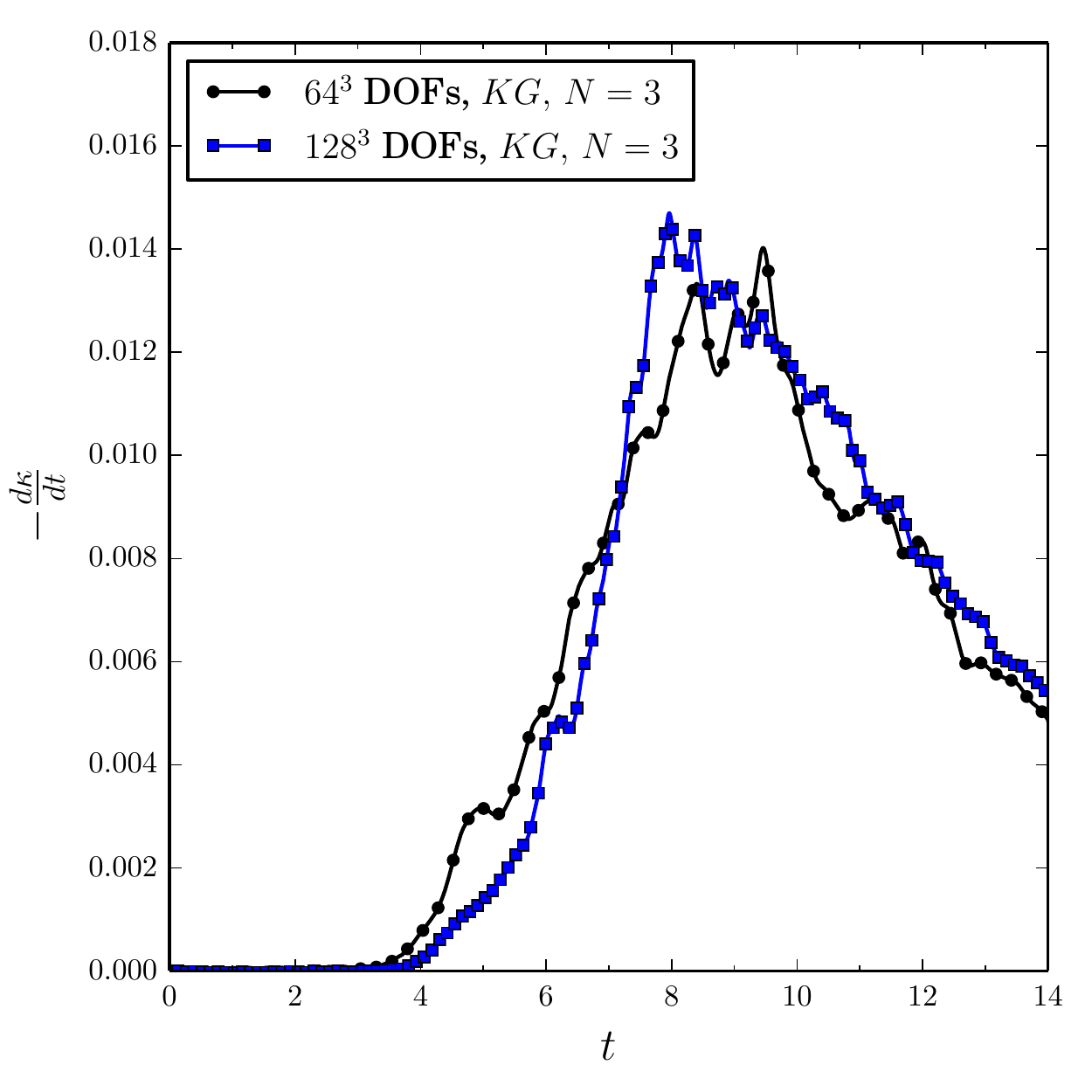}
\caption{\label{fig:dkdt_N3} Plot of the total kinetic energy dissipation rate. Left: $N=3$ with $64^3$ degrees of freedom for all stable split forms. Right: KG scheme with $N=3$ and $16^3$ and $32^3$ grid cells respectively.}
\end{figure}
It is interesting to note that again, the dissipation rate does not change significantly.

It seems that the dissipation introduced by the numerical surface flux function has a certain limit \textblue{ on the dissipation of the integrated kinetic energy}. This is further supported by the results in Fig. \ref{fig:dkdt_N7_15}, where for configurations with fixed $64^3$ degrees of freedom the polynomial degrees are increased to $N=7$ and $N=15$. Note, that the dissipation rates do not significantly change. 
\begin{figure}[!htbp]
\centering
\includegraphics[trim=0 0 0 0,clip,width=0.45\textwidth]{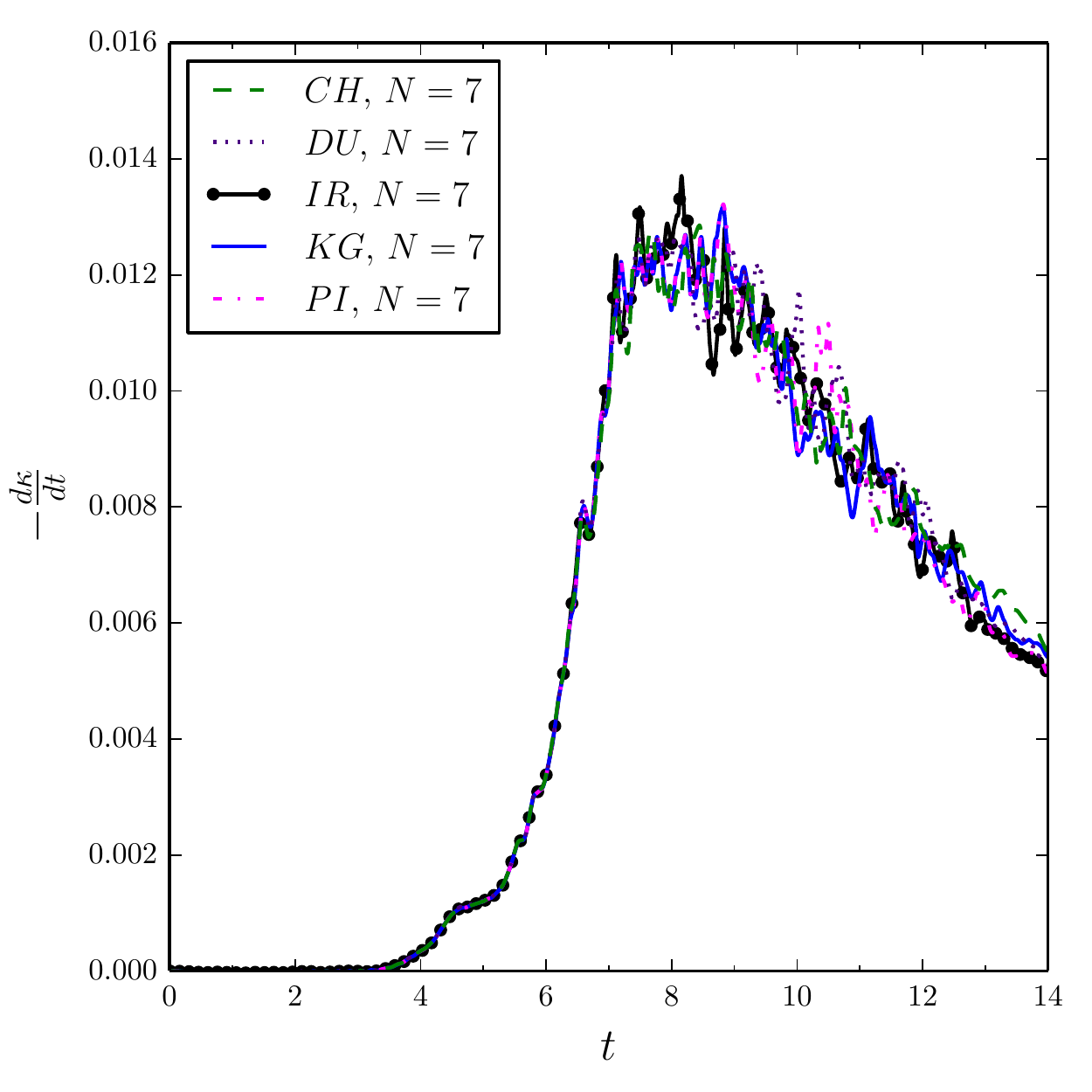}\hspace*{2mm}\includegraphics[trim=0 0 0 0,clip,width=0.45\textwidth]{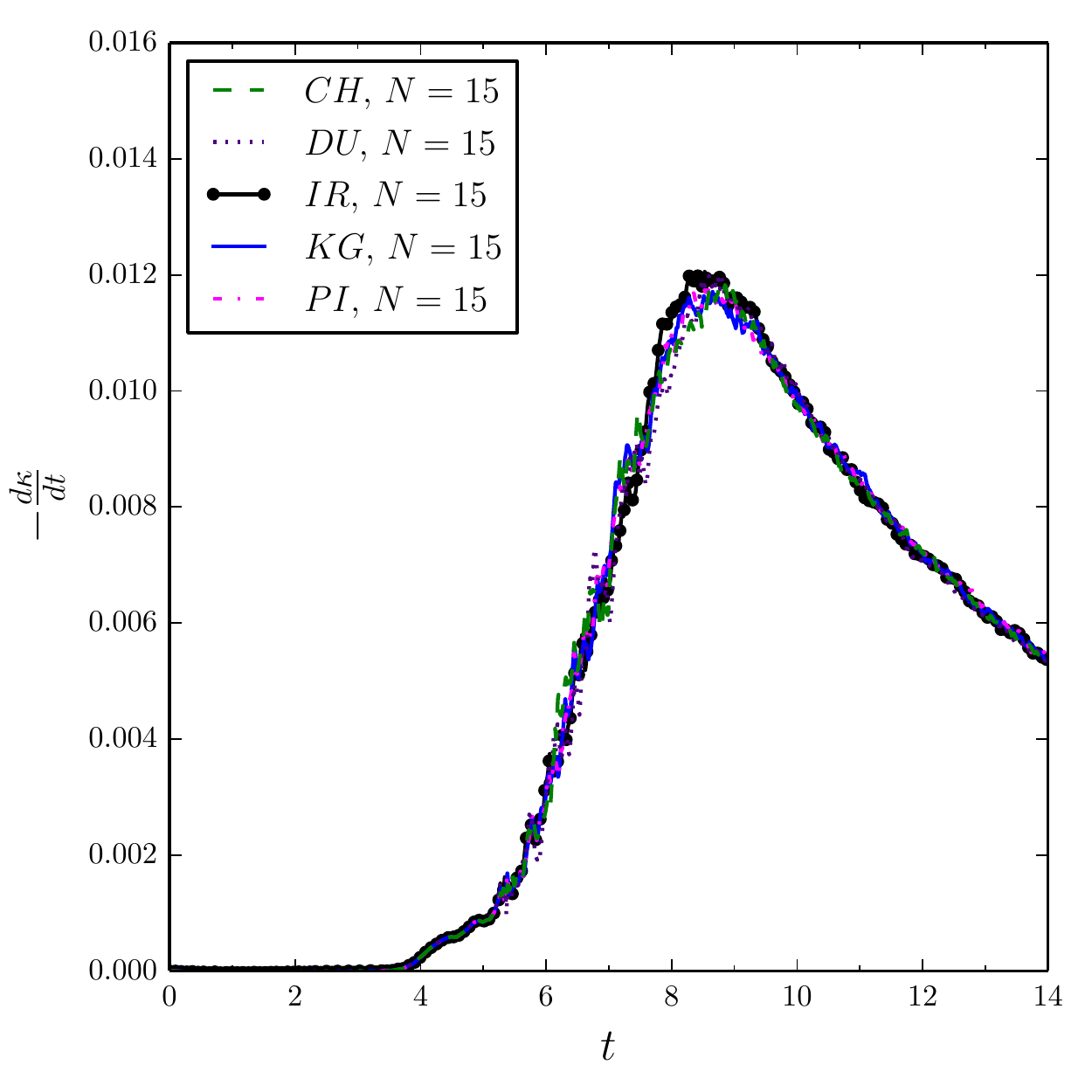}
\caption{\label{fig:dkdt_N7_15} Plot of the total kinetic energy dissipation rate. Left: $N=7$ with $8^3$ grid cells for all stable split forms. Right: $N=15$ with $4^3$ grid cells for all stable split forms.}
\end{figure}

In contrast to the behaviour of the dissipation rates, the evolution of the enstrophy strongly depends on the spatial resolution and on the polynomial degree. Figure \ref{fig:enstrophy} shows the enstrophy as a function of time for a fixed $64^3$ degrees of freedom in the left part. The difference between the split forms is again small, however the impact of the polynomial degree increases the magnitude of the total of enstrophy by an order of magnitude. 
The right part of Fig. \ref{fig:enstrophy} considers again the KG scheme with $N=3$ and $16^3$ and $32^3$ grid cells, respectively. By comparing the left and right plot, it is interesting to note that the enstrophy of the configuration with $N=7$ and $8^3$ grid cells ($64^3$ degrees of freedom) has a higher maximum than the configuration with $N=3$ and $32^3$ grid cells ($128^3$ degrees of freedom). 
\begin{figure}[!htbp]
\centering
\includegraphics[trim=0 0 0 0,clip,width=0.45\textwidth]{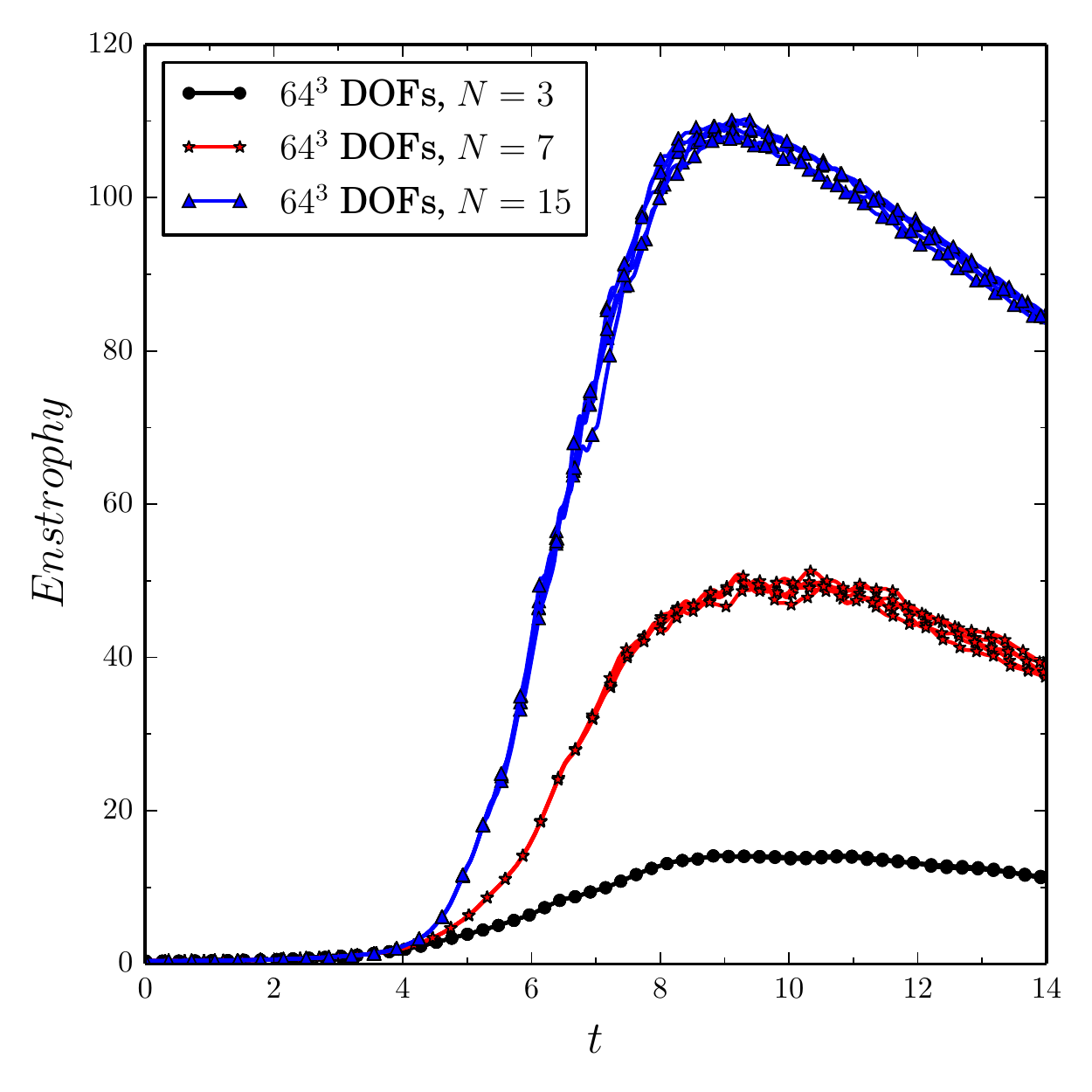}\hspace*{2mm}\includegraphics[trim=0 0 0 0,clip,width=0.45\textwidth]{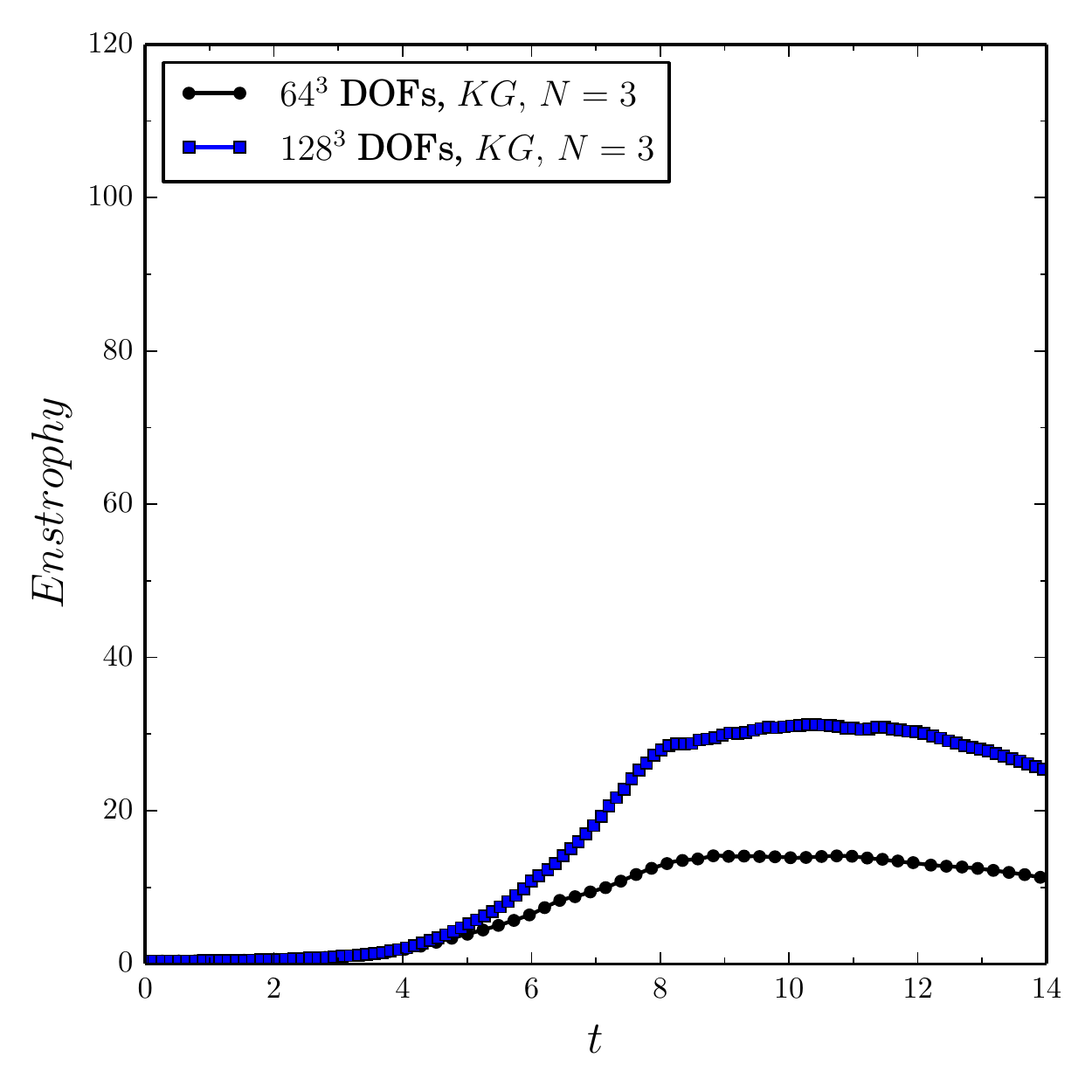}
\caption{\label{fig:enstrophy} Temporal evolution of the enstrophy. Left: Fixed $64^3$ degrees of freedom for all stable split forms for configurations with polynomial degree $N=3,\,7,\,15$. Right: Enstrophy of the KG scheme with $N=3$ and $16^3$ and $32^3$ grid cells respectively.}
\end{figure}

The significant impact of the polynomial degree on the evolution of the enstrophy has a direct consequence for the estimated numerical viscosity of the discretisations. For solutions of the incompressible viscous Taylor-Green vortex, the dissipation rate is directly linked to the enstrophy via the physical viscosity $\mu$ \cite{tcfd2012}
\begin{equation}
-\frac{d\,\kappa}{dt} = 2\,\mu\,\sigma.
\end{equation}
We exploit this relationship to find an estimate for the viscosity introduced by the numerical discretisation. By relating the discrete dissipation rate and the enstrophy over time, we get an evolution of the numerical dissipation of the scheme
\begin{equation}
\mu_{num} \approx \frac{-\frac{d\,\kappa}{dt}}{2\,\sigma}.
\end{equation} 
The numerical viscosity estimate is plotted in the left part of Fig. \ref{fig:numvisc}, where again the degrees of freedom are fixed to $64^3$. Due to the significantly higher enstrophy for the high polynomial degree discretisations, the estimate of the numerical viscosity is much lower. This fits to the apparent  ``high-order schemes have lower dissipation" paradigm. However, it is interesting to note that the actual dissipation rate of the kinetic energy does not vary much. In the right part of Fig. \ref{fig:numvisc},  we again compare the KG scheme for the $N=3$ and $16^3$ configuration. 
\begin{figure}[!htbp]
\centering
\includegraphics[trim=0 0 0 0,clip,width=0.45\textwidth]{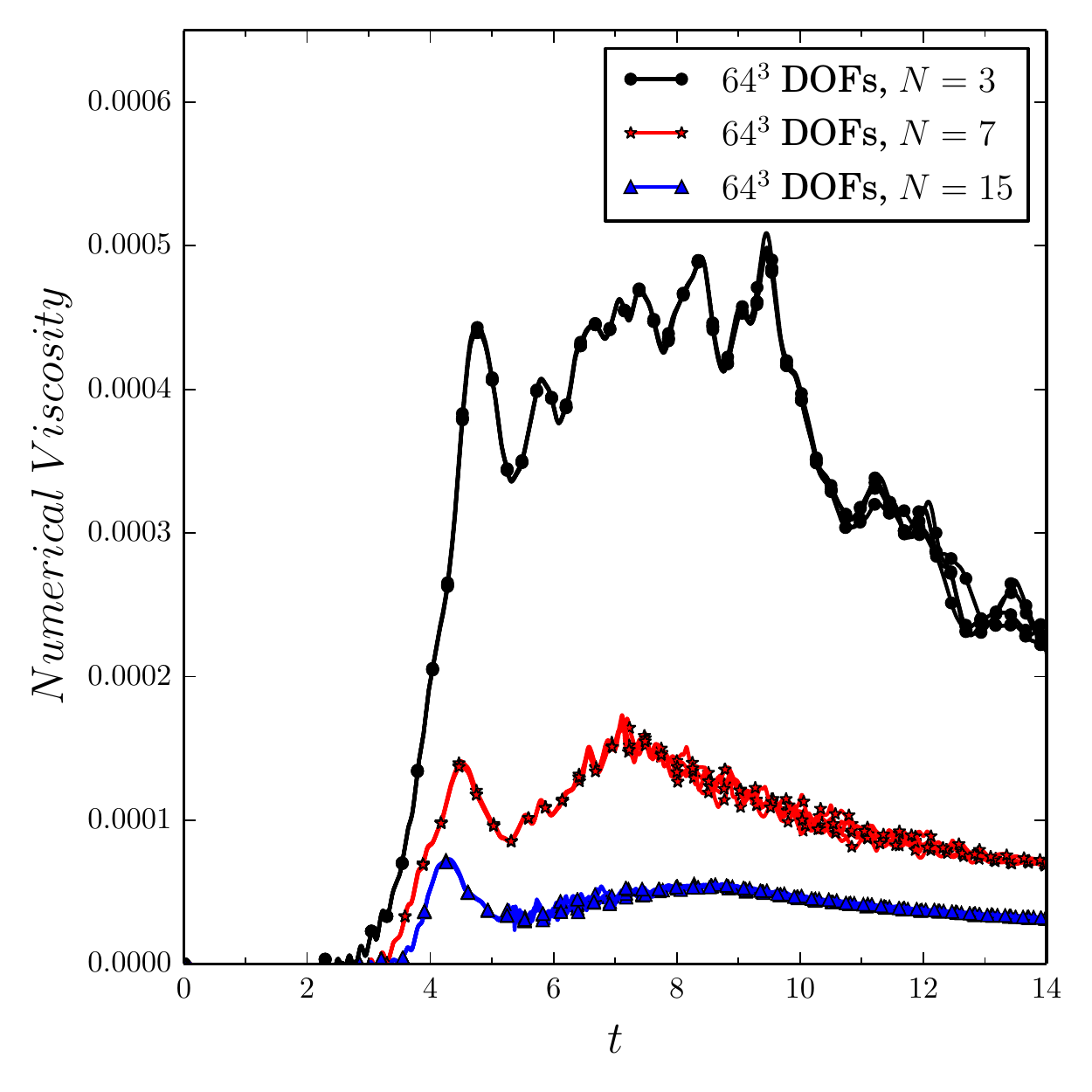}\hspace*{2mm}\includegraphics[trim=0 0 0 0,clip,width=0.45\textwidth]{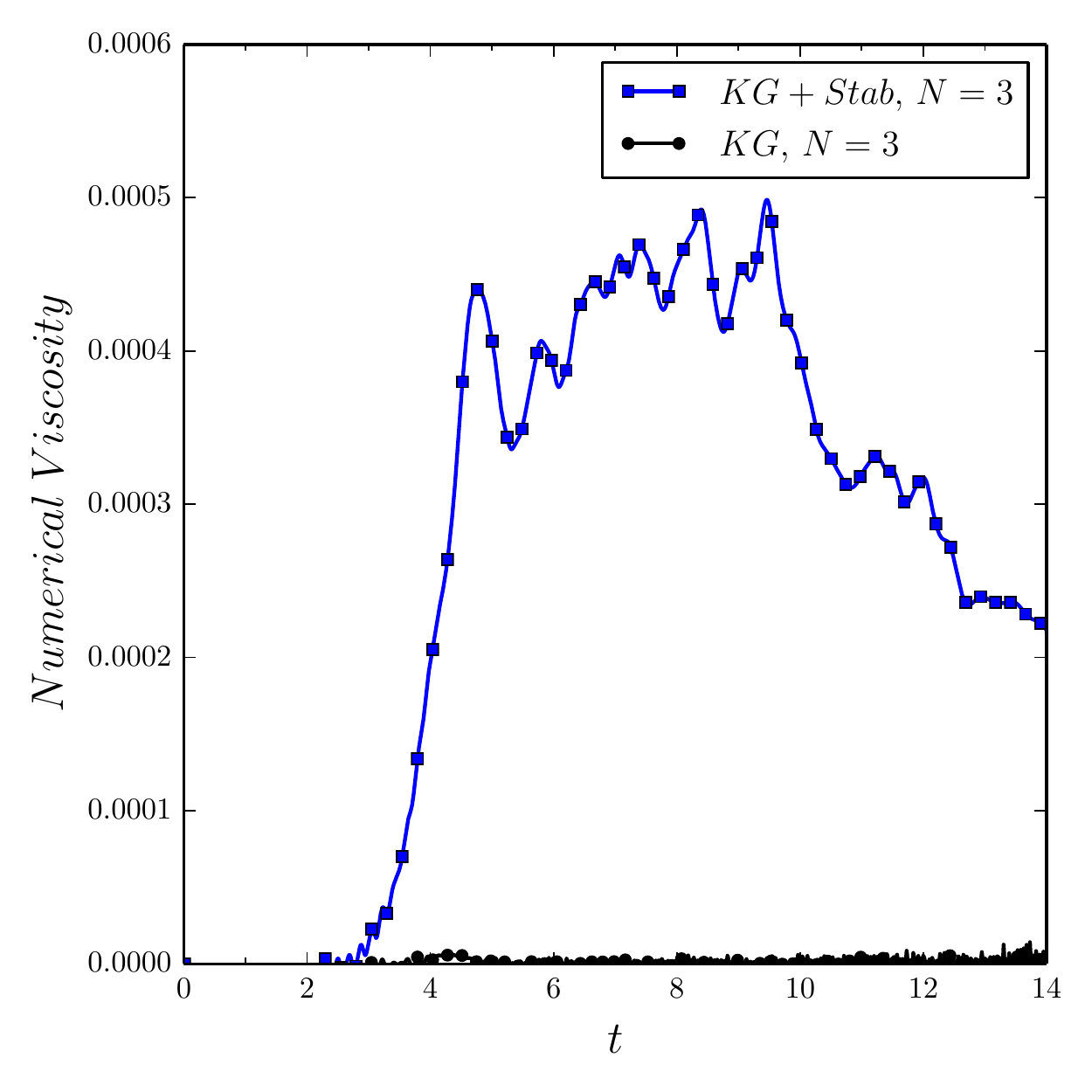}
\caption{\label{fig:numvisc} Plot of the numerical viscosity over time. Left: Fixed $64^3$ degrees of freedom for all stable split forms for configurations with polynomial degree $N=3,\,7,\,15$. Right: Comparison of KG results with and without stabilisation for $N=3$ and $16^3$ grid cells.}
\end{figure}
We directly compare the impact of the interface stabilisation terms on the numerical viscosity and plot results from the simulation with interface stabilisation terms and without. It can be clearly observed that the no stabilisation run has almost zero numerical viscosity, while being stable for this severely under-resolved test case. 

At the end of this section, we note that the variant KG (and PI) with cubic split forms was introduced by Kennedy and Gruber \cite{kennedy2008} to account for large density variations. Thus, it is interesting to investigate the difference of DU and KG/PI for a test configuration with higher compressibility. To do so we change the initial pressure for inviscid the Taylor-Green vortex so that the Mach number increases to $Ma=0.4$. This introduces more compressibility effects and as shown in Fig. \ref{fig:M04_1} and \ref{fig:M04_2}, this causes problems for the DU variant. For the configuration $N=7$ with $16^3$ grid cells, we can observe that DU crashes at $t\approx 10.8$, even with very small CFL numbers. The variants KG, PI and IR, CH all run until the final time, demonstrating increased robustness of those variants for problems with higher density variations. 
\begin{figure}[!htbp]
\centering
\includegraphics[trim=0 0 0 0,clip,width=0.45\textwidth]{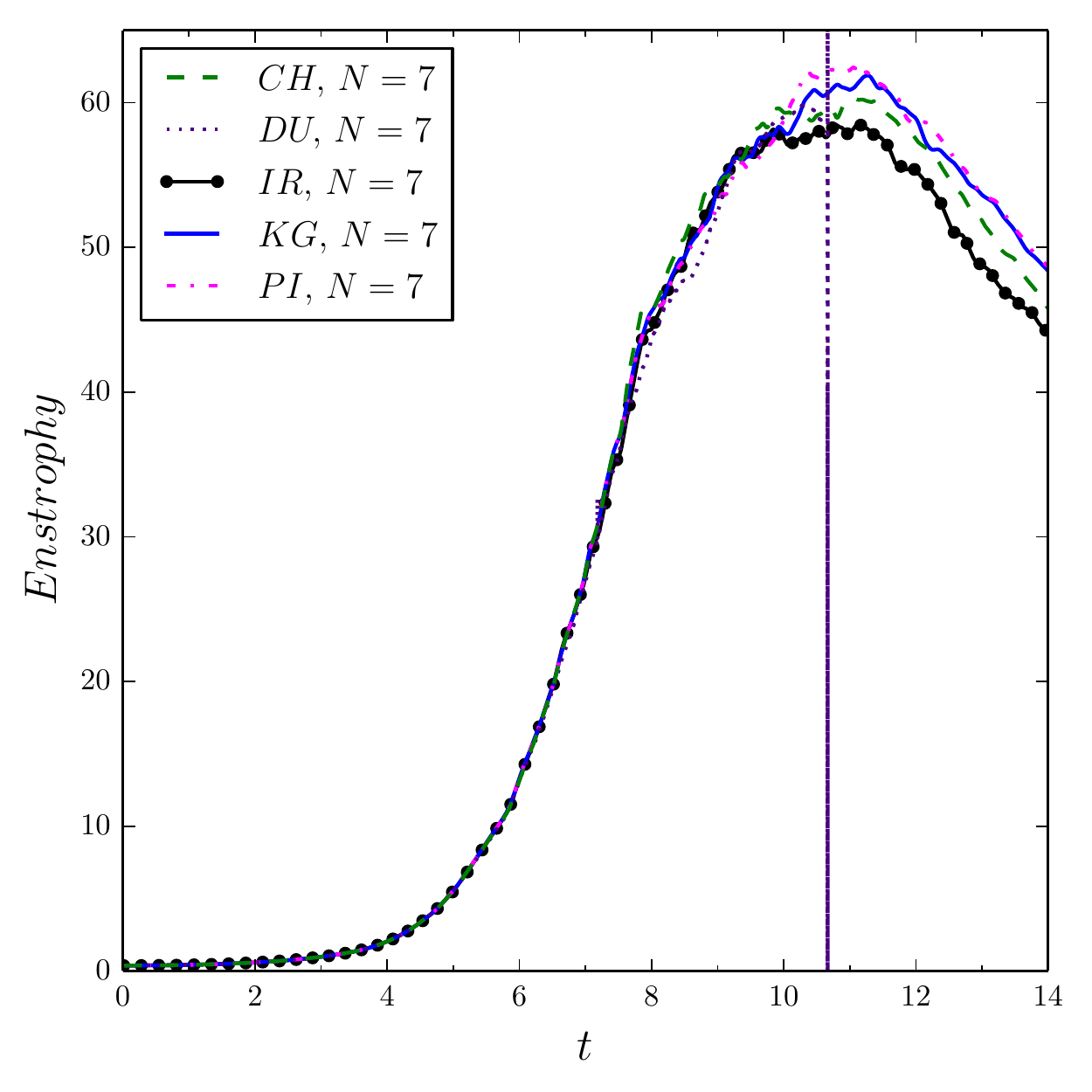}\hspace*{2mm}\includegraphics[trim=0 0 0 0,clip,width=0.45\textwidth]{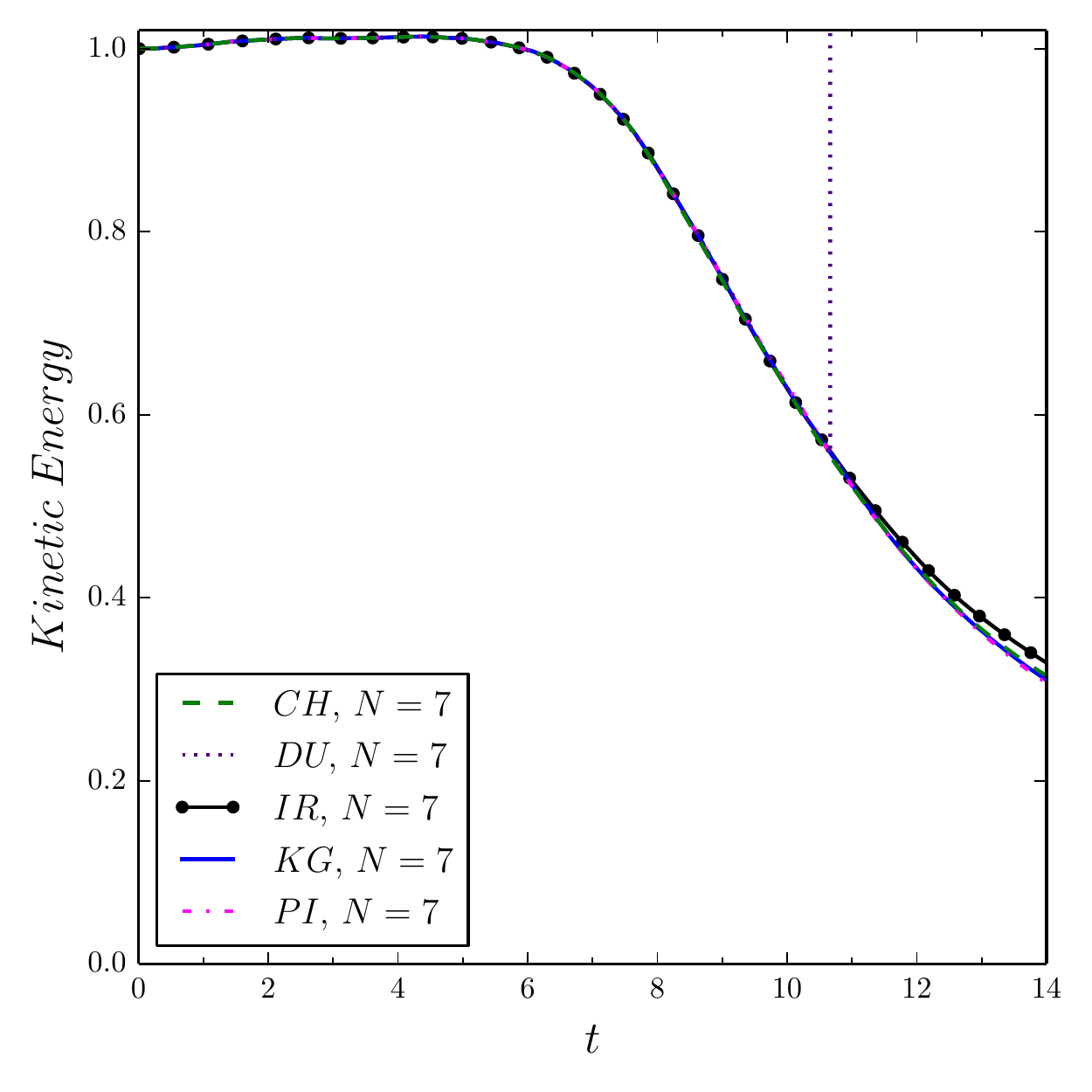}
\caption{\label{fig:M04_1} Results of simulation with $16^3$ grid cells and $N=7$ ($128^3$ degrees of freedom) for the Taylor-Green vortex with a Mach number of $Ma=0.4$. Left: Plot of total enstrophy. Right: Plot of total kinetic energy. Note that DU crashes at $t\approx 10.8$, indicated by an extreme spike of the quantities.}
\end{figure}

\begin{figure}[!htbp]
\centering
\includegraphics[trim=0 0 0 0,clip,width=0.45\textwidth]{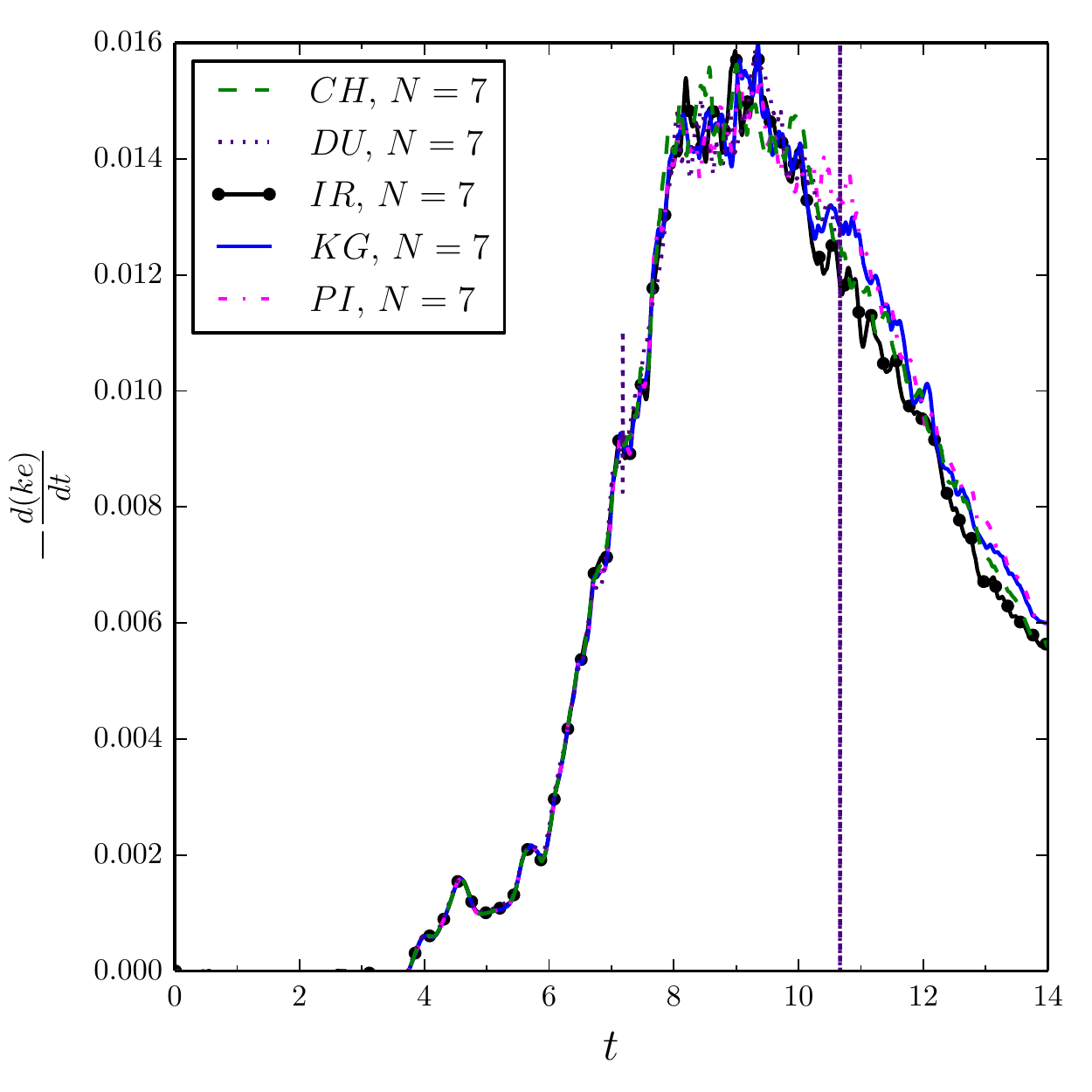}\hspace*{2mm}\includegraphics[trim=0 0 0 0,clip,width=0.45\textwidth]{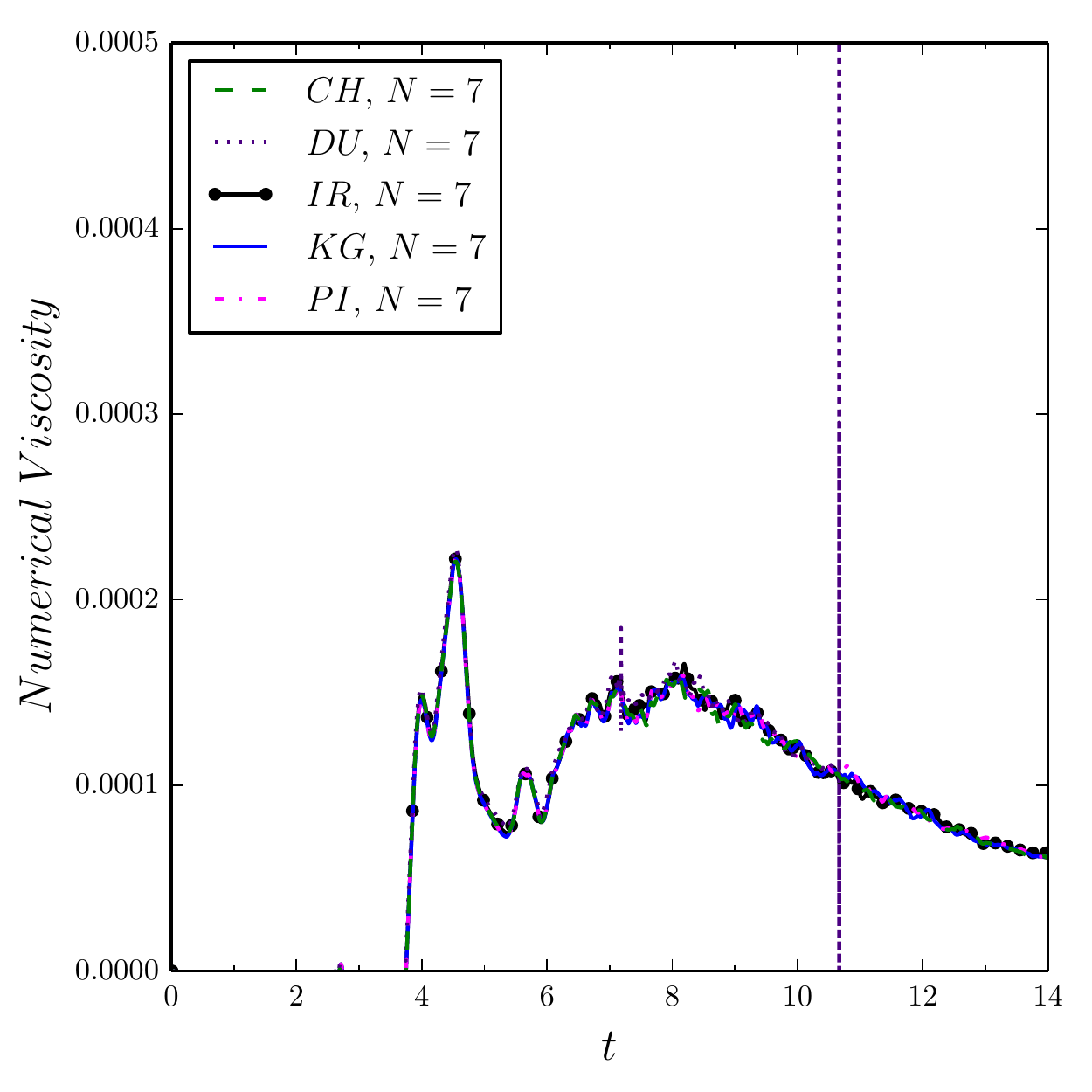}
\caption{\label{fig:M04_2} Results of simulation with $16^3$ grid cells and $N=7$ ($128^3$ degrees of freedom) for the Taylor-Green vortex with a Mach number of $Ma=0.4$. Left: Plot of dissipation rate of total kinetic energy. Right: Plot of Numerical viscosity estimation. Note that DU crashes at $t\approx 10.8$, indicated by an extreme spike of the quantities.}
\end{figure}

Summarising the investigations of this section, the most important result is that all split form schemes, except the variant MO,  significantly improve the robustness of high-order DG discretisations compared to polynomial de-aliasing with over-integration. We could demonstrate that the cubic split forms are more robust in comparison to the quadratic splitting when the flow is compressible. Furthermore, the influence of the numerical viscosity can be directly observed and tracked by comparison to nearly dissipation free variants of the scheme. This seems to be a beneficial and interesting feature for future investigations on turbulence modelling.

\section{\textblue{Final Remarks}}\label{sec:conclusion}

\subsection{\textblue{Discussion}}

As a final discussion, the identities in Lem. \ref{Lem} form the basis for a dictionary of a translation from split forms into flux forms and vice versa. We used these to translate known alternative formulations of the non-linear compressible Euler terms. However, this dictionary can be used to generate a multitude of new split forms by choosing symmetric and consistent flux approximations. To demonstrate this, we consider an example based on the Roe variables
\begin{equation}
\begin{split}
q_1 &= \sqrt{\rho},\\
q_2 &= \sqrt{\rho}\,u,\\
q_3 &= \sqrt{\rho}\,v,\\
q_4 &= \sqrt{\rho}\,w,\\
q_5 &= \sqrt{\rho}\,h.
\end{split}
\end{equation}
This specific set of variables is interesting, as it reformulates the flux into only quadratic products, e.g. the flux in the $x-$direction is
\begin{equation}
F(U) = 
\begin{pmatrix}
q_1\,q_2\\
q_2^2 + \frac{\gamma -1}{\gamma}\left(q_1\,q_5 - \frac{1}{2}\left(q_2^2+q_3^2+q_4^2\right) \right)\\
q_2\,q_3\\
q_2\,q_4\\
q_2\,q_5
\end{pmatrix}.
\end{equation} 
Assuming now a quadratic splitting of the products, we can directly translate this into an equivalent numerical volume flux $QU$, based on the quadratic identities of Lem. \ref{Lem}
\begin{equation}
F^{\#}_{QU}(U_{ijk},U_{mjk}) = 
\begin{pmatrix}
\average{q_1}\,\average{q_2}\\
\average{q_2}^2 + \frac{\gamma -1}{\gamma}\left(\average{q_1}\,\average{q_5} - \frac{1}{2}\left(\average{q_2}^2+\average{q_3}^2+\average{q_4}^2\right) \right)\\
\average{q_2}\,\average{q_3}\\
\average{q_2}\,\average{q_4}\\
\average{q_2}\,\average{q_5}
\end{pmatrix},
\end{equation} 
and thus, generate a new scheme. In Fig. \ref{fig:finale} we provide a comparison of the dissipation rates of the new scheme against the previous results.
\begin{figure}[!htbp]
\centering
\includegraphics[trim=0 0 0 0,clip,width=0.45\textwidth]{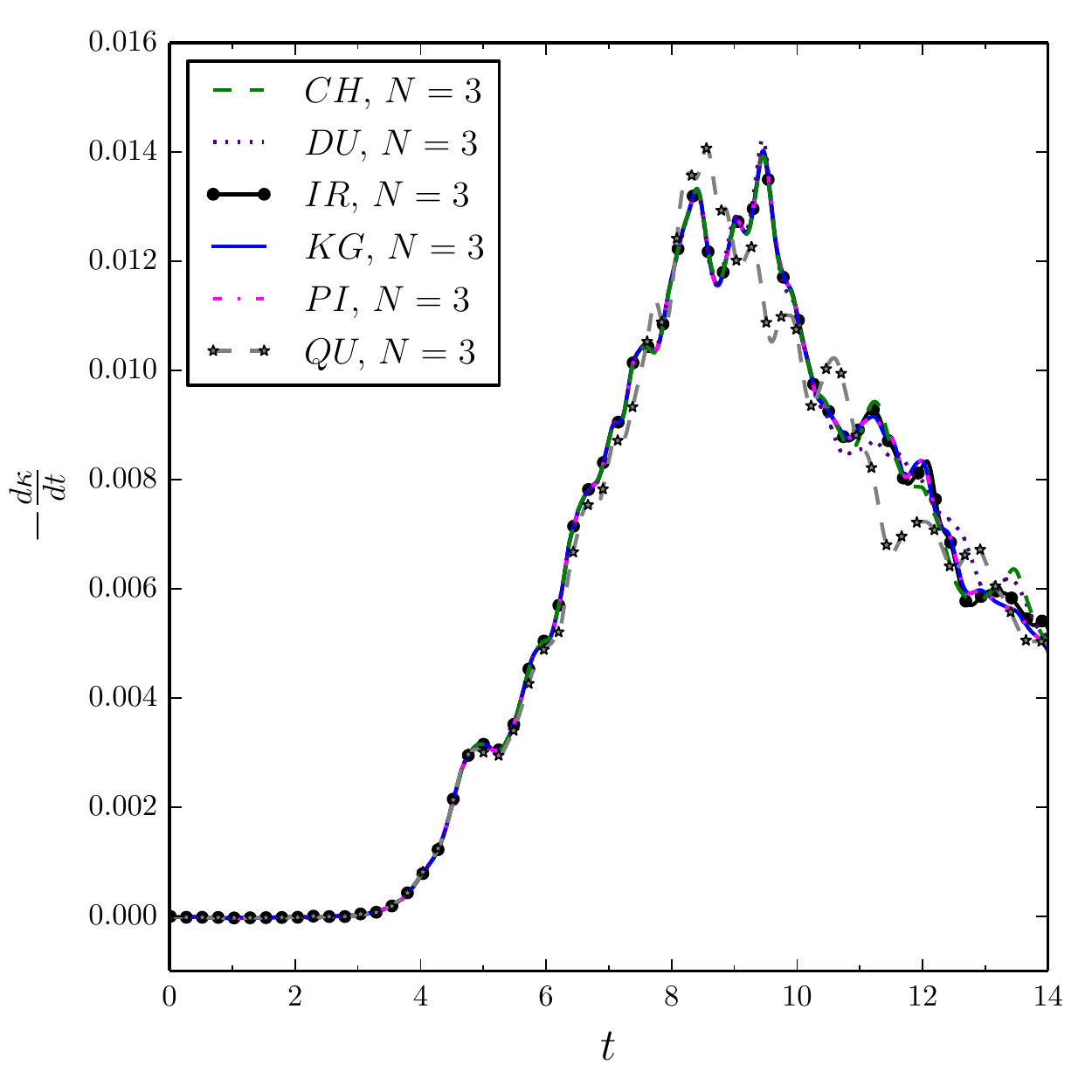}
\caption{\label{fig:finale} Plot of dissipation rates for all variants including the Roe variable based splitting variant $QU$ for $N=3$ and $16^3$ grid cells.}
\end{figure}
With this simple construction and translation guide, it is straightforward to introduce novel splittings, as demonstrated, and perhaps find forms that offer benefits over the ones currently available in the literature.

\subsection{\textblue{Conclusion}}

In the first part of this paper, we exploit the volume flux difference form introduced by Carpenter et al. and Fisher et al. \cite{carpenter_esdg,fisher2013} to construct a unified formulation for split forms of quadratic and cubic products. From a simple relationship between split forms for quadratic and cubic products, we showed that it is possible to directly translate known alternative compressible Euler formulations into the volume flux difference form \eqref{eq:RHSfluxform} by introducing a numerical volume flux $F^\#$, $G^\#$, $H^\#$. It is possible to simply change the volume flux and generate a new DGSEM. We prove that the kinetic energy preserving numerical volume fluxes generate high-order accurate discretisations that also preserve kinetic energy.

Several alternative formulations of the compressible Euler equations from the literature are investigated in the numerical results sections. All schemes show expected convergence behaviour and expected auxiliary conservation properties. An outcome of the numerical assessment is that the new DG schemes are more robust than DG with over-integration. Configurations that crash with over-integration can be successfully finished with the split form DGSEM as well as with the entropy stable DGSEM introduced by Carpenter et al. \cite{carpenter_esdg}. The numerical volume fluxes of the KG, PI and DU variants offer a computational advantage compared to the more complex entropy stable variants IR and CH. We note that the volume flux differencing form is computationally more intense than the standard DGSEM. On the other hand, we found that for the inviscid Taylor-Green vortex, all configurations for the standard DGSEM (without over-integration) immediately crash.

It is also worth noting that not all alternative formulations automatically improve the robustness, even if the split form is kinetic energy preserving, such as the MO variant. Similar to the standard DGSEM, this variant crashes for all configurations after a short simulation time. Whereas the DU variant is as stable as the KG, PI, IR and CH variants for low Mach numbers, for higher Mach numbers (more compressibility effects), DU is less robust. A conclusion so far is that a good balance between robustness and computational effort is offered by the variants KG and PI. However, it is clear that with the newly established framework for de-aliasing split forms of DG for the compressible Euler equations \cite{ducros2000,kennedy2008,FDaliasing,larsson2007} more assessment of the many interesting aspects, such as robustness, accuracy and efficiency is necessary.

Additionally, the numerical viscosity of the different discretisations is investigated and it was possible to directly trace the influence of the stabilisation terms. The key for this is that even without interface stabilisation terms, some configurations of the schemes were stable. This offers interesting possibilities in the control and design of dissipation tailored for turbulence modelling.

\appendix

\section{Proofs}
\subsection{Proof of Lem. \ref{Lem} [Discrete split forms]}\label{LemProof}

\begin{proof}
First we consider and make use of the relation $\dmat=\mmat^{-1}\qmat$ to get
\begin{equation}\label{aAverage}
\begin{aligned}
2\,\sum\limits_{m=0}^N Q_{im}\,\average{a}_{im} &= \sum_{m=0}^N Q_{im} (a_i + a_m),\\
&= a_i\sum_{m=0}^N Q_{im} + \sum_{m=0}^N Q_{im}a_m,\\
&= \sum_{m=0}^N Q_{im}a_m,\\
\end{aligned}
\end{equation}
for $ i = 0,\ldots,N$ and we have used the fact that the rows of $\mat{Q}$ sum to zero \cite{Strand199447}. We compress the result \eqref{aAverage} into the matrix vector product form
\begin{equation}
2\,\sum\limits_{m=0}^N D_{im}\,\average{a}_{im} = (\mat{D}\,\vec{a})_i.
\end{equation}

Next, we prove the identity for the product of two averages
\begin{equation}\label{substSimp}
\begin{aligned}
2\,\sum\limits_{m=0}^N Q_{im}\,\average{a}_{im}\,\average{b}_{im}  &= \half\sum_{m=0}^NQ_{im}(a_i+a_m)(b_i+b_m), \\ 
&=\frac{1}{2}\sum_{m=0}^NQ_{im}a_m\,b_m+\frac{a_i}{2}\sum_{m=0}^NQ_{im}\,b_m + \frac{b_i}{2}\sum_{m=0}^NQ_{im}\,a_m + \frac{a_i\,b_i}{2}\sum_{m=0}^NQ_{im},\\
&=\frac{1}{2}\left\{ \sum_{m=0}^NQ_{im}\,a_m\,b_m\, + a_i\sum_{m=0}^NQ_{im}\,b_m + b_i\sum_{m=0}^NQ_{im}\,a_m\right\},
\end{aligned}
\end{equation}
for $i = 0,\ldots,N$. Again, we have used the property that the sum of each row of the SBP matrix $\mat{Q}$ vanishes. We see that after scaling by $i$-th term of $\mmat^{-1}$ the flux difference result \eqref{substSimp} is the $i^{\textrm{th}}$ row of the split form
\begin{equation}
2\,\sum\limits_{m=0}^N D_{im}\,\average{a}_{im}\,\average{b}_{im}   = \frac{1}{2}\left(\dmat\,\uuline{a}\,\vec{b} +\uuline{a}\,\dmat\,\vec{b} + \uuline{b}\,\dmat\,\vec{a} \right)_i.
\end{equation}

We prove the final identity of \eqref{eq:skewsym_identities} by considering a triple product of averages 
\begin{equation}\label{tripleProdProof}
\resizebox{0.85\textwidth}{!}{$
\begin{aligned}
2\,\sum\limits_{m=0}^N Q_{im}\average{a}_{im}\average{b}_{im}\average{c}_{im} &= \frac{1}{4}\sum_{m=0}^N Q_{im}(a_i+a_m)\,(b_i+b_m)\,(c_i+c_m),\\
&= \frac{1}{4}\sum_{m=0}^N Q_{im}\,\left(a_i b_i c_i + a_i b_i c_m + a_i b_m c_i + a_i b_m c_m \right.\\
&\qquad\qquad\quad\left.+ a_m\, b_i\, c_i+a_m\, b_i\, c_m + a_m\, b_m\, c_i + a_m\, b_m\, c_m\right),\\
&= \frac{a_i\,b_i\,c_i}{4}\sum_{m=0}^N Q_{im} + \frac{a_i\,b_i}{4}\sum_{m=0}^N Q_{im}\,c_m + \frac{a_i\,c_i}{4}\sum_{m=0}^N Q_{im}\,b_m + \frac{a_i}{4}\sum_{m=0}^N Q_{im}\,b_m\,c_m\\
&\qquad+\frac{b_i\,c_i}{4}\sum_{m=0}^N Q_{im}\, a_m+\frac{b_i}{4}\sum_{m=0}^N Q_{im}\,a_m\, c_m + \frac{c_i}{4}\sum_{m=0}^N Q_{im}\,a_m\,b_m + \sum_{m=0}^N Q_{im}\,a_m\,b_m\,c_m,\\
&= \frac{1}{4}\left\{\sum_{m=0}^N Q_{im}\,a_m\,b_m\,c_m+ a_i\sum_{m=0}^N Q_{im}\,b_m\,c_m + {b_i}\sum_{m=0}^N Q_{im}\,a_m\,c_m + {c_i}\sum_{m=0}^N Q_{im}\,a_m\,b_m \right.\\
&\qquad\qquad\left.+{b_i\,c_i}\sum_{m=0}^N Q_{im}\, a_m+ {a_i\,c_i}\sum_{m=0}^N Q_{im}\,b_m+{a_i\,b_i}\sum_{m=0}^N Q_{im}\,c_m\right\},\\
\end{aligned}$}
\end{equation}
for $i = 0,\ldots,N$. The vanishing sum property of the rows of $\mat{Q}$ is, again, applied. We condense the final result of \eqref{tripleProdProof} into matrix-vector form to obtain the final split form identity
\begin{equation}
2\,\sum\limits_{m=0}^N D_{im}\,\average{a}_{im}\,\average{b}_{im}\,\average{c}_{im} = \frac{1}{4}\left(\dmat\,\uuline{a}\,\uuline{b}\,\vec{c} + \uuline{a}\,\dmat\,\uuline{b}\,\vec{c} + \uuline{b}\,\dmat\,\uuline{a}\,\vec{c} + \uuline{c}\,\dmat\,\uuline{a}\,\vec{b} + \uuline{b}\,\uuline{c}\,\dmat\,\vec{a} + \uuline{a}\,\uuline{c}\,\dmat\,\vec{b} + \uuline{a}\,\uuline{b}\,\dmat\,\vec{c}\right)_i.
\end{equation}
\end{proof}

\subsection{Proof of Thm. \ref{Thm3} [Kinetic energy preservation]}\label{Thm3Proof}
\begin{proof}
The kinetic energy is calculated as
\begin{equation}
\kappa:=\frac{1}{2}\rho\,(u^2+v^2+w^2) = \frac{(\rho\,u)^2+(\rho\,v)^2+(\rho\,w)^2}{2\,\rho}.
\end{equation}
The kinetic energy balance for the compressible Euler equations is 
\begin{equation}
\label{eq:kep_balance}
\resizebox{0.9\textwidth}{!}{$
\frac{\partial \kappa}{\partial t} + \left(\frac{1}{2}\rho\,u\,(u^2+v^2+w^2)\right)_x +u\,p_x  + \left(\frac{1}{2}\rho\,v\,(u^2+v^2+w^2)\right)_y +v\,p_y+ \left(\frac{1}{2}\rho\,w\,(u^2+v^2+w^2)\right)_z +w\,p_z=0.
$}
\end{equation}
We can also express the time derivative of the kinetic energy in terms of the time derivatives of the mass and momentum 
\begin{equation}
\label{eq:dkdt}
\begin{split}
\frac{\partial \kappa}{\partial t}:&= \frac{\partial k}{\partial\rho}\,(\rho)_t +\frac{\partial k}{\partial(\rho\,u)}\,(\rho\,u)_t +\frac{\partial k}{\partial(\rho\,v)}\,(\rho\,v)_t +\frac{\partial k}{\partial(\rho\,w)}\,(\rho\,w)_t,\\
&=-\left(\frac{u^2+v^2+w^2}{2}\right)(\rho)_t +(u)\,(\rho\,u)_t+(v)\,(\rho\,v)_t+(w)\,(\rho\,w)_t.
\end{split}
\end{equation}

We focus on the volume parts  of the continuity equation and consider a single node $ijk$ to get 
\begin{equation}
\label{eq:drdt}
\resizebox{0.9\textwidth}{!}{$
\begin{split}
-\left(\frac{u^2+v^2+w^2}{2}\right)(\rho)_t &\approx  -\left(\frac{(u_{ijk})^2+(v_{ijk})^2+(w_{ijk})^2}{2}\right) 2\,\sum\limits_{m=0}^N D_{im}\,F^{\#,1}(U_{ijk},U_{mjk}) + D_{jm}\,G^{\#,1}(U_{ijk},U_{imk}) + D_{km}\,H^{\#,1}(U_{ijk},U_{ijm}) ,
\end{split}$}
\end{equation}
and
\begin{equation}
\resizebox{0.9\textwidth}{!}{$
\begin{split}
(u)\,(\rho\,u)_t &\approx u_{ijk}\,2\,\sum\limits_{m=0}^N D_{im}\, F^{\#,1}(U_{ijk},U_{mjk})\,\average{u}_{(i,m)jk} + D_{jm}\, G^{\#,1}(U_{ijk},U_{imk})\,\average{u}_{i(j,m)k}+D_{km}\, H^{\#,1}(U_{ijk},U_{ijm})\,\average{u}_{ij(k,m)} \\
&+  u_{ijk}\, 2\,\sum\limits_{m=0}^N D_{im}\widetilde{p}(U_{ijk},U_{mjk}),\\
(v)\,(\rho\,v)_t &\approx v_{ijk}\,2\,\sum\limits_{m=0}^N D_{im}\, F^{\#,1}(U_{ijk},U_{mjk})\,\average{v}_{(i,m)jk} + D_{jm}\, G^{\#,1}(U_{ijk},U_{imk})\,\average{v}_{i(j,m)k}+D_{km}\, H^{\#,1}(U_{ijk},U_{ijm})\,\average{v}_{ij(k,m)}\\
&+  v_{ijk}\, 2\,\sum\limits_{m=0}^N D_{jm}\widetilde{p}(U_{ijk},U_{imk}),\\
(w)\,(\rho\,w)_t &\approx w_{ijk}\,2\,\sum\limits_{m=0}^N D_{im}\, F^{\#,1}(U_{ijk},U_{mjk})\,\average{w}_{(i,m)jk} + D_{jm}\, G^{\#,1}(U_{ijk},U_{imk})\,\average{w}_{i(j,m)k}+D_{km}\, H^{\#,1}(U_{ijk},U_{ijm})\,\average{w}_{ij(k,m)} \\
&+  w_{ijk}\, 2\,\sum\limits_{m=0}^N D_{km}\widetilde{p}(U_{ijk},U_{ijm}),
\end{split}$}
\end{equation}
where we used a notation of the average of two states, e.g. $\average{u}_{(i,m)jk} :=\frac{1}{2}(u_{ijk}+u_{mjk})$ and inserted the kinetic energy preserving structure of the numerical flux functions, see \eqref{eq:kep_property}.

The next step is to sum all four expressions according to \eqref{eq:dkdt}. We first note that as long as the numerical trace approximation of the pressure $\widetilde{p}$ is consistent and symmetric, it follows from the work of Fisher et al. and Carpenter et al. \cite{fisher2013,carpenter_esdg} that the resulting discretisation of $u\,p_x + v\,p_y+w\,p_z$ of the pressure work is consistent and high-order accurate.

If we look at the advective terms, we note that the values $u_{ijk}$, $v_{ijk}$, and $w_{ijk}$ do not depend on the summation index $m$ and thus can be pulled into the sum. By doing this, we inspect the product of the velocities and the averages. Looking at the terms involving $F^{\#,1}(U_{ijk},U_{mjk})$, 
\begin{equation}
\resizebox{0.9\textwidth}{!}{$
\begin{split}
u_{ijk}\average{u}_{(i,m)jk}\,F^{\#,1}(U_{ijk},U_{mjk}) &= u_{ijk}\,\frac{1}{2}(u_{ijk}+u_{mjk})\,F^{\#,1}(U_{ijk},U_{mjk}) = \left[\frac{(u_{ijk})^2}{2}+ \frac{u_{ijk}\,u_{mjk}}{2}\right]\,F^{\#,1}(U_{ijk},U_{mjk}),\\
v_{ijk}\average{v}_{(i,m)jk}\,F^{\#,1}(U_{ijk},U_{mjk})   &=v_{ijk}\,\frac{1}{2}(v_{ijk}+v_{mjk})\,F^{\#,1}(U_{ijk},U_{mjk})  =  \left[\frac{(v_{ijk})^2}{2}+ \frac{v_{ijk}\,v_{mjk}}{2}\right]\,F^{\#,1}(U_{ijk},U_{mjk}),\\
w_{ijk}\average{w}_{(i,m)jk}\,F^{\#,1}(U_{ijk},U_{mjk})   &=w_{ijk}\,\frac{1}{2}(w_{ijk}+w_{mjk})\,F^{\#,1}(U_{ijk},U_{mjk})  =  \left[\frac{(w_{ijk})^2}{2}+ \frac{w_{ijk}\,w_{mjk}}{2}\right]\,F^{\#,1}(U_{ijk},U_{mjk}).
\end{split}$}
\end{equation}
Summing these three terms gives 
\begin{equation}
\resizebox{0.9\textwidth}{!}{$
\begin{split}
\left(u_{ijk}\average{u}_{(i,m)jk} + v_{ijk}\average{v}_{(i,m)jk} + w_{ijk}\average{w}_{(i,m)jk}\right)\,F^{\#,1}(U_{ijk},U_{mjk})&=\frac{(u_{ijk})^2+(v_{ijk})^2+(w_{ijk})^2}{2}\,F^{\#,1}(U_{ijk},U_{mjk})\\
&+\frac{u_{ijk}\,u_{mjk}+v_{ijk}\,v_{mjk}+w_{ijk}\,w_{mjk}}{2}\,F^{\#,1}(U_{ijk},U_{mjk}).
\end{split}$}
\end{equation}
Note, that the first term is identical to the first term on the right hand side of \eqref{eq:drdt}, except for the sign. Thus, when adding, this term cancels and we are left with the remainder term
\begin{equation}
f^{\#,\kappa}(U_{ijk},U_{mjk}) :=\frac{1}{2}\,F^{\#,1}(U_{ijk},U_{mjk})\left(u_{ijk}\,u_{mjk}+v_{ijk}\,v_{mjk}+w_{ijk}\,w_{mjk}\right),
\end{equation}
which is symmetric in the arguments and is consistent with the advective flux of the kinetic energy balance $f=\frac{1}{2}\rho\,u\,(u^2+v^2+w^2)$.\\
\\
Analogously, if we focus on the terms involving $G^{\#,1}(U_{ijk},U_{imk})$, we get 
\begin{equation}
g^{\#,\kappa}(U_{ijk},U_{imk}) :=\frac{1}{2}\,G^{\#,1}(U_{ijk},U_{imk})\left(u_{ijk}\,u_{imk}+v_{ijk}\,v_{imk}+w_{ijk}\,w_{imk}\right),
\end{equation}
which is again symmetric in the arguments and consistent to $g=\frac{1}{2}\rho\,v\,(u^2+v^2+w^2)$. For the terms involving $H^{\#,1}(U_{ijk},U_{ijm})$ we get
\begin{equation}
h^{\#,\kappa}(U_{ijk},U_{ijm}) :=\frac{1}{2}\,H^{\#,1}(U_{ijk},U_{ijm})\left(u_{ijk}\,u_{ijm}+v_{ijk}\,v_{ijm}+w_{ijk}\,w_{ijm}\right),
\end{equation}
which is also symmetric and consistent to $h=\frac{1}{2}\rho\,w\,(u^2+v^2+w^2)$.

Thus, summarising, we get the following volume term in our discrete kinetic energy balance
\begin{equation}
\begin{split}
\left(\frac{\partial \kappa}{\partial t}\right)_{ijk}\approx &-2\,\sum\limits_{m=0}^N D_{im}\,f^{\#,\kappa}(U_{ijk},U_{mjk})+D_{jm}\,g^{\#,\kappa}(U_{ijk},U_{imk})  
+D_{km}\,h^{\#,\kappa}(U_{ijk},U_{ijm})\\
&-2\,\sum\limits_{m=0}^N u_{ijk}\,D_{im}\,\widetilde{p}(U_{ijk},U_{mjk})+v_{ijk}\,D_{jm}\,\widetilde{p}(U_{ijk},U_{imk})+w_{ijk}\,D_{km}\,\widetilde{p}(U_{ijk},U_{ijm}),
\end{split}
\end{equation}
which is a consistent and high-order accurate approximation of the continuous kinetic energy balance \eqref{eq:kep_balance}, with the advective terms in conservative form. 
\end{proof}

\section{Curvilinear flux differencing form}\label{sec:curvilinear}

Now the computational domain is divided into non-overlapping \emph{curved} hexahedral elements $C$. We create a polynomial transformation $\vec{X}(\vec\xi)$ to map computational coordinates in the reference cube $\vec{\xi} = (\xi^1,\xi^2,\xi^3)= (\xi,\eta,\zeta)$ to physical coordinates $\vec{x} = (x,y,z)\in C$, for details see \cite{koprivabook}.

The curvilinear coordinate system for $\vec\xi$ on the reference cube has three covariant basis vectors, $\vec{a}_i$, computed directly from the transformation
\begin{equation}\label{covariantVecs}
\vec{a}_i = \pderivative{\vec{X}}{\xi^i},\quad i=1,2,3.
\end{equation}
From the covariant basis vectors we derive the contravariant basis vectors $\vec{a}^i$, scaled by the Jacobian of the transformation $J $
\begin{equation}\label{contraVecs}
J \vec{a}^i = J \nabla{\xi^i} = \vec{a}_j\times\vec{a}_k,\quad (i,j,k)\textrm{ cyclic}.
\end{equation}
Alternatively, there is an explicitly divergence-free form of the contravariant basis vectors derived in \cite{Kopriva:2006er}
\begin{equation}\label{divFreeContra}
J a^i_n = -\hat{x}_i\cdot\nabla_\xi\times(x_l\nabla_\xi x_m),\quad i = 1,2,3;\; n = 1,2,3;\; (n,m,l)\;\textrm{cyclic}. 
\end{equation}
The divergence-free form \eqref{divFreeContra} of the contravariant basis vectors is particularly important to prevent spurious oscillations in the solution on curved sided hexahedral elements \cite{Kopriva:2006er}. However, for two dimensional problems and straight-sided hexahedral meshes (e.g. Cartesian meshes) the cross product formulation \eqref{contraVecs} is sufficient to prevent the generation of spurious waves by a mesh \cite{Kopriva:2006er}. 

The conservation law in the physical domain transforms to a conservation law equation in the reference domain of the form
\begin{equation}\label{newConsLaw}
J{U}_t + \widetilde{\mathcal{L}}^{div}_{\xi}(U) + \widetilde{\mathcal{L}}^{div}_{\eta}(U) + \widetilde{\mathcal{L}}^{div}_{\zeta}(U) = 0,
\end{equation}
where the Jacobian values of the transformation are computed from the covariant basis vectors \eqref{covariantVecs} and the contravariant operators incorporate the metric terms \eqref{contraVecs} 
\begin{equation}\label{newVars}
\begin{aligned}
J &= \vec{a}_1\cdot(\vec{a}_2\times \vec{a}_3),\\
\widetilde{\mathcal{L}}^{div}_{\xi}(U) &= \left(J {a}^1_1\,F(U)+J {a}^1_2\,G(U)+J {a}^1_3\,H(U)\right)_\xi,\\[0.1cm]
\widetilde{\mathcal{L}}^{div}_{\eta}(U) &= \left(J {a}^2_1\,F(U)+J {a}^2_2\,G(U)+J {a}^2_3\,H(U)\right)_\eta,\\[0.1cm]
\widetilde{\mathcal{L}}^{div}_{\zeta}(U) &= \left(J {a}^3_1\,F(U)+J {a}^3_2\,G(U)+J {a}^3_3\,H(U)\right)_\zeta.\\[0.1cm]
\end{aligned}
\end{equation}
Next, consider the component $l$ of the mapped system \eqref{newConsLaw} and an GL node $(i,j,k)$, the DGSEM approximation in strong form is
\begin{equation}\label{eq:curve_operator}
\begin{aligned}
\left(\widetilde{\mathcal{L}}^{div}_{\xi}(U)\right)_{ijk}^l &\approx \left[\widetilde{F}^{*,l}(1,\eta_j,\zeta_k;\vec{n}) - \tilde{F}^l_{Njk}\right] -  \left[\widetilde{F}^{*,l}(-1,\eta_j,\zeta_k;\vec{n}) - \widetilde{F}^l_{0jk}\right] + \sum_{m=0}^N D_{im}\widetilde{F}^l_{mjk},\\[0.1cm]
\widetilde{F}^l_{ijk} &\approx Ja^1_1\,F^l_{ijk} + Ja^1_2\,G^l_{ijk} + Ja_3^1\,H^l_{ijk},
\end{aligned}
\end{equation}
where we use collocation for the non-linear flux functions, e.g. $F^l_{ijk} := F^l(U_{ijk})$ and denote the outward pointing normal vector by $\vec{n}$. Note that the terms $\left(\widetilde{\mathcal{L}}^{div}_{\eta}(U)\right)^l_{ijk}$ and $\left(\widetilde{\mathcal{L}}^{div}_{\zeta}(U)\right)^l_{ijk}$ have an analogous structure to \eqref{eq:curve_operator}.

\subsection{Split form stabilisation for curvilinear DGSEM}

Just as in Sec. \ref{sec:DGSEM_fluxform}, we consider alternative formulations for the volume part of the discretisation \eqref{eq:curve_operator}. Again, this means starting from a standard strong form DGSEM implementation and modifying the volume integrals. We note that for curvilinear meshes the metric terms are now polynomials. Thus, the non-linearity in the contravariant terms increases. The increased non-linearity due to the curvilinear grid is an additional cause for aliasing driven non-linear instabilities.  For a detailed algorithmic description of the curvilinear flux difference formulation see \cite{wintermeyer2015}.

In his PhD thesis, Fisher \cite{fisher2012} extended the de-aliasing flux differencing technique \eqref{eq:highorder_flux} to curvilinear grids with corresponding element-wise mappings. For this extension, the contravariant operators in each direction are needed. We collect the main result below and proofs can be found in \cite{fisher2012}. In general curvilinear coordinates the high-order flux difference form has the structure
\begin{equation}\label{eq:highOrderFluxCurvilinear}
\begin{aligned}
\frac{\overline{\widetilde{F}^l}_{(i+1)jk}-\overline{\widetilde{F}^l}_{(i)jk}}{\omega_i} \approx 2\sum_{m=0}^N D_{im}&\left[F^{\#,l}(U_{ijk},U_{mjk})\average{Ja^1_1}_{(i,m)jk}\right.\\[-0.2cm]
&\;+G^{\#,l}(U_{ijk},U_{mjk})\average{Ja^1_2}_{(i,m)jk}\\
&\left.\;+H^{\#,l}(U_{ijk},U_{mjk})\average{Ja^1_3}_{(i,m)jk}\right],
\end{aligned}
\end{equation}
where 
\begin{equation}\label{eq:metricTermAvg}
\average{Ja^1_r}_{(i,m)jk} = \half\left[\left(Ja_r^1\right)_{ijk} + \left(Ja_r^1\right)_{mjk}\right],\quad r = 1,2,3
\end{equation}
are the average of the metric term components. For the approximation of the $\overline{\widetilde{G}^l}$ and $\overline{\widetilde{H}^l}$ contravariant flux differences the superscript on the metric term averages in \eqref{eq:metricTermAvg} change to $2$ and $3$ respectively. Each of the volume flux functions in \eqref{eq:highOrderFluxCurvilinear} are consistent and symmetric. Also, the resulting curvilinear approximation is high-order accurate and conservative in the Lax-Wendroff sense. We use the form of the volume discretisation \eqref{eq:highOrderFluxCurvilinear} rather than the standard discretisation \eqref{eq:curve_operator} 
\begin{equation}\label{eq:curve_operatorFinished}
\resizebox{0.95\hsize}{!}{$
\begin{aligned}
\left(\widetilde{\mathcal{L}}^{div}_{\xi}(U)\right)_{ijk}^l \approx \left[\widetilde{F}^{*,l}(1,\eta_j,\zeta_k;\vec{n}) - \tilde{F}^l_{Njk}\right] -  \left[\widetilde{F}^{*,l}(-1,\eta_j,\zeta_k;\vec{n}) - \widetilde{F}^l_{0jk}\right] +2\sum_{m=0}^N D_{im}&\left[F^{\#,l}(U_{ijk},U_{mjk})\average{Ja^1_1}_{(i,m)jk}\right.\\[-0.2cm]
&\;+G^{\#,l}(U_{ijk},U_{mjk})\average{Ja^1_2}_{(i,m)jk}\\
&\left.\;+H^{\#,l}(U_{ijk},U_{mjk})\average{Ja^1_3}_{(i,m)jk}\right], 
\\[0.1cm]
\left(\widetilde{\mathcal{L}}^{div}_{\eta}(U)\right)_{ijk}^l \approx \left[\widetilde{G}^{*,l}(\xi_i,1,\zeta_k;\vec{n}) - \tilde{G}^l_{iNk}\right] -  \left[\widetilde{G}^{*,l}(\xi_i,-1,\zeta_k;\vec{n}) - \widetilde{G}^l_{i0k}\right] +2\sum_{m=0}^N D_{jm}&\left[F^{\#,l}(U_{ijk},U_{imk})\average{Ja^2_1}_{i(j,m)k}\right.\\[-0.2cm]
&\;+G^{\#,l}(U_{ijk},U_{imk})\average{Ja^2_2}_{i(j,m)k}\\
&\left.\;+H^{\#,l}(U_{ijk},U_{imk})\average{Ja^2_3}_{i(j,m)k}\right],\\[0.1cm]
\left(\widetilde{\mathcal{L}}^{div}_{\zeta}(U)\right)_{ijk}^l \approx \left[\widetilde{H}^{*,l}(\xi_i,\eta_j,1;\vec{n}) - \tilde{H}^l_{ijN}\right] -  \left[\widetilde{H}^{*,l}(\xi_i,\eta_j,-1;\vec{n}) - \widetilde{H}^l_{ij0}\right] +2\sum_{m=0}^N D_{km}&\left[F^{\#,l}(U_{ijk},U_{ijm})\average{Ja^3_1}_{ij(k,m)}\right.\\[-0.2cm]
&\;+G^{\#,l}(U_{ijk},U_{ijm})\average{Ja^3_2}_{ij(k,m)}\\
&\left.\;+H^{\#,l}(U_{ijk},U_{ijm})\average{Ja^3_3}_{ij(k,m)}\right],
\end{aligned}$}
\end{equation}
where the choice of the numerical surface fluxes in the normal direction, $\vec{n}$, for each of the $\#$ variants of DGSEM are identical to those discussed in Sec. \ref{sec:numflux}.

\section*{References}
\bibliographystyle{plain}
\bibliography{References}

\begin{thebibliography}{10}

\bibitem{blaisdell1996effect}
G.~A. Blaisdell, E.~T. Spyropoulos, and J.~H. Qin.
\newblock The effect of the formulation of nonlinear terms on aliasing errors
  in spectral methods.
\newblock {\em Applied Numerical Mathematics}, 21(3):207--219, 1996.

\bibitem{carpenter_esdg}
M.~Carpenter, T.~Fisher, E.~Nielsen, and S.~Frankel.
\newblock Entropy stable spectral collocation schemes for the
  {N}avier--{S}tokes equations: Discontinuous interfaces.
\newblock {\em SIAM Journal on Scientific Computing}, 36(5):B835--B867, 2014.

\bibitem{Kennedy1994}
Mark~H. Carpenter and Christopher~A. Kennedy.
\newblock Fourth-order 2{N}-storage {R}unge- {K}utta schemes.
\newblock Technical Report NASA TM 109111, NASA Langley Research Center, 1994.

\bibitem{chandrashekar2013}
Praveen Chandrashekar.
\newblock Kinetic energy preserving and entropy stable finite volume schemes
  for compressible {E}uler and {N}avier-{S}tokes equations.
\newblock {\em Communications in Computational Physics}, 14(5):1252--1286,
  2013.

\bibitem{CS5}
B.~Cockburn, S.~Hou, and C.~W. Shu.
\newblock The {Runge}-{Kutta} local projection discontinuous {Galerkin} finite
  element method for conservation laws {IV}: {T}he multidimensional case.
\newblock {\em Math. Comput.}, 54:545--581, 1990.

\bibitem{CS4}
B.~Cockburn, S.~Y. Lin, and C.W. Shu.
\newblock {TVB} {Runge}-{Kutta} local projection discontinuous {Galerkin}
  finite element method for conservation laws {III}: {O}ne dimensional systems.
\newblock {\em J. Comput. Phys.}, 84:90--113, 1989.

\bibitem{CS2}
B.~Cockburn and C.~W. Shu.
\newblock {TVB} {Runge}-{Kutta} local projection discontinuous {Galerkin}
  finite element method for conservation laws {II}: {G}eneral framework.
\newblock {\em Math. Comput.}, 52:411--435, 1989.

\bibitem{CS6}
B.~Cockburn and C.~W. Shu.
\newblock The {Runge}-{Kutta} discontinuous {Galerkin} method for conservation
  laws {V}: {M}ultidimensional systems.
\newblock {\em J. Comput. Phys.}, 141:199--224, 1998.

\bibitem{CS1}
B.~Cockburn and C.W. Shu.
\newblock The {Runge}-{Kutta} local projection $p^1$-discontinuous {G}alerkin
  method for scalar conservation laws.
\newblock {\em M$^2$AN}, 25:337--361, 1991.

\bibitem{ducros2000}
F.~{Ducros}, F.~{Laporte}, T.~{Soul{\`e}res}, V.~{Guinot}, P.~{Moinat}, and
  B.~{Caruelle}.
\newblock High-order fluxes for conservative skew-symmetric-like schemes in
  structured meshes: Application to compressible flows.
\newblock {\em Journal of Computational Physics}, 161:114--139, 2000.

\bibitem{fisher2012}
Travis~C. Fisher.
\newblock {\em High-order $L_2$ stable multi-domain finite difference method
  for compressible flows}.
\newblock PhD thesis, Purdue University, 2012.

\bibitem{fisher2013}
Travis~C. Fisher and Mark~H. Carpenter.
\newblock High-order entropy stable finite difference schemes for nonlinear
  conservation laws: {F}inite domains.
\newblock {\em Journal of Computational Physics}, 252:518--557, 2013.

\bibitem{gassner_skew_burgers}
Gregor~J. Gassner.
\newblock A skew-symmetric discontinuous {Galerkin} spectral element
  discretization and its relation to {SBP-SAT} finite difference methods.
\newblock {\em SIAM Journal on Scientific Computing}, 35(3):A1233--A1253, 2013.

\bibitem{gassner_kepdg}
Gregor~J. Gassner.
\newblock A kinetic energy preserving nodal discontinuous {G}alerkin spectral
  element method.
\newblock {\em International Journal for Numerical Methods in Fluids},
  76(1):28--50, 2014.

\bibitem{tcfd2012}
Gregor~J. Gassner and Andrea~D. Beck.
\newblock On the accuracy of high-order discretizations for underresolved
  turbulence simulations.
\newblock {\em Theoretical and Computational Fluid Dynamics}, 2012.
\newblock DOI: 10.1007/s00162-011-0253-7.

\bibitem{gassner2015}
Gregor~J. Gassner, Andrew~R. Winters, and David~A. Kopriva.
\newblock A well balanced and entropy conservative discontinuous {G}alerkin
  spectral element method for the shallow water equations.
\newblock {\em Applied Mathematics and Computation}, page
  http://dx.doi.org/10.1016/j.amc.2015.07.014, 2015.

\bibitem{HM14_764}
Andreas Hiltebrand and Siddhartha Mishra.
\newblock Entropy stable shock capturing space-time discontinuous galerkin
  schemes for systems of conservation laws.
\newblock {\em Numerische Mathematik}, 126(1):103--151, 2014.

\bibitem{Hindenlang2012}
F.~Hindenlang, G.~Gassner, C.~Altmann, A.~Beck, M.~Staudenmaier, and C.-D.
  Munz.
\newblock Explicit discontinuous {G}alerkin methods for unsteady problems.
\newblock {\em Computers and Fluids}, 61:86--93, 2012.

\bibitem{ismail2009}
Farzad Ismail and Philip~L. Roe.
\newblock Affordable, entropy-consistent {E}uler flux functions {II}: Entropy
  production at shocks.
\newblock {\em Journal of Computational Physics}, 228(15):5410--5436, 2009.

\bibitem{jameson2008}
Antony Jameson.
\newblock Formulation of kinetic energy preserving conservative schemes for gas
  dynamics and direct numerical simulation of one-dimensional viscous
  compressible flow in a shock tube using entropy and kinetic energy preserving
  schemes.
\newblock {\em Journal of Scientific Computing}, 34(3):188--208, 2008.

\bibitem{cell_entropy_dg}
Guangshan Jiang and Chi-Wang Shu.
\newblock On a cell entropy inequality for discontinuous galerkin methods.
\newblock {\em Mathematics of Computation}, 62(206):pp. 531--538, 1994.

\bibitem{kennedy2008}
Christopher~A. Kennedy and Andrea Gruber.
\newblock Reduced aliasing formulations of the convective terms within the
  {N}avier--{S}tokes equations for a compressible fluid.
\newblock {\em Journal of Computational Physics}, 227:1676--1700, 2008.

\bibitem{Kirby2003}
R.M. Kirby and G.E. Karniadakis.
\newblock De-aliasing on non-uniform grids: algorithms and applications.
\newblock {\em Journal of Computational Physics}, 191:249--264, 2003.

\bibitem{Kopriva:2006er}
David~A. Kopriva.
\newblock Metric identities and the discontinuous spectral element method on
  curvilinear meshes.
\newblock {\em Journal of Scientific Computing}, 26(3):301--327, March 2006.

\bibitem{koprivabook}
David~A. Kopriva.
\newblock {\em Implementing Spectral Methods for Partial Differential
  Equations: Algorithms for Scientists and Engineers}.
\newblock Springer Publishing Company, Incorporated, 1st edition, 2009.

\bibitem{KoprivaGassner_GaussLob}
David~A. Kopriva and Gregor~J. Gassner.
\newblock On the quadrature and weak form choices in collocation type
  discontinuous {G}alerkin spectral element methods.
\newblock {\em Journal of Scientific Computing}, 44(2):136--155, 2010-08-01.

\bibitem{kopriva2014}
David~A. Kopriva and Gregor~J. Gassner.
\newblock An energy stable discontinuous {G}alerkin spectral element
  discretization for variable coefficient advection problems.
\newblock {\em SIAM Journal on Scientific Computing}, 36(4):2076--2099, 2014.

\bibitem{FDaliasing}
A.G. Kravchenko and P.~Moin.
\newblock On the effect of numerical errors in large eddy simulations of
  turbulent flows.
\newblock {\em Journal of Computational Physics}, 131(2):310 -- 322, 1997.

\bibitem{larsson2007}
J~Larsson, S.~K. Lele, and Moin P.
\newblock Effect of numerical dissipation on the predicted spectra for
  compressible turbulence.
\newblock {\em Annual Research Briefs}, pages 47--57, 2007.

\bibitem{mengaldo2015}
G~Mengaldo, D~De~Grazia, D~Moxey, P.~E. Vincent, and S.~J. Sherwin.
\newblock Dealiasing techniques for high-order spectral element methods on
  regular and irregular grids.
\newblock {\em Journal of Computational Physics}, 299:56--81, 2015.

\bibitem{Morinishi2010276}
Yohei Morinishi.
\newblock Skew-symmetric form of convective terms and fully conservative finite
  difference schemes for variable density low-{M}ach number flows.
\newblock {\em Journal of Computational Physics}, 229(2):276 -- 300, 2010.

\bibitem{rodrigo_iLES}
Rodrigo~Costa Moura, Spencer~J. Sherwin, and Joaquim Peiro.
\newblock On {DG}-based i{LES} approaches at very high {Reynolds} numbers.
\newblock Report, Research Gate, 2015.

\bibitem{parsaniCITEKEY}
Matteo Parsani, Mark~H Carpenter, and Eric~J Nielsen.
\newblock Entropy stable discontinuous interfaces coupling for the
  three-dimensional compressible {N}avier-{S}tokes equations.
\newblock {\em Journal of Computational Physics}, 290:132--138, 2015.

\bibitem{parsani2015entropy}
Matteo Parsani, Mark~H Carpenter, and Eric~J Nielsen.
\newblock Entropy stable wall boundary conditions for the three-dimensional
  compressible {N}avier-{S}tokes equations.
\newblock {\em Journal of Computational Physics}, 292:88--113, 2015.

\bibitem{pirozzoli2010}
Sergio Pirozzoli.
\newblock Generalized conservative approximations of split convective
  derivative operators.
\newblock {\em Journal of Computational Physics}, 229(19):7180--7190, 2010.

\bibitem{pirozzoli2011}
Sergio Pirozzoli.
\newblock Numerical methods for high-speed flows.
\newblock {\em Annual Review of Fluid Mechanics}, 43:163--194, 2011.

\bibitem{Strand199447}
Bo~Strand.
\newblock Summation by parts for finite difference approximations for d/dx.
\newblock {\em Journal of Computational Physics}, 110(1):47 -- 67, 1994.

\bibitem{svard2014entropy}
Magnus Sv{\"a}rd and Hatice {\"O}zcan.
\newblock Entropy-stable schemes for the {E}uler equations with far-field and
  wall boundary conditions.
\newblock {\em Journal of Scientific Computing}, 58(1):61--89, 2014.

\bibitem{wintermeyer2015}
Niklas Wintermeyer, Andrew~R. Winters, Gregor~J. Gassner, and David~A. Kopriva.
\newblock An entropy stable nodal discontinuous {G}alerkin method for the two
  dimensional shallow water equations on unstructured curvilinear meshes with
  discontinuous bathymetry.
\newblock {\em Journal of Computational Physics}, (submitted), 2015.

\bibitem{Zang199127}
Thomas~A. Zang.
\newblock On the rotation and skew-symmetric forms for incompressible flow
  simulations.
\newblock {\em Applied Numerical Mathematics}, 7:27 -- 40, 1991.

\end{thebibliography}

\end{document}